\documentclass[11pt]{amsart}
\usepackage{graphicx}
\usepackage{amssymb}
\usepackage{amsmath}
\usepackage{amsthm}
\usepackage{mathtools}
\usepackage{tikz}
\usepackage{psfrag}
\usepackage{xcolor}

\newtheorem{thm}{Theorem}[section]
\newtheorem{lem}[thm]{Lemma}
\newtheorem{pro}[thm]{Proposition}
\newtheorem{coro}[thm]{Corollary}

\theoremstyle{definition}
\newtheorem{defn}[thm]{Definition}
\newtheorem{remk}[thm]{Remark}

\newcommand{\n}{\mathbf n}

\newcommand{\B}{\mathbf B}
\newcommand{\R}{\mathbb R}
\newcommand{\dist}{\operatorname{dist}}

\newcommand{\X}{\mathcal X}
\newcommand{\Y}{\mathcal Y}
\newcommand{\G}{\mathcal G}
\newcommand{\tr}{\operatorname{tr}}
\newcommand{\rank}{\operatorname{rank}}
\newcommand{\dv}{\operatorname{div}}

\newcommand{\wcon}{\rightharpoonup}

\makeatletter
\numberwithin{equation}{section}
\makeatother

\begin{document}



       \author{Seonghak Kim}
       \address{Institute for Mathematical Sciences\\ Renmin University of China \\  Beijing 100872, PRC}
       \email{kimseo14@gmail.com}


       \author{Baisheng Yan}

       \address{Department of Mathematics\\ Michigan State University\\ East Lansing, MI 48824, USA}


       \email{yan@math.msu.edu}


       \title[forward-backward parabolic equations]{On Lipschitz solutions for some  forward-backward parabolic  equations}


\begin{abstract}
We investigate the existence and properties of Lipschitz solutions for some  forward-backward parabolic equations in all dimensions. Our main approach to existence  is motivated by reformulating such equations into   partial differential inclusions  and relies on a  Baire's category  method. In this way, the existence of infinitely many Lipschitz solutions to certain  initial-boundary value problem  of  those equations  is  guaranteed under  a pivotal  density condition. Finally, we study two important cases of  forward-backward anisotropic diffusion   in which the density condition can be realized and therefore the existence results follow together with micro-oscillatory behavior of solutions. The first case is a generalization of the Perona-Malik model in image processing and the other  that of H\"ollig's  model related to the Clausius-Duhem inequality in the second law of thermodynamics.
\end{abstract}
\maketitle

\section{Introduction}
The evolution process of many quantities in applications can be modeled by a diffusion partial differential equation of the form
\begin{equation}\label{de0}
u_t=\dv (A(Du)) \quad \mbox{in $\Omega\times (0,T)$,}
\end{equation}
where $\Omega\subset \R^n$ is a bounded domain, $T>0$ is any fixed  number,  and $u=u(x,t)$ is the density of the quantity at position  $x$ and time $t$, with  $Du=(u_{x_1},\cdots,u_{x_n})$ and $u_t$ denoting its spatial gradient and    rate of change, respectively.
The vector function $A\colon \R^n\to \R^n$ here represents the \emph{diffusion flux} of the evolution  process. The usual heat equation corresponds to the case of \emph{isotropic} diffusion given by the Fourier law: $A(p)=kp \; (p\in\R^n)$, where $k>0$ is the diffusion constant.

For standard diffusion equations, the flux $A(p)$ is assumed to be {\em monotone}; namely,
\begin{equation*}
(A(p)-A(q))\cdot (p-q)\ge 0 \quad (p,\, q\in\R^n).
\end{equation*}
In this case, equation (\ref{de0}) is parabolic and can be studied by the standard methods of parabolic equations and monotone operators \cite{Br, LSU, Ln}. In particular,  when  $A(p)$ is given by a smooth convex  function $W(p)$ through
$A(p)=D_p W(p) \;  (p\in\R^n),$
 (\ref{de0}) can be viewed and thus studied as a certain gradient flow generated by the energy functional
\[
I(u)=\int_\Omega W(Du(x))\,dx.
\]

However, for certain  applications of the evolution process  to some important physical problems, the underlying diffusion fluxes $A(p)$ may not be monotone, yielding   non-parabolic equations (\ref{de0}).
In this paper, we study the diffusion equation (\ref{de0}) for some   non-monotone diffusion fluxes $A(p).$ We focus on  the initial-boundary value problem
\begin{equation}\label{ib-P}
\begin{cases} u_t =\dv (A(Du) )& \mbox{in $\Omega_T$,} \\
A(Du)\cdot \n =0 & \mbox{on $\partial \Omega\times (0,T)$,}\\
u =u_0 & \mbox{on $\Omega\times \{t=0\},$}
\end{cases}
\end{equation}
where $\Omega_T=\Omega\times (0,T)$, $\n$ is the outer unit normal on $\partial\Omega$, $u_0=u_0(x)$ is a given initial datum, and
the flux  $A(p)$ is of the form
\begin{equation}\label{fun-A}
A(p)=f(|p|^2)p\quad (p\in\R^n),
\end{equation}
given by a function $f\colon [0,\infty)\to \R$ with \emph{profile} $\sigma(s)=s f(s^2)$ having  one of the graphs  in Figures \ref{fig1}, \ref{fig2} and \ref{fig2-2} below.
(Precise structural assumptions on $\sigma(s)=s f(s^2)$ will be given in Section \ref{sec:structure}.)

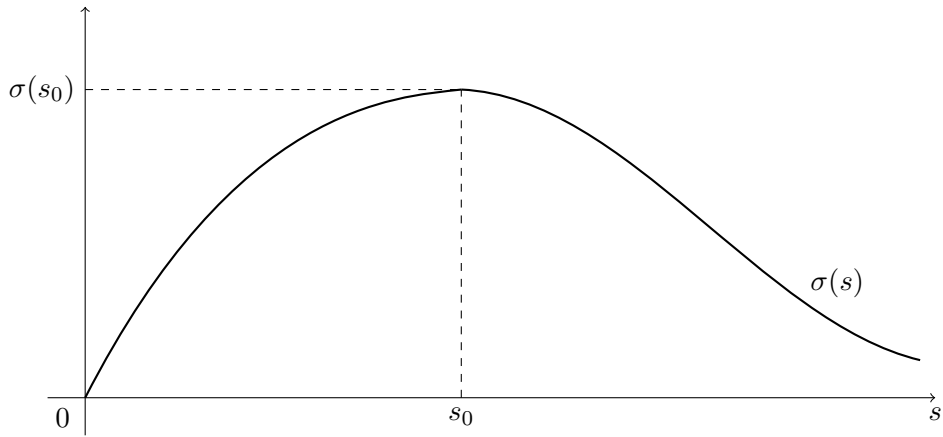
\begin{figure}[ht]
\begin{center}
\begin{tikzpicture}[scale = 1]
    \draw[->] (-.5,0) -- (11.3,0);
	\draw[->] (0,-.5) -- (0,5.2);
 \draw[dashed] (0,4.1)--(5,4.1);
    \draw[dashed] (5, 0)  --  (5, 4.1) ;
	\draw[thick]   (0, 0) .. controls  (2, 3.9) and (4.1,4) ..(5,4.1);
	\draw[thick]   (5, 4.1) .. controls  (7.2, 4) and (9,1) ..(11.1,0.5 );
	\draw (11.3,0) node[below] {$s$};
    \draw (-0.3,0) node[below] {{$0$}};
    \draw (10, 1.2) node[above] {$\sigma(s)$};
    \draw (0, 4.1) node[left] {$\sigma(s_0)$};
   \draw (5, 0) node[below] {$s_0$};
    \end{tikzpicture}
\end{center}
\caption{ \textbf{Case I:} Perona-Malik type of profiles   $\sigma(s)$.}
\label{fig1}
\end{figure}

\begin{figure}[ht]
\begin{center}
\begin{tikzpicture}[scale =1]
    \draw[->] (-.5,0) -- (11.3,0);
	\draw[->] (0,-.5) -- (0,5.2);
 \draw[dashed] (0,1.7)--(5.6,1.7);
 \draw[dashed] (0,3.3)--(9.1,3.3);
 \draw[dashed] (9,0)--(9,3.3);
    \draw[dashed] (5.6, 0)  --  (5.6, 1.7) ;
   \draw[dashed] (0.85, 0)  --  (0.85, 1.7) ;
   \draw[dashed] (2.8, 0)  --  (2.8, 3.3) ;
	\draw[thick]   (0, 0) .. controls (2,5) and  (4, 3)   ..(5,2);
	\draw[thick]   (5, 2) .. controls  (6, 1) and (9,3) ..(11.1, 5 );
	\draw (11.3,0) node[below] {$s$};
    \draw (-0.3,0) node[below] {{$0$}};
    \draw (10.5, 3.5) node[above] {$\sigma(s)$};
 \draw (0, 3.3) node[left] {$\sigma(s_1)$};
 \draw (0, 1.7) node[left] {$\sigma(s_2)$};
   \draw (5.6, 0) node[below] {$s_2$};
   \draw (2.8, 0) node[below] {$s_1$};
   \draw (0.85, 0) node[below] {$s_1^*$};
   \draw (9, 0) node[below] {$s_2^*$};
    \end{tikzpicture}
\end{center}
\caption{ \textbf{Case II:}  H\"ollig  type of profiles  $\sigma(s)$ with $\sigma(s_2)>0$.}
\label{fig2}
\end{figure}

\begin{figure}[ht]
\begin{center}
\begin{tikzpicture}[scale =1]
    \draw[->] (-.5,0) -- (11.3,0);
	\draw[->] (0,-.5) -- (0,5.2);
 \draw[dashed] (0,3.3)--(9,3.3);
 \draw[dashed] (9,0)--(9,3.3);
   \draw[dashed] (2.8, 0)  --  (2.8, 3.3) ;
	\draw[thick]   (0, 0) .. controls (3.56,6.8) and  (3.57, 1)   ..(5.6,0);
	\draw[thick]   (5.6, 0) .. controls  (7.5,0.3) and (8,1.5) ..(10, 5 );
	\draw (11.3,0) node[below] {$s$};
    \draw (-0.3,0) node[below] {{$0$}};
    \draw (10.5, 3.5) node[above] {$\sigma(s)$};
 \draw (0, 3.3) node[left] {$\sigma(s_1)$};
   \draw (5.6, 0) node[below] {$s_2$};
   \draw (2.8, 0) node[below] {$s_1$};
   \draw (9, 0) node[below] {$s_2^*$};
    \end{tikzpicture}
\end{center}
\caption{ \textbf{Case II:}  H\"ollig  type of profiles  $\sigma(s)$ with $\sigma(s_2)=0$.}
\label{fig2-2}
\end{figure}

The  two cases  in Figures \ref{fig1}, \ref{fig2} and \ref{fig2-2} correspond to the applications in image processing  proposed by Perona and Malik \cite{PM} and in phase transitions of thermodynamics studied by H\"ollig \cite{Ho}.
For these diffusions, we have
\begin{equation*}
\sigma'(\sqrt{s})=f(s)+2sf'(s)<0 \quad \mbox{for some range of $s>0$.}
\end{equation*}
In these cases, the diffusion is \emph{anisotropic} since the diffusion  matrix $(A^i_{p_j})$, where
\[
A_{p_j}^i(p)=f(|p|^2)\delta_{ij} + 2 f'(|p|^2) p_ip_j \quad (i,\, j=1,2,\cdots,n),
\]
has   the eigenvalues  $f(|p|^2)$ (of multiplicity $n-1$) and  $f(|p|^2)+2|p|^2f'(|p|^2);$ hence the diffusion coefficients could also be negative, and problem (\ref{ib-P}) becomes {\em forward-backward parabolic}. Moreover, setting
\[
W(p) =\int_0^{|p|} \sigma(r)\,dr,\quad
I(u)=\int_\Omega W(Du)\,dx,
\]
the initial-boundary value problem (\ref{ib-P}) becomes a  {$L^2$-gradient flow} of the energy functional $I(u);$ however,  $I(u)$  is \emph{non-convex}.
Consequently, neither the standard methods of parabolic equations and monotone operators nor the non-linear semigroup theory can  be applied to study  (\ref{ib-P}).

Forward-backward parabolic equations have found many  important applications  in  the mathematical modeling  of some  physical problems, but mathematically, due to the backward parabolicity, the initial-boundary value problem (\ref{ib-P}) of such equations is highly ill-posed, and in many cases even the notion and existence of reasonable solutions remain largely unsettled.

In \cite{PM}, the original Perona-Malik model (\ref{ib-P}) in image processing used the profile
\begin{equation}\label{PM-fun}
\sigma(s)=se^{-s^2/2s_0^2}\quad \mbox{or} \quad
\sigma(s)=\frac{s}{1+ {s^2}/{s_0^2}}
\end{equation}
  for denoising and edge enhancement of a computer vision; in this case, we call (\ref{de0}) the \emph{Perona-Malik equation}.
In this model, $u(x,t)$ represents an improved version of the initial gray level $u_0(x)$ of a noisy picture.
The anisotropic diffusion $\dv(A(Du))$ is forward parabolic in the \emph{subcritical} region where $|Du|<s_0$ and  backward parabolic in the  \emph{supercritical} region where $|Du|>s_0.$ The expectation of the model is that disturbances with small gradient in the subcritical region will be smoothed out by the forward parabolic diffusion, while sharp edges corresponding to  large gradient  in the supercritical region  will be enhanced by the backward parabolic equation. Such  expected phenomenology has been implemented and observed in some numerical experiments \cite{Es, PM}, showing the stability and effectiveness of the model.
 Mathematically, there have been extensive  works on the Perona-Malik type of equations with profiles $\sigma(s)$ as in Figure \ref{fig1}; however, most of  these works  have focused  on the analysis of numerical or approximate solutions by different methods.  For example, in dimension $n=1$,  the singular perturbation and its $\Gamma$-limit related  to the Perona-Malik type of equations were investigated in \cite{BF,BFG}. In \cite{Gu}, a mild regularization of the Perona-Malik equation with a viscous term was used to extract a unique approximate Young measure solution in any dimension. The works \cite{CZ,TTZ} studied  Young measure solutions for the  Perona-Malik equation in dimensions $n=1$ and $n=2$. Also the works   \cite{Es,EG} focused on the numerical schemes for the Perona-Malik model. Recently, classical solutions for the Perona-Malik equation were studied  in \cite{GG1,GG2}, where the existence of solutions was proved for certain initial data $u_0$  that can be \emph{transcritical}  in the sense that the sets $\{|Du_0(x)|>s_0\}$ and $\{|Du_0(x)|<s_0\}$ are both nonempty ($s_0$ is the number in Figure \ref{fig1}); however, their initial data cannot be arbitrarily prescribed.

For the H\"ollig model, the one-dimensional forward-backward parabolic problem (\ref{ib-P}) under the piecewise linearity assumption on the profile $\sigma(s)$ (see Figure \ref{fig2-1}) was studied in \cite{Ho}, motivated by the Clausius-Duhem inequality in the second law of thermodynamics  in continuum mechanics. For such a special profile, it was proved that there exist infinitely many $L^2$-weak solutions to (\ref{ib-P}) in dimension $n=1$. The piecewise linearity of $\sigma(s)$  in $n=1$  was much relaxed later to include a more general class of profiles $\sigma(s)$ (as in Figure \ref{fig2}) in the work of Zhang \cite{Zh1},  using  an entirely different approach  from H\"ollig's.

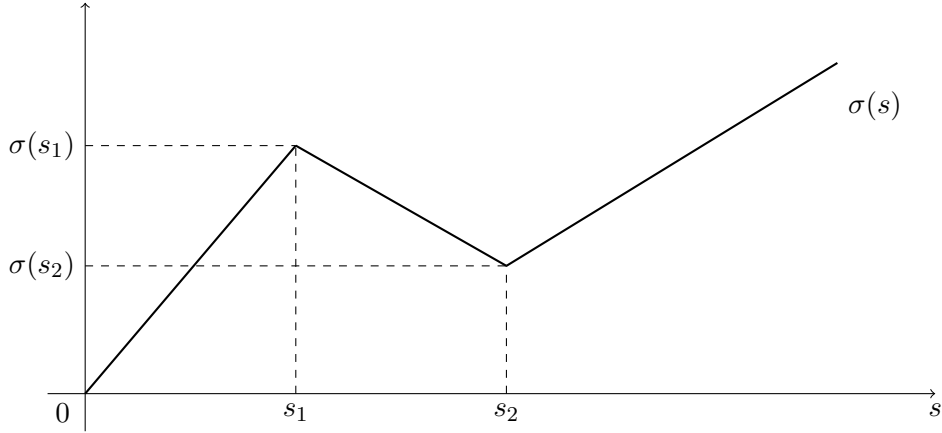
\begin{figure}[ht]
\begin{center}
\begin{tikzpicture}[scale =1]
    \draw[->] (-.5,0) -- (11.3,0);
	\draw[->] (0,-.5) -- (0,5.2);
 \draw[dashed] (0,1.7)--(5.6,1.7);
 \draw[dashed] (0,3.3)--(2.8,3.3);
    \draw[dashed] (5.6, 0)  --  (5.6, 1.7) ;
   \draw[dashed] (2.8, 0)  --  (2.8, 3.3) ;
	\draw[thick]   (0, 0) -- (2.8,3.3);
    \draw[thick]   (2.8,3.3) -- (5.6,1.7);
	\draw[thick]   (5.6,1.7) -- (10, 4.4 );
	\draw (11.3,0) node[below] {$s$};
    \draw (-0.3,0) node[below] {{$0$}};
    \draw (10.5, 3.5) node[above] {$\sigma(s)$};
 \draw (0, 3.3) node[left] {$\sigma(s_1)$};
 \draw (0, 1.7) node[left] {$\sigma(s_2)$};
   \draw (5.6, 0) node[below] {$s_2$};
   \draw (2.8, 0) node[below] {$s_1$};
    \end{tikzpicture}
\end{center}
\caption{H\"ollig's piecewise linear profile  $\sigma(s)$.}
\label{fig2-1}
\end{figure}

The question concerning the existence of \emph{exact} weak solutions to problem (\ref{ib-P}) of the Perona-Malik   and H\"ollig types  had remained open until {Zhang}  \cite{Zh,Zh1}  first established that   the \emph{one-dimensional} problem  (\ref{ib-P}) of each type has infinitely many   Lipschitz   solutions  for any suitably given smooth initial data $u_0.$
Zhang's pivotal idea was  to reformulate equation (\ref{de0}) in $n=1$ into a $2\times 2$ non-homogeneous partial differential inclusion  and then to  prove the existence by using a modified method of convex integration, following the ideas of \cite{Ki,MSv1}.  The study of general partial differential inclusions  has stemmed from the successful understanding of homogeneous differential inclusions of the form $Du(x)\in K$, first encountered in the study of crystal microstructures  \cite{BJ,CK,MSv1}; recently, the methods  of  differential inclusions  have been successfully applied to other important problems; see, e.g.,  \cite{CFG,DS,MP,MSv2,Sh,Ya1}.

In general,  a  function  $u\in W^{1,\infty}(\Omega_T)$ is called a \emph{Lipschitz solution} to problem (\ref{ib-P})   provided that equality
\begin{equation}\label{def:sol}
 \int_\Omega  (u(x,s)\zeta(x,s)-u_0(x) \zeta(x,0) )dx =
 \int_0^s \int_\Omega (u \zeta_t -A(Du)\cdot D\zeta ) dxdt
\end{equation}
holds for  each $\zeta\in C^\infty(\bar\Omega_T)$ and each $s\in [0,T].$ Let $\zeta\equiv 1$; then it is immediate  from the definition that any Lipschitz solution $u$ to (\ref{ib-P}) conserves the total mass over time:
\[
\int_\Omega  u(x,t)dx = \int_\Omega  u_0(x)dx \quad\forall t\in[0,T].
\]

In \cite{KY}, we extended Zhang's method  of partial differential inclusions to the Perona-Malik type of equations in all dimensions $n$ for balls $\Omega=B_R(0):=\{x\in\R^n\,|\, |x|<R\}$ and  non-constant radially symmetric  smooth  initial data $u_0.$ In this case, the $n$-dimensional equation for radial solutions can still be reformulated into a $2\times  2$ non-homogeneous partial differential inclusion, and the existence of infinitely many radial Lipschitz solutions to (\ref{ib-P}) is established.
However, for general domains and initial data, the $n$-dimensional problem (\ref{ib-P}) can only be recast as a $(1+n)\times (n+1)$ non-homogeneous partial differential inclusion that has some uncontrollable gradient components, making the construction of Lipschitz solutions for this inclusion impractical.  In our recent work \cite{KY1}, we overcame this difficulty by developing  a suitably modified density method, still motivated by the method of differential inclusions  but based  on a Baire's category argument. In this work, it  was proved  that  for all bounded convex  domains $\Omega\subset\R^n$ with $\partial \Omega\in C^{2+\alpha}$ and  initial data $u_0\in C^{2+\alpha}(\bar\Omega)$ with $Du_0\cdot \n=0$ on $\partial\Omega$, there exist infinitely many Lipschitz solutions to (\ref{ib-P}) for the \emph{exact} Perona-Malik diffusion flux $A(p)=\frac{p}{1+|p|^2}$ ($p\in\R^n$). However, the proof heavily relies on the explicit formula of the rank-one convex hull of a certain matrix set defined by this special flux; such an  explicit formula for  the general Perona-Malik type of flux $A(p)$ (with profile $\sigma(s)$ as in Figure \ref{fig1}) is unattainable.

The main purpose of this paper is to generalize the method of \cite{KY1} to   problem (\ref{ib-P})   in all dimensions for   general diffusion profiles of the Perona-Malik and H\"ollig types. To state our main theorems, we  make the following assumptions on the domain $\Omega$ and initial datum $u_0$:
\begin{equation}\label{assume-1}
\begin{cases} \mbox{$\Omega\subset\R^n$ is a bounded    domain  with $\partial \Omega$ of $C^{2+\alpha}$,}\\
\mbox{$u_0\in C^{2+\alpha}(\bar\Omega)$ is non-constant with $Du_0\cdot \n |_{\partial \Omega}=0,$} \end{cases}
\end{equation}
where $\alpha\in (0,1)$ is a given number.
Without specifying the precise conditions on the profiles  $\sigma(s)$ for \textbf{Cases I} and {\bf II}, we first state our main existence theorems as below.

\begin{thm}\label{thm:1} Assume condition $(\ref{assume-1})$ is fulfilled.  Let $A(p)$ be given by  $(\ref{fun-A})$ with the profile  $\sigma(s).$ Then problem  $(\ref{ib-P})$ has infinitely many Lipschitz  solutions in the following two cases.

 {\bf Case I:}  $\sigma(s)$ is of the \emph{Perona-Malik type} as    in {Figure\,\ref{fig1}}, and in addition $\Omega$ is convex.

 {\bf Case II:} $\sigma(s)$ is of  the \emph{H\"ollig type} as    in  {Figure\,\ref{fig2}}, and  in addition $u_0$ satisfies   $|Du_0(x_0)|\in(s_1^*,s_2^*)$ at
 some $x_0\in\Omega.$
\end{thm}

The precise assumptions on the profiles $\sigma(s)$ and detailed statements of the theorems for both cases will be given later (see   Theorems \ref{thm:PM-1} and \ref{thm:H-1}) along with some discussions  about a  certain implication  of our result on the Perona-Malik model in image processing and the breakdown of uniqueness for the H\"ollig type of equations.

In certain cases, among infinitely many Lipschitz solutions, problem (\ref{ib-P}) admits a unique classical solution;  we have the following result (see also \cite{KK}).

\begin{thm}\label{thm:2} Assume condition $(\ref{assume-1})$ is fulfilled and $\Omega$ is convex.  Let $A(p)$ be given by  $(\ref{fun-A})$ with the profile  $\sigma(s).$ Assume $\|Du_0\|_{L^\infty(\Omega)}<s_0$ if $\sigma(s)$ is of the \emph{Perona-Malik type}  or $\|Du_0\|_{L^\infty(\Omega)}<s_1$ if  $\sigma(s)$ is of the \emph{H\"ollig type}. Then problem $(\ref{ib-P})$ has  a unique solution $u\in C^{2+\alpha,1+\alpha/2}(\bar{\Omega}_T)$ satisfying
$\|Du\|_{L^\infty(\Omega_T)}=\|Du_0\|_{L^\infty(\Omega)}.$
\end{thm}

This theorem will  be also separated into \textbf{Cases I} and \textbf{II} in Section \ref{sec:structure} in order to be compared with our detailed main results,  Theorems \ref{thm:PM-1} and \ref{thm:H-1}.

Finally, we have the following \emph{coexistence} result. 

\begin{thm}\label{coex}  Assume that $\Omega=B_R(0)$ is a  ball in $\R^n$ and that $u_0$ is \emph{radial} and satisfies condition $(\ref{assume-1})$. Then there are infinitely many radial and non-radial Lipschitz solutions to \emph{(\ref{ib-P})}.
\end{thm}



The rest of the paper is organized as follows. Section \ref{sec:approach} begins with a motivational approach to problem (\ref{ib-P}) as a non-homogeneous partial differential inclusion and its limitation. Then the general existence theorem, Theorem \ref{thm:main},   is formulated under a key density hypothesis, and some general properties of Lipschitz solutions to (\ref{ib-P}) are also provided. In section \ref{sec:structure}, precise structural conditions on the profiles $\sigma(s)$ of the Perona-Malik and H\"ollig types are specified, and the detailed statements for the main result, Theorem \ref{thm:1}, are introduced as Theorems \ref{thm:PM-1} and \ref{thm:H-1} that are followed by  some  remarks and related results. Section \ref{sec:prelim} prepares some useful results that may be of independent interest. As the core analysis of the paper, the geometry of related matrix sets is scrutinized in Section \ref{sec:geometry}, leading to the relaxation result on a homogeneous differential inclusion, Theorem \ref{main-lemma}. Section \ref{sec:add-set} is devoted to the construction of suitable boundary functions and admissible sets for \textbf{Cases I} and \textbf{II} in the contexts of Theorems \ref{thm:PM-1} and \ref{thm:H-1}, respectively. Then the pivotal density hypothesis in each case is realized in Section \ref{sec:den-proof}, completing the proofs of Theorems \ref{thm:PM-1} and \ref{thm:H-1}. Lastly, Section \ref{sec:r-nr} deals with the proof of the coexistence result, Theorem \ref{coex}.

Throughout the paper, we use the boldface letters for \textbf{Cases I} and \textbf{II} for clear distinction. The paper follows a parallel exposition to deal with  both cases simultaneously. But for a better  readability, we strongly recommend for a reader to follow each case  one at a time.

\section{Existence by a general density approach}\label{sec:approach}

Here and below,  we assume without loss of generality that the initial datum $u_0\in W^{1,\infty}(\Omega)$ in problem (\ref{ib-P}) satisfies
\begin{equation}\label{av-0}
\int_\Omega u_0(x)\,dx=0,
\end{equation}
since otherwise we may solve (\ref{ib-P}) with initial datum $\tilde u_0=u_0-\bar u_0,$ where $\bar u_0=\frac{1}{|\Omega|}\int_\Omega u_0\,dx$.

\subsection{Motivation: non-homogeneous partial differential inclusion} To motivate our approach, let us reformulate problem (\ref{ib-P}) into a non-homogeneous partial differential inclusion.

First, assume that we have a function $\Phi=(u^*,v^*)\in W^{1,\infty}(\Omega_T;\R^{1+n})$  satisfying
\begin{equation}\label{bdry-1}
\begin{cases}
 u^*(x,0)=u_0(x), &x\in \Omega,\\
\dv v^*(x,t)=u^*(x,t), &  \mbox{a.e. $(x,t)\in\Omega_T$},
\\
v^*(\cdot,t) \cdot \n|_{\partial\Omega} =0, &  t\in [0,T],
\end{cases}
\end{equation}
which will be called a \emph{boundary function} for the initial datum $u_0.$

Let $A\in C(\R^n;\R^n)$ be the diffusion flux. For each $l\in\R$, let $K(l)$   be a subset of the matrix space $\mathbb M^{(1+n)\times (n+1)}$  defined by
\begin{equation}\label{set-K-l}
 K(l)= \left\{ \begin{pmatrix} p & c\\ B & A(p)\end{pmatrix}\,\Big | \,
  p\in\R^n,\, c\in\R,\,B\in\mathbb{M}^{n\times n},\,  \tr B=l
 \right\}.
\end{equation}

Suppose that a function $w=(u,v)\in W^{1,\infty}(\Omega_T;\R^{1+n})$ solves  the Dirichlet problem of non-homogeneous partial differential inclusion
\begin{equation}\label{pdi-D}
\begin{cases} \nabla w(x,t)\in  K(u(x,t)),  &  \textrm{a.e. $(x,t)\in\Omega_T$,}\\
w(x,t)= \Phi(x,t), &(x,t)\in\partial \Omega_T,
\end{cases}
\end{equation}
where $\nabla w$ is the space-time Jacobian matrix of $w$:
\[
\nabla w=\begin{pmatrix}Du & u_t\\ Dv & v_t\end{pmatrix}.
\]
Then it can be easily shown that $u$ is a Lipschitz solution to  $(\ref{ib-P})$.

Although solving problem (\ref{pdi-D}) is sufficient to obtain a Lipschitz solution to (\ref{ib-P}), all known general existence results \cite{DM1,MSy} are not applicable to solve    (\ref{pdi-D}). If  $n=1$,  (\ref{pdi-D})   implies $v_x=u$ and thus $v\in W^{1,\infty}$ if $u\in W^{1,\infty}$; so Zhang \cite{Zh, Zh1} was able to solve (\ref{pdi-D}) for the Perona-Malik and H\"ollig types in dimension $n=1$ by using a modified convex integration method.  However, for $n\ge 2$, since it is impossible to bound $\|Dv\|_{L^\infty(\Omega_T)}$ in terms of $\dv v$, the function $v$  may not be in $W^{1,\infty}(\Omega_T;\R^n)$ even when $u\in W^{1,\infty}(\Omega_T);$ therefore, solving (\ref{pdi-D}) in $W^{1,\infty}(\Omega_T;\R^{1+n})$ is impractical in dimensions $n\ge 2$. To overcome this difficulty,  we make  the following key observation.

\begin{pro} \label{gen-lem} Suppose $u\in W^{1,\infty}(\Omega_T)$ is such that $u(x,0)=u_0(x),$   there exists a vector function $v\in   W^{1,2}((0,T);L^2(\Omega;\R^n))$ with weak time-derivative $v_t$  satisfying
\begin{equation}\label{v-t}
v_t=A(Du) \quad \mbox{a.e.\, in $\Omega_T$,}
\end{equation}
 and  for each $\zeta\in C^\infty(\bar\Omega_T)$ and each   $t\in [0,T],$
\begin{equation}\label{div-v=u}
\int_\Omega v(x,t)\cdot D\zeta(x,t)\,dx=-\int_\Omega u(x,t)\zeta(x,t)\,dx.
\end{equation}
Then $u$ is a Lipschitz solution to $(\ref{ib-P})$.
\end{pro}

\begin{proof}  To verify (\ref{def:sol}), given any $\zeta\in C^\infty(\bar\Omega_T)$, let
\[
g(t)=\int_\Omega u(x,t)\zeta(x,t)dx,\;\;  h(t)=\int_\Omega u(x,t)\zeta_t(x,t)dx \quad (t\in [0,T]).
\]
 Then by (\ref{div-v=u}), for each $\psi\in C^\infty_c(0,T),$
\[
\int_0^T \psi_t g  dt  = -\int_0^T\int_\Omega \psi_t v\cdot D\zeta dxdt,\;\;
\int_0^T \psi h  dt = -\int_0^T\int_\Omega \psi v\cdot D\zeta_t dxdt.
\]
Since $v\in W^{1,2}((0,T);L^2(\Omega;\R^n))$ and $v_t=A(Du)$ a.e. in $\Omega_T$, we have
\[
\int_0^T\int_\Omega (\psi D\zeta)_t\cdot v dxdt=  - \int_0^T\int_\Omega  A(Du)\cdot \psi    D\zeta \,dxdt.
\]
As  $(\psi  D\zeta )_t=\psi_t D\zeta + \psi D\zeta_t$, combining  the previous equations, we obtain
\[
\int_0^T \psi_t g\,dt= \int_0^T \psi \left (-h+\int_\Omega   A(Du)\cdot D\zeta\,dx\right )dt,
\]
which proves that $g$ is weakly differentiable in $(0,T)$ with its weak derivative
\[
g'(t)=h(t)-\int_\Omega A(Du(x,t))\cdot D\zeta(x,t)\,dx \; \; \mbox{ a.e. $t\in(0,T)$.}
\]
From this, upon integration,  (\ref{def:sol}) follows for each $s\in [0,T].$
\end{proof}

\subsection{Admissible set and the density approach}\label{subsec-density} Let  $\Phi=(u^*,v^*)$   be any boundary function for $u_0$  defined by  (\ref{bdry-1}) above.
Denote by $W^{1,\infty}_{u^*}(\Omega_T)$, $W^{1,\infty}_{v^*}(\Omega_T;\R^n)$ the usual \emph{Dirichlet classes} with boundary traces $u^*, \, v^*,$ respectively.

We say that $\mathcal U\subset W^{1,\infty}_{u^*}(\Omega_T)$ is  an \emph{admissible set}  provided that it is  nonempty and  bounded in $W^{1,\infty}_{u^*}(\Omega_T)$ and that for each $u\in \mathcal U$,  there exists a vector function  $v\in W_{v^*}^{1,\infty}(\Omega_T; \R^n)$ satisfying
\[
\mbox{$\dv v=u$ \, a.e. in $\Omega_T$,\quad$\|v_t\|_{L^\infty(\Omega_T)}\le R$,}
\]
 where $R>0$ is any fixed number. If\; $\mathcal U$ is an admissible set,  for each  $\epsilon>0$,  let $\mathcal U_\epsilon$ be   the set of all  $u\in\mathcal U$ such that  there exists a  function $v\in W_{v^*}^{1,\infty}(\Omega_T; \R^n)$ satisfying
\[
\begin{split}
&\mbox{$\dv v=u$ \, a.e. in $\Omega_T$,\quad$\|v_t\|_{L^\infty(\Omega_T)}\le R$,}\\
&\int_{\Omega_T} |v_t(x,t)-A(Du(x,t))|\,dxdt \leq\epsilon|\Omega_T|.\end{split}
\]

The following general existence theorem relies on a pivotal density hypothesis   of\;  $\mathcal U_\epsilon$ in $\mathcal U.$

\begin{thm}\label{thm:main} Let\, $\mathcal U\subset W^{1,\infty}_{u^*}(\Omega_T)$ be an admissible set satisfying the  {\em density property:}
\begin{equation}\label{den-0}
\mbox{$\mathcal U_\epsilon$  is dense in $\mathcal U$ under the $L^\infty$-norm for each  $\epsilon>0$.}
\end{equation}
Then,  given any $\varphi\in \mathcal U$, for each $\delta>0,$
there exists a Lipschitz solution $u\in  W_{u^*}^{1,\infty}(\Omega_T)$ to   $(\ref{ib-P})$ satisfying
$\|u-\varphi\|_{L^\infty(\Omega_T)}<\delta.$
Furthermore,  if $\mathcal U$ contains a function  which is not a Lipschitz solution to $(\ref{ib-P}),$ then $ (\ref{ib-P})$ itself admits  infinitely  many Lipschitz solutions.
\end{thm}

\begin{proof}
For clarity, we divide the proof into several steps.

1. Let $\X$ be the closure of\, $\mathcal U$ in the metric space $L^\infty(\Omega_T).$
Then $(\mathcal X,L^\infty)$ is a nonempty complete metric space.  By assumption, each $\mathcal U_\epsilon$ is dense in $\X.$ Moreover, since $\mathcal U$ is bounded in $W_{u^*}^{1,\infty}(\Omega_T)$,  we have  $\X\subset W_{u^*}^{1,\infty}(\Omega_T)$.

2. Let  $\Y =L^1(\Omega_T;\R^{n})$. For $h>0$, define $T_h\colon  \X \to \Y$ as follows. Given any $u\in \X$, write $u=u^* +w$ with $w\in W_0^{1,\infty}(\Omega_T)$ and define
\[
T_h (u) =Du^* + D(\rho_h * w),
\]
where $\rho_h(z)=h^{-N}\rho(z/h)$, with $z=(x,t)$ and $N=n+1$, is the standard $h$-mollifier in $\R^{N}$, and $\rho_h * w$ is the usual convolution in $\R^{N}$ with $w$ extended to be zero outside $\bar{\Omega}_T.$
Then, for each $h>0$, the map $T_h \colon   (\X, L^\infty) \to (\Y, L^1)$ is continuous, and for each $u\in \X$,
\[
\lim_{h\to 0^+} \|T_h (u)-Du\|_{L^1(\Omega_T)}=\lim_{h\to 0^+} \|\rho_h * Dw-Dw\|_{L^1(\Omega_T)}=0.
\]
Therefore, the spatial gradient operator
$D\colon \mathcal X\to \mathcal  Y$ is the pointwise limit of a sequence of continuous maps $T_h \colon \X\to \Y$; hence $D\colon \mathcal X\to \mathcal  Y$ is  a {\em Baire-one map}.
By Baire's category theorem (e.g., \cite[Theorem 10.13]{BBT}), there exists a
{\em residual set} $\G\subset \mathcal X$ such that the operator $D$ is
continuous at each point of $\mathcal G.$  Since $\X\setminus \G$  is of the {\em first category}, the set $\G$ is {\em dense} in $\X$. Therefore, given any $\varphi\in \X,$ for each $\delta>0$, there exists a function $u\in \G$ such that $\|u-\varphi\|_{L^\infty(\Omega_T)}<\delta.$

3. We now prove that each   $u\in \G$ is a Lipschitz solution to (\ref{ib-P}).   Let $u\in \G$ be given. By the density of $\mathcal U_\epsilon$ in $(\X,L^\infty)$ for each $\epsilon>0$, for every $j\in\mathbb N$, there exists a function $ u_j\in\mathcal U_{1/j}$ such that $\|u_j-u\|_{L^\infty(\Omega_T)} <1/j$. Since the  operator $D\colon (\X, L^\infty)\to (\Y, L^1)$ is continuous at $u$,  we have  $Du_j\to Du$ in $L^1(\Omega_T;\R^n).$  Furthermore, from (\ref{bdry-1}) and the definition of\, $\mathcal U_{1/j}$, there exists a  function $v_j\in W^{1,\infty}_{v^*}(\Omega_T;\R^n)$ such that for each $\zeta\in C^\infty(\bar\Omega_T)$ and each   $t\in [0,T],$
\begin{equation}\label{div-v3}
\begin{split} & \int_\Omega v_j(x,t)\cdot D\zeta(x,t)\,dx  =-\int_\Omega u_j(x,t)\zeta(x,t)\,dx,\\
 \|( & v_j)_t\|_{L^\infty(\Omega_T)}  \le R,\quad  \int_{\Omega_T} |(v_j)_t-A(Du_j)|\,dxdt \leq\frac{1}{j}|\Omega_T|.\end{split}
\end{equation}
Since $v_j(x,0)=v^*(x,0)\in  W^{1,\infty}(\Omega;\R^n)$ and  $\|(v_j)_t\|_{L^\infty(\Omega_T)}  \le R$, it follows that both sequences $\{v_j\}$ and  $\{(v_j)_t\}$ are bounded in $L^2(\Omega_T;\R^n)\approx L^2((0,T);L^2(\Omega;\R^n)).$  So we may  assume
\[
\mbox{$v_j \wcon v $ and $(v_j)_t\wcon v_t$ in $L^2((0,T);L^2(\Omega;\R^n))$}
\]
 for some $v\in W^{1,2}((0,T);L^2(\Omega;\R^n)),$ where $\wcon$ denotes the weak convergence.  Upon taking the limit as $j\to \infty$ in (\ref{div-v3}), since  $v\in C([0,T];L^2(\Omega;\R^n))$ and $A\in C(\R^n;\R^n)$,  we obtain
\[
\begin{split}
    \int_\Omega  v(x,t)&\cdot D\zeta(x,t)\,dx    =  -\int_\Omega u(x,t)\zeta(x,t)\,dx \quad (t\in [0,T]), \\
 &v_t(x,t)= A(Du(x,t)) \quad   a.e. \; (x,t)\in \Omega_T. \end{split}
\]
Consequently, by Proposition \ref{gen-lem}, $u$ is a Lipschitz solution to (\ref{ib-P}).

4. Finally, assume $\mathcal U$ contains a function which is not a Lipschitz solution to (\ref{ib-P}); hence   $\G\ne \mathcal U.$ Then $\G$ cannot be a finite set, since otherwise the $L^\infty$-closure $\X=\overline{\G}  =\overline{\mathcal U}$ would be a  finite set, making  $\mathcal U=\G.$ Therefore, in this case,   (\ref{ib-P}) admits infinitely many Lipschitz solutions.
The proof is complete.
\end{proof}

The following result implies that the density approach streamlined in Theorem \ref{thm:main} can be useful only when   problem (\ref{ib-P}) is \emph{non-parabolic}, that is, when $A(p)$ is non-monotone.

\begin{pro}\label{unique}
If $A\colon \R^n\to \R^n$ is monotone, then   $(\ref{ib-P})$  can have at most one Lipschitz solution.
\end{pro}
\begin{proof}
Let $u,\,\tilde u\in W^{1,\infty}(\Omega_T)$ be any two Lipschitz solutions to (\ref{ib-P}). Use the identity (\ref{def:sol}) for both $u$ and $\tilde u$, subtract the two  equations, take $\zeta=u-\tilde u$ (by approximations), and then apply the monotonicity of $A(p)$:
\[
(A(Du)-A(D\tilde u))\cdot (Du-D\tilde u) \ge 0
\]
to obtain
\[
\int_\Omega (u(x,s)-\tilde u(x,s))^2 dx\le \int_0^s\int_\Omega (u-\tilde u)(u-\tilde u)_t \,dxdt
\]
\[
=\frac12 \int_\Omega (u(x,s)-\tilde u(x,s))^2 dx \quad (0\le s\le T).
\]
Thus $u\equiv \tilde u$ in $\Omega_T.$
\end{proof}

As a general property for Lipschitz solutions to (\ref{ib-P}) for all continuous fluxes $A(p)$ satisfying  a \emph{positivity} condition below, we prove the following result; clearly, the positivity condition is  satisfied by the fluxes $A(p)$ given by (\ref{fun-A}) with the profiles $\sigma(s)$ of the Perona-Malik and H\"ollig types illustrated in Figures \ref{fig1}, \ref{fig2} and \ref{fig2-2}. Note that the positivity condition is consistent with the Clausius-Duhem inequality in the second law of thermodynamics (e.g., \cite[p. 79]{Dy} and \cite[p. 116]{Tr}).

\begin{pro}\label{pro:minmax} Let $u_0\in W^{1,\infty}(\Omega)$ and $A\colon \R^n\to \R^n$ satisfy the \emph{positivity  condition:}
\[
A(p)\cdot p\ge 0\;\; (p\in \R^n).
\]
Then any Lipschitz solution $u$ to $(\ref{ib-P})$ satisfies
\begin{equation}\label{minmax}
\min_{\bar\Omega}u_0\le u  \le \max_{\bar\Omega}u_0 \;\;\; \mbox{in $\Omega_T.$}
\end{equation}
\end{pro}

\begin{proof}
Let $u\in W^{1,\infty}(\Omega_T)$ be any Lipschitz solution to (\ref{ib-P}). By (\ref{def:sol}),
for all $\zeta\in C^\infty(\bar\Omega_T)$,
\[
 \int_{\Omega_T} u_t(x,t)\zeta(x,t)dxdt =-\int_{\Omega_T} A(D u)\cdot D \zeta dxdt;
\]
hence by approximation, this equality holds for all  $\zeta\in W^{1,\infty}(\Omega_T).$ Taking $\zeta(x,t)=\phi(x,t)\psi(t)$ with arbitrary $\phi\in W^{1,\infty}(\Omega_T)$ and $\psi\in W^{1,\infty}(0,T)$, we deduce that
\[
\int_\Omega u_t (x,t)\phi(x,t)\,dx = -\int_\Omega A(Du(x,t))\cdot D\phi(x,t)\,dx
\]
for a.e.\,$t\in (0,T)$ and all $\phi\in W^{1,\infty}(\Omega_T)$. Now taking $\phi=u^{2k-1}$ with $k=1,2,\cdots$, we have for a.e.\,$t\in (0,T),$
\[
\begin{split}
 \frac{d}{dt}\Big(&\int_\Omega u^{2k} dx\Big)=2k\int_\Omega u_t \phi\,dx=-2k\int_\Omega A(Du)\cdot D\phi dx \\
 &=-2k(2k-1)\int_\Omega u^{2k-2} A(Du)\cdot Du dx \le 0.
\end{split}
\]
From this we deduce that $L^{2k}$-norm of $u(\cdot,t)$ is non-increasing on $t\in [0,T]$; in particular,
\[
\|u(\cdot,t)\|_{L^{2k}(\Omega)}\le \|u_0\|_{L^{2k}(\Omega)}\quad \forall\; t\in [0,T], \;\; k=1,2,\cdots.
\]
Letting $k\to \infty$, we obtain $\|u(\cdot,t)\|_{L^{\infty}(\Omega)}\le \|u_0\|_{L^{\infty}(\Omega)}$; hence
\begin{equation}\label{max-norm}
\|u\|_{L^{\infty}(\Omega_T)} = \|u_0\|_{L^{\infty}(\Omega)}.
\end{equation}
Now let $
m_1=\min_{\bar\Omega} u_0$ and $m_2=\max_{\bar\Omega} u_0.$
We  show $m_1\le u(x,t)\le m_2$ for all $(x,t)\in \Omega_T$  to complete the proof. We proceed with three cases.

{\bf (a)}: $m_2>0$ and $|m_1|\le m_2.$ In this case,  $\|u_0\|_{L^\infty(\Omega)}= m_2$; so by (\ref{max-norm})
\[
u(x,t)\le \|u\|_{L^\infty(\Omega_T)}=\|u_0\|_{L^\infty(\Omega)}= m_2.
\]
To obtain the lower bound, let $\tilde u_0=-u_0+m_2+m_1$ and $\tilde u=-u+m_2+m_1.$ Then $\tilde u$ is a Lipschitz solution to (\ref{ib-P}) with new flux function $\tilde A(p)=-A(-p)$ and initial data $\tilde u_0.$ Since $m_1\le \tilde u_0(x)\le m_2$, as above, we have $\tilde u(x,t)\le m_2$; hence $u(x,t)\ge m_1$ for all $(x,t)\in \Omega_T.$

{\bf (b)}: $m_2>0$ and $m_1<-m_2.$ Let $\tilde u_0=-u_0$ and $\tilde u=-u.$ Then $\tilde u$ is a Lipschitz solution to (\ref{ib-P}) with new flux function $\tilde A(p)=-A(-p)$ and initial data $\tilde u_0.$ Since $-m_2\le \tilde u_0(x)\le -m_1$ for all $x\in \Omega$ and $-m_1>0, \; |-m_2|=m_2\le -m_1,$ it follows from Case (a) that $-m_2\le \tilde u(x,t)\le -m_1$ and hence $m_1\le u(x,t)\le m_2$ for all $(x,t)\in \Omega_T.$

{\bf (c)}: $m_2\le 0.$ In this case $m_1\le 0.$ If $m_1=0$ then $m_2=0$ and hence $u_0\equiv 0;$ so, by (\ref{max-norm}), $u\equiv 0$. Now assume $m_1<0.$ Let again as  in Case (b) $\tilde u_0=-u_0$ and $\tilde u=-u.$  Since $-m_2\le \tilde u_0(x)\le -m_1$ for all $x\in \Omega$ and $-m_1>0, \; |-m_2|=-m_2\le -m_1,$ it follows again from Case (a) that $-m_2\le \tilde u(x,t)\le -m_1$ and hence $m_1\le u(x,t)\le m_2$ for all $(x,t)\in \Omega_T.$
\end{proof}

The rest of the paper is devoted to the construction of   suitable boundary functions $\Phi=(u^*,v^*)$ and admissible sets $\mathcal U\subset W_{u^*}^{1,\infty}(\Omega_T)$ fulfilling the  {density property} (\ref{den-0}) for \textbf{Cases I} and \textbf{II}.

\section{Structural conditions on the profiles: \\ detailed statements of   main theorems}\label{sec:structure}
In this section, we assume the domain $\Omega$ and initial datum $u_0$ satisfy (\ref{assume-1}). We consider the diffusion  fluxes  $A(p)$ given by  (\ref{fun-A}) and  present the  detailed statements of our main theorems by specifying  the structural conditions on the profiles $\sigma(s)=s f(s^2)$ illustrated   in Figures \ref{fig1}, \ref{fig2} and \ref{fig2-2}. 

\subsection{Case I: Perona-Malik type of equations} In this case, we assume  the following structural condition on the profile   $\sigma(s).$

\textbf{Hypothesis (PM):} (See Figures \ref{fig1} and \ref{fig4}.)
\begin{itemize}
\item[(i)]  There exists a number $s_0>0$ such that
\[
f\in  C^0([0,\infty))\cap C^{3}([0,s_0^2))\cap C^1(s_0^2,\infty).
\]
\item[(ii)]
$\sigma'(s)>0\;\;\forall s\in[0,s_0),\; \sigma'(s)<0\;\;\forall s\in(s_0,\infty)$, and
\[
\lim_{s\to\infty}\sigma(s)=0.
\]
\end{itemize}
In this case, for  each $r\in(0,\sigma(s_0))$, let $s_-(r)\in(0,s_0)$ and $s_+(r)\in(s_0,\infty)$ denote the unique numbers with $r=\sigma(s_\pm(r))$. Then by (ii),
\begin{equation}\label{pro-pm-ii}
\lim_{r\to 0^+}s_-(r)=0, \quad \lim_{r\to 0^+}s_+(r)=\infty.
\end{equation}
Note that both profiles  in (\ref{PM-fun}) for the  Perona-Malik model \cite{PM}   satisfy Hypothesis (PM).

The following is the first main result of this paper in detail.

\begin{thm}\label{thm:PM-1} Let $\Omega$ be convex and $M_0=\|Du_0\|_{L^\infty(\Omega)}$. Then for each $r\in(0,\sigma(M_0))$, there exists a number $l=l_r\in(0,r)$ such that for all $\tilde r\in(l, r)$ and all but at most countably many  $\bar r\in(0,\tilde r)$, there exist two disjoint   open sets $\,\Omega_T^1,\,\Omega_T^2\subset\Omega_T$ with $|\Omega_T^1\cup\Omega_T^2|=|\Omega_T|$ and infinitely many Lipschitz solutions $u$ to $(\ref{ib-P})$ satisfying
\[
u\in C^{2+\alpha,1+\alpha/2}(\bar{\Omega}^1_T),\quad u_t=\dv(A(Du))\;\; \mbox{pointwise in}\;\;\Omega_T^1,
\]
\[
|Du(x,t)|<s_-(\bar r)\;\;\forall(x,t)\in\Omega_T^1, \quad \Omega_0^{\bar r}\subset\partial\Omega_T^1,
\]
\begin{equation*}
|S|+|L|=|\Omega^2_T|,\;\;|L|>0,
\end{equation*}
where
\[
\Omega_0^{\bar r}=\{(x,0) \,|\, x\in\Omega,\,|Du_0(x)|<s_-(\bar r) \},
\]
\[
S=\{(x,t)\in \Omega_T^2\;|\; s_-(\bar r) \le |Du(x,t)|\le s_-(r)\},
\]
and
\[
L=\{(x,t)\in \Omega_T^2\;|\;  s_+(r)\le|Du(x,t)|\le s_+(\tilde r)\}.
\]
\end{thm}

\begin{remk}\label{remk-PM-1}
As we will see later, once the numbers $\tilde r,\,\bar r$ are chosen as above, the corresponding Lipschitz solutions in the theorem  are all identically equal to a single function $u^*$ in $\Omega^1_T$, which is the classical solution to some modified uniformly parabolic Neumann problem. Here the function $u^*$ depends on $\tilde r$ but not on $\bar r$, whereas $\Omega^1_T,\,\Omega^2_T$ do on $\bar r$:
\[
\Omega^1_T=\{(x,t)\in \Omega_T\, | \, |Du^*(x,t)|<s_-(\bar r)\},
\]
\[
\Omega^2_T=\{(x,t)\in \Omega_T\, | \, |Du^*(x,t)|>s_-(\bar r)\}.
\]
The choice of $\bar r$ is made to guarantee that the interface $\Omega_T\setminus (\Omega^1_T\cup\Omega^2_T)$ has $(n+1)$-dimensional measure zero; this will be crucial for the proof of the theorem later. Choosing $\bar r=\tilde r$ may not be safe in this regard, since we have not enough information on the function $u^*$ to be sure that the interface measure $|\{|Du^*|=s_-(\tilde r)\}|=0$. This forces us to sacrifice the benefit of the choice $\bar r=\tilde r$ that would separate the space-time domain $\Omega_T$ into two disjoint parts  where the Lipschitz solutions are $C^{2+\alpha,1+\alpha/2}$ in one but nowhere $C^1$ in the other.
\end{remk}

\begin{remk}
By (\ref{pro-pm-ii}),  if $0<r\ll \sigma(M_0)$,  the corresponding Lipschitz solutions $u$ have \emph{large} and \emph{small} gradient regimes $L$ and $\Omega^1_T\cup S$ in $\Omega_T$ up to measure zero,  representing the almost constant and sharp edge parts of $u$ in $\Omega_T$, respectively. Although there is a fine mixture of the disjoint regimes $L,\,S\subset\Omega^2_T$ due to a micro-structured ramping with alternate gradients of finite size,   such properties together with (\ref{minmax}) for solutions $u$ are \emph{somehow}  reflected in numerical simulations; see Figure \ref{fig3}, taken from Perona and Malik \cite{PM}.  On the other hand, it has been observed in \cite{BF, BFG} that as the limits of solutions to a class of regularized equations, infinitely many different evolutions may  arise under the same initial datum $u_0$. Our non-uniqueness result seems to reflect this pathological behavior of forward-backward problem (\ref{ib-P}).
\end{remk}

\begin{figure}[ht]
\begin{center}
\includegraphics[scale=0.4]{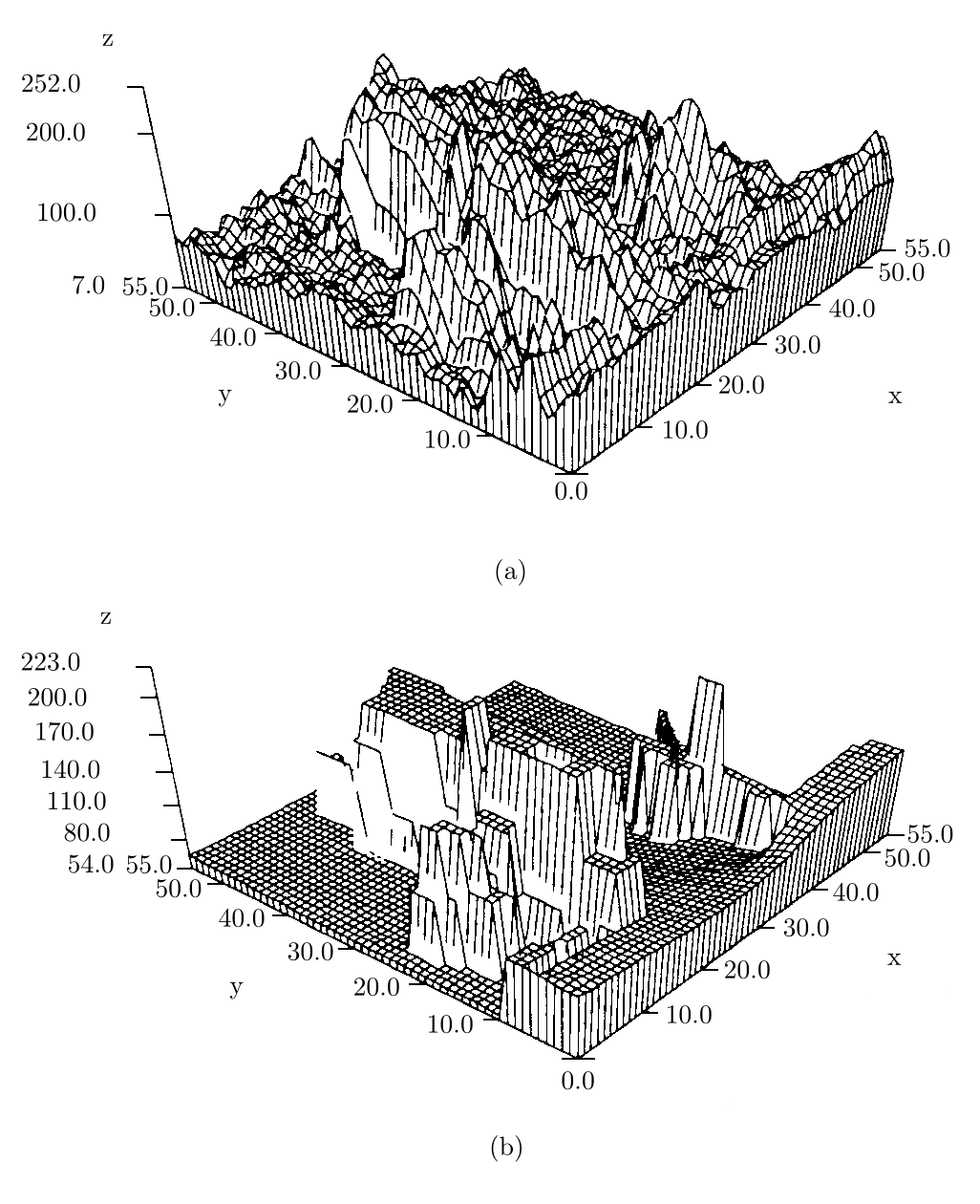}
\end{center}
\caption{Scale-space using anisotropic diffusion profile $\sigma(s)=\frac{s}{1+s^2/s^2_0}.$ Three dimensional plot of the brightness of Fig. 12 in \cite{PM}. (a) Original image, (b) after smoothing with anisotropic diffusion.}
\label{fig3}
\end{figure}

We now restate Theorem \ref{thm:2} only for \textbf{Case I}. The proof of this appears in Subsection \ref{subsec-modi}.

\begin{thm}\label{thm:PM-2}
Let $\Omega$ and $M_0$ be as in Theorem \ref{thm:PM-1}. If $M_0<s_0$, then \emph{(\ref{ib-P})} has a unique solution $u\in C^{2+\alpha,1+\alpha/2}(\bar \Omega_T)$ satisfying  $\|Du\|_{L^\infty(\Omega_T)}=\|Du_0\|_{L^\infty(\Omega)}$.
\end{thm}

For the given initial datum $u_0$ with $M_0<s_0$, by Theorem \ref{thm:PM-1}, there are infinitely many Lipschitz solutions to problem (\ref{ib-P}). On the other hand, by the theorem here, we have a \emph{special} Lipschitz solution $u$ to (\ref{ib-P}), which is also classical. It may be interesting to observe that none of the Lipschitz solutions from Theorem \ref{thm:PM-1} coincide with this special solution $u$.

\subsection{Case II: H\"ollig type of equations} In this case, we impose  the following structural condition on the profile   $\sigma(s).$

\textbf{Hypothesis (H):} (See Figures \ref{fig2}, \ref{fig2-2} and \ref{fig5}.)

\begin{itemize}
\item[(i)]  There exist two numbers $s_2>s_1>0$ such that
\[
f\in C^{0}([0,\infty))\cap C^{1+\alpha}([0,s_1^2)\cup(s_2^2,\infty)).
\]

\item[(ii)] $\sigma'(s)>0\;\;\forall s\in[0,s_1)\cup(s_2,\infty)$, $\sigma(s_1)>\sigma(s_2)\ge 0$, $\lambda\le\sigma'(s)\le\Lambda\;\;\forall s\ge 2s_2$, where $\Lambda\ge\lambda>0$ are constants.

\item[(iii)]  Let $s_1^*\in[0,s_1)$ and $s_2^*\in(s_2,\infty)$ denote the unique numbers with
$
\sigma(s_1^*)=\sigma(s_2)
$
and $\sigma(s_2^*)=\sigma(s_1)$, respectively.
\end{itemize}
Note that the cases $\sigma(s_2)>0$ and $\sigma(s_2)=0$ correspond to Figures \ref{fig2} and \ref{fig2-2}, respectively. In addition to Hypothesis (H), if we suppose $\sigma(s)\ge 0$ for all $s\in(s_1,s_2)$, then $A(p)\cdot p\ge 0$ for all $p\in \R^n$, and so (\ref{minmax}) is satisfied by any Lipschitz solution $u$ to (\ref{ib-P}); but we do not explicitly assume such positivity for \textbf{Case II} unless otherwise stated.

With Hypothesis (H),  for  each $r\in(\sigma(s_2),\sigma(s_1))$, let $s_-(r)\in(s_1^*,s_1)$ and $s_+(r)\in(s_2,s_2^*)$ denote the unique numbers with $r=\sigma(s_\pm(r))$.

The following theorem is the second main result of this paper that generalizes those of \cite{Ho,Zh1} to any dimension $n\ge 1$.

\begin{thm}\label{thm:H-1}
Let $M_0=\|Du_0\|_{L^\infty(\Omega)}$, $M_0'=\min\{M_0,s_1\}$, and $|Du_0(x_0)|\in(s_1^*,s_2^*)$ for
some $x_0\in\Omega.$    Then for each $r\in (\sigma(s_2),\sigma(M_0'))$, there exists a number $l=l_r\in (\sigma(s_2),r)$ such that for all $\tilde r\in (l,r)$ and all but countably many  $\bar r_1 \in (\sigma(s_2),\tilde r)$, $\bar r_3 \in (r,\sigma(s_1))$, there exist three disjoint open sets $\Omega^1_T,\,\Omega^2_T,\,\Omega^3_T \subset \Omega_T$ with $|\Omega^1_T\cup\Omega^2_T\cup\Omega^3_T|=|\Omega_T|$ and infinitely many Lipschitz solutions $u$ to \emph{(\ref{ib-P})} satisfying
\[
u\in C^{2+\alpha,1+\alpha/2}(\bar{\Omega}^1_T\cup \bar{\Omega}^3_T),\quad u_t=\dv(A(Du))\;\; \mbox{pointwise in}\;\;\Omega_T^1\cup \Omega_T^3,
\]
\[
|Du(x,t)|<s_-(\bar r_1)\;\;\forall(x,t)\in\Omega_T^1, \quad \Omega_0^{\bar r_1}\subset\partial\Omega_T^1,
\]
\[
|Du(x,t)|>s_+(\bar r_3)\;\;\forall(x,t)\in\Omega_T^3, \quad \Omega_0^{\bar r_3}\subset\partial\Omega_T^3,
\]
\begin{equation*}
|S|+|L|=|\Omega^2_T|,\;\;|S|>0,\;\;|L|>0,
\end{equation*}
where
\[
\Omega_0^{\bar r_1}=\{(x,0) \,|\, x\in\Omega,\,|Du_0(x)|<s_-(\bar r_1) \},
\]
\[
\Omega_0^{\bar r_3}=\{(x,0) \,|\, x\in\Omega,\,|Du_0(x)|>s_+(\bar r_3) \},
\]
\[
S=\{(x,t)\in \Omega_T^2\;|\; s_-(\bar r_1) \le |Du(x,t)|\le s_-(r)\},
\]
and
\[
L=\{(x,t)\in \Omega_T^2\;|\;  s_+(\tilde r)\le|Du(x,t)|\le s_+(\bar r_3)\}.
\]
\end{thm}

In regard to this theorem, an explanation similar to Remark \ref{remk-PM-1} can be made; but we omit this.
As a byproduct, we also have the following simple existence result whose proof appears after that of Theorem \ref{thm:H-1}.

\begin{coro}\label{coro:H-3}
For any initial datum $u_0\in C^{2+\alpha}(\bar\Omega)$ with $Du_0\cdot\n|_{\partial\Omega}=0$,   \emph{(\ref{ib-P})} has at least one Lipschitz solution.
\end{coro}

We now restate Theorem \ref{thm:2} only for \textbf{Case II}.

\begin{thm}\label{thm:H-2}
Let $\Omega$ be convex and $M_0=\|Du_0\|_{L^\infty(\Omega)}$. Assume further that $f\in C^3([0,s_1^2)\cup(s_2^2,\infty))$. If $M_0<s_1$, then   \emph{(\ref{ib-P})} has a unique solution $u\in C^{2+\alpha,1+\alpha/2}(\bar \Omega_T)$ satisfying  $\|Du\|_{L^\infty(\Omega_T)}=\|Du_0\|_{L^\infty(\Omega)}$.
\end{thm}

Unlike Theorem \ref{thm:H-1} and its corollary, the convexity assumption on $\Omega$ cannot be dropped in this theorem. If $s_1^*< M_0 < s_1$, Theorems \ref{thm:H-1} and \ref{thm:H-2} apply to give infinitely many Lipschitz solutions and one \emph{special} Lipschitz (classical) solution to problem (\ref{ib-P}) under the same initial datum $u_0$, respectively. However, a subtlety arises when $M_0\le s_1^*$: Is there a Lipschitz solution to (\ref{ib-P}) other than the classical one in this case? Let us discuss more on this in the remark below.

\begin{remk}[Breakdown of uniqueness]
In this remark, let  $\sigma(s)$ satisfy
\begin{equation}\label{breakdown}
\sigma(s)> 0\quad\forall s\in(s_1,s_2]
\end{equation}
in addition to Hypothesis (H) (as in Figure \ref{fig2}), and let $f$ be as in Theorem \ref{thm:H-2}.  Then by Corollary \ref{coro:H-3}, for any $u_0\in C^{2+\alpha}(\bar\Omega)$ with $Du_0\cdot\n|_{\partial\Omega}=0$,  problem (\ref{ib-P}) has at least one Lipschitz solution. In particular, if $M_0=0$ with $M_0:=\|Du_0\|_{L^\infty(\Omega)}$, that is, $u_0$ is constant in $\Omega$, then by Proposition \ref{pro:minmax}, the constant function $u\equiv u_0$ in $\Omega_T$ is a unique solution to (\ref{ib-P}). So a natural question is to ask if there is a number $\tilde M_0>0$ such that (\ref{ib-P}) has a unique Lipschitz solution for all initial data  $u_0\in C^{2+\alpha}(\bar\Omega)$ with $Du_0\cdot\n|_{\partial\Omega}=0$ and $M_0\le \tilde M_0$. To address this question, define
\[
M^*_0=\sup\left\{\tilde M_0\ge 0 \, \Big| \, \begin{array}{l}
                                  \mbox{$\forall u_0\in C^{2+\alpha}(\bar\Omega)$ with $Du_0\cdot\n|_{\partial\Omega}=0$ and} \\

                                  \mbox{$M_0\le\tilde M_0$, (\ref{ib-P}) has a unique Lipschitz solution}
                                \end{array}
 \right\},
\]
which we may call as the \emph{critical threshold of uniqueness} (CTU) for (\ref{ib-P}).
The CTU $M^*_0$ may depend not only on the profile $\sigma(s)$ but also on the dimension $n$, final time $T$, and domain $\Omega$ of problem (\ref{ib-P}). 
By Theorem \ref{thm:H-1}, we have an estimate:
\[
0\le M^*_0\le s^*_1,
\]
which is valid for all dimensions $n\ge 1$, final times $T>0$, and domains $\Omega$ satisfying (\ref{assume-1}). Our question now boils down to the point of asking if $M^*_0>0$.

When $n=1$, the answer is affirmative; in this case, the CTU $M^*_0$  is positive and depends only on the profile $\sigma(s)$. This fact can be derived by using an \emph{a priori} estimate on the spatial derivatives of Lipschitz solutions.
From \cite{HN}, when $n=1$, any Lipschitz solution $u$ to (\ref{ib-P}) satisfies
\begin{equation}\label{apriori}
\|u_x\|_{L^\infty(\Omega_T)}\le\frac{\|\sigma(|u_0'|)\|_{L^\infty(\Omega)}}{c},
\end{equation}
where $c:=\inf_{s\in(0,\infty)}\frac{\sigma(s)}{s}>0$, by Hypothesis (H) and (\ref{breakdown}). Note
\[
\frac{\sigma(s_1^*)}{s_1}>\frac{\sigma(s_2)}{s_2}\ge c\,;
\]
that is, $\sigma(s_1^*)> cs_1$. As $\sigma(0)=0$ and $\sigma(s)$ is strictly increasing on $s\in[0,s_1^*]$, there is a unique number $M^{l}_0\in(0,s_1^*)$ with $\sigma(M^{l}_0)=cs_1$. Suppose now that $\|u_0'\|_{L^\infty(\Omega)}<M^{l}_0$. Then by Theorem \ref{thm:H-2}, there exists a classical solution $\tilde u\in C^{2+\alpha,1+\alpha/2}(\bar\Omega_T)$ to (\ref{ib-P}) with $\|\tilde u_x\|_{L^\infty(\Omega_T)}=\|u_0'\|_{L^\infty(\Omega)}<M^l_0.$ On the other hand, if $u$ is any Lipschitz solution to (\ref{ib-P}),   we have from (\ref{apriori}) that  $\|u_x\|_{L^\infty(\Omega_T)}< s_1.$ Modifying the profile $\sigma(s)$ to the right of the threshold $\max\{\|u_0'\|_{L^\infty(\Omega)},\|u_x\|_{L^\infty(\Omega_T)}\}<s_1$, both $\tilde u$ and $u$ become Lipschitz solutions to some uniformly parabolic  problem of \emph{type} (\ref{ib-P}) with monotone flux under the same initial datum $u_0$. By Proposition \ref{unique}, we thus have $\tilde u\equiv u$ in $\Omega_T$; hence, the Lipschitz solution $\tilde u$ to (\ref{ib-P}) is unique. Therefore,   we have an improved estimate for the CTU $M^{*}_0$ in $n=1$:
\[
0<M^{l}_0\le M^{*}_0\le s^*_1.
\]

As a toy example for $n=1$, we consider the general piecewise linear profile $\sigma(s)$ of the H\"ollig type. Let $s_1,\,s_2,\,k_1,\,k_2$ and $k_3$ be positive numbers with $s_2>s_1$ and $-k_2 (s_2-s_1)+k_1s_1>0$.
Let us take
\[
\sigma(s) = \left\{ \begin{array}{ll}
                      k_1 s, & 0\le s\le s_1, \\
                      -k_2(s-s_1)+k_1 s_1, & s_1\le s\le s_2, \\
                      k_3(s-s_2)-k_2(s_2-s_1)+k_1 s_1, & s\ge s_2;
                    \end{array}
 \right.
\]
then, if $-k_2 (s_2-s_1)+k_1s_1 < k_3s_2 $, we have
\[
0<-s_1\frac{k_2}{k_1}+\frac{s_1^2}{s_2}\big(1+\frac{k_2}{k_1}\big)=M^{l}_0\le M^{*}_0\le s^*_1 = -s_2\frac{k_2}{k_1}+s_1\big(1+\frac{k_2}{k_1}\big),
\]
or if $-k_2 (s_2-s_1)+k_1s_1 \ge k_3s_2 $, it follows that
\[
0<s_1\frac{k_3}{k_1}=M^{l}_0\le M^{*}_0\le s^*_1 = -s_2\frac{k_2}{k_1}+s_1\big(1+\frac{k_2}{k_1}\big).
\]

Unfortunately,  estimate (\ref{apriori}) from \cite{HN} is not directly generalized to dimensions $n\ge2$. However, we still expect that the CTU $M^*_0>0$ even in dimensions $n\ge 2$ for all $T>0$ as long as $\Omega$ is convex. Since this study is beyond the scope of this paper, we leave it to the interested readers.


\end{remk}

\section{Some useful results}\label{sec:prelim}

This section prepares some essential ingredients for the proofs of existence theorems, Theorems \ref{thm:PM-1} and \ref{thm:H-1}.

\subsection{Uniformly parabolic equations} We refer to the standard references (e.g., \cite{LSU, Ln}) for some notations concerning functions and domains of class $C^{k+\alpha}$ with an integer $k\ge 0$.

Assume $\tilde f\in C^{1+\alpha}([0,\infty))$ is a function satisfying
\begin{equation}\label{para}
  \theta \le \tilde f(s)+2s\tilde f'(s)\le \Theta \quad \forall\;  s\ge 0,
\end{equation}
where $\Theta\ge\theta>0$ are constants. This condition is equivalent to $\theta\le (s\tilde f(s^2))'\le \Theta$ for all $s\in\R;$ hence, $\theta\le \tilde f(s)\le\Theta$ for all $s\ge 0.$ Let
\[
\tilde A(p)=\tilde f(|p|^2)p \quad (p\in \R^n).
\]
Then we have
\[
\tilde A^i_{p_j}(p)  = \tilde f(|p|^2)\delta_{ij} + 2\tilde f'(|p|^2) p_ip_j \quad (i,j=1,2,\cdots,n;\; p\in \R^n)
\]
and hence  the \emph{uniform ellipticity condition}:
\begin{equation}\label{para-0}
\theta |q|^2 \le \sum_{i,j=1}^n \tilde A^i_{p_j}(p) q_iq_j\le \Theta |q|^2\quad \forall\; p,\;q\in\R^n.
\end{equation}

\begin{thm}\label{existence-gr-max} Assume  $(\ref{assume-1})$.
Then the initial-Neumann boundary value problem
\begin{equation}\label{ib-parabolic}
  \begin{cases}
  u_t=\dv (\tilde A(Du)) & \mbox{in }\Omega_T, \\
  {\partial u}/{\partial \n}=0 & \mbox{on }\partial\Omega\times(0,T), \\
  u(x,0)=u_0(x) & \mbox{for } x\in\Omega
\end{cases}
\end{equation}
has a unique solution $u\in C^{2+\alpha,1+\alpha/2}(\bar\Omega_T)$. Moreover, if $\tilde f\in C^{3}([0,\infty))$ and $\Omega$ is convex, then the \emph{gradient maximum principle} holds:
\begin{equation}\label{gr-max-P}
\|Du\|_{L^\infty(\Omega_T)}=\|Du_0\|_{L^\infty(\Omega)}.
\end{equation}
\end{thm}

\begin{proof}  1. As problem (\ref{ib-parabolic}) is uniformly parabolic by (\ref{para-0}),  the existence of unique  classical solution $u$ in $C^{2+\alpha,1+\alpha/2}(\bar\Omega_T)$  follows from the standard theory; see \cite[Theorem 13.24]{Ln}.

2. To prove  the gradient  maximum principle (\ref{gr-max-P}),  assume hereafter $\tilde f\in C^{3}([0,\infty))$ and $\Omega$ is convex. Note that, since $\tilde A\in C^3 (\R^n)$, a standard bootstrap argument based on the regularity theory of linear parabolic equations \cite{LSU, Ln} shows that the solution $u$ has all continuous partial derivatives $u_{x_ix_jx_k}$ and $u_{x_it}$   within $\Omega_T$  for $1\le i,j,k\le n.$

3. Let $v=|Du|^2.$ Then, within $\Omega_T,$ we compute
\begin{eqnarray*}
&\Delta v =2 Du \cdot D(\Delta u) + 2 |D^2u|^2,&\\
&u_t=\dv (\tilde A(Du))= \dv (\tilde f(v)Du)= \tilde f'(v) Dv \cdot Du + \tilde f(v) \Delta u,&\\
&
Du_t =  \tilde f''(v) (Dv\cdot Du) Dv + \tilde f'(v) (D^2u)  Dv &\\
&+ \tilde f'(v) (D^2v) Du + \tilde f'(v) (\Delta u) Dv + \tilde f(v) D(\Delta u).
\end{eqnarray*}
Putting these equations into $v_t= 2Du\cdot Du_t$, we obtain
\begin{equation}\label{para-2}
v_t- \mathcal L (v) - B\cdot Dv =-2\tilde f(|Du|^2) |D^2u|^2\le 0 \quad \mbox{in $\Omega_T$,}
\end{equation}
where operator $\mathcal L(v)$ and  coefficient $B$ are defined by
\begin{eqnarray*}
&\mathcal L(v)= \tilde f(|Du|^2)\Delta v + 2 \tilde f'(|Du|^2) Du\cdot (D^2 v)Du,&\\
&B= 2\tilde f''(v) (Dv\cdot Du)Du   + 2 \tilde f'(v) (D^2u)  Du +
 2 \tilde f'(v) (\Delta u)  Du.&
\end{eqnarray*}
We write  $\mathcal L (v)= \sum_{i,j=1}^n a_{ij} v_{x_ix_j},$
 with   coefficients $a_{ij}=a_{ij}(x,t)$ given by
\[
a_{ij}=\tilde A^i_{p_j}(Du)=\tilde f(|Du|^2)\delta_{ij} + 2 \tilde f'(|Du|^2)u_{x_i}u_{x_j} \quad  (i,j=1,\cdots,n).
\]
Note that on $\bar\Omega_T$  all eigenvalues of the matrix $(a_{ij})$ lie in $[\theta,\Theta]$.

4. We  show
\[
 \max_{(x,t)\in\bar\Omega_T}v(x,t) = \max_{x\in \bar\Omega} v(x,0),
\]
which proves (\ref{gr-max-P}).  We prove this by contradiction. Suppose
\begin{equation}\label{claim-1}
M:=\max_{(x,t)\in\bar\Omega_T}v(x,t) > \max_{x\in \bar\Omega} v(x,0).
\end{equation}
  Let  $(x_0,t_0)\in\bar\Omega_T$ be such that  $v(x_0,t_0)=M;$ then $t_0>0.$ If $x_0\in \Omega$, then the strong maximum principle applied to (\ref{para-2})  would imply  that $v$ is constant on $\Omega_{t_0},$ which yields  $v(x,0)\equiv M$ on $\bar\Omega$, a contradiction to (\ref{claim-1}). Consequently $x_0\in \partial\Omega$ and thus $v(x_0,t_0)=M>v(x,t)$ for all $(x,t)\in \Omega_T.$   We can then apply Hopf's Lemma for parabolic equations \cite{PW}  to (\ref{para-2}) to deduce $
 \partial v(x_0,t_0) /\partial \n >0.$ However, mainly due the convexity of $\Omega$, a result of \cite[Lemma 2.1]{AR} (see also \cite[Theorem 2]{Ka}) asserts  that $\partial v/\partial\n\le 0$ on $\partial\Omega\times [0,T],$ which gives a desired contradiction.
 \end{proof}

\subsection{Modification of the  profile functions}\label{subsec-modi}

The following elementary results can be proved in a similar way as in \cite{CZ, Zh}; we omit the proofs.

\begin{lem}[\textbf{Case I:} Perona-Malik type; see Figure \ref{fig4}]\label{lem:modi-PM}
Assume \emph{Hypothesis (PM)}.
For every $0< r_1<r_2<\sigma(s_0)$, there exists a function  $\tilde\sigma\in C^{3}([0,\infty))$  such that
\begin{equation}
\label{modi-PM-1}
\begin{cases} \tilde \sigma(s)
                =\sigma(s), & 0\leq s \leq s_-(r_1), \\
                   \tilde\sigma(s)   <\sigma(s), & s_-(r_1)< s < s_+(r_2),\\
  \theta \le \tilde\sigma'(s) \le \Theta, & 0\leq s<\infty
\end{cases}
\end{equation}
for some constants  $\Theta\ge\theta>0.$
With such a function $\tilde \sigma$, define $\tilde f(s)=\tilde\sigma(\sqrt s)/\sqrt s$ ($s>0$) and $\tilde f(0)=f(0);$ then $\tilde f\in C^3([0,\infty))$ fulfills condition $(\ref{para})$.
\end{lem}

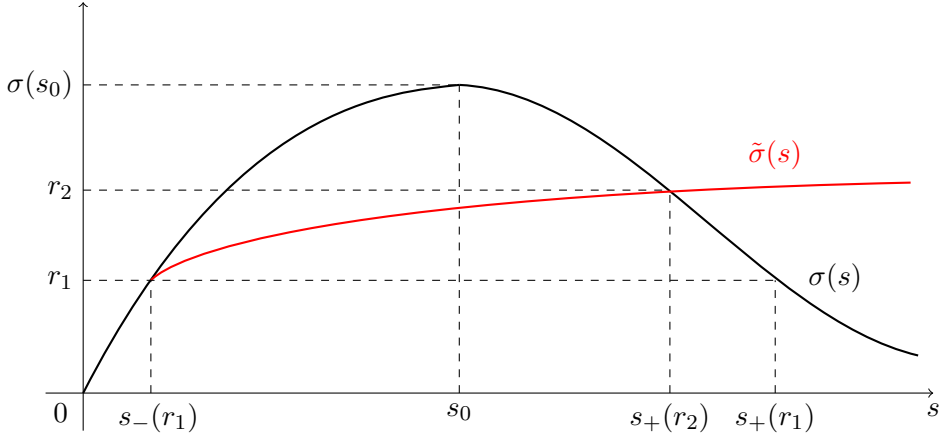
\begin{figure}[ht]
\begin{center}
\begin{tikzpicture}[scale = 1]
    \draw[->] (-.5,0) -- (11.3,0);
	\draw[->] (0,-.5) -- (0,5.2);
 \draw[dashed] (0,4.1)--(5,4.1);
    \draw[dashed] (5, 0)  --  (5, 4.1) ;
	\draw[thick]   (0, 0) .. controls  (2, 3.9) and (4.1,4) ..(5,4.1);
	\draw[thick]   (5, 4.1) .. controls  (7.2, 4) and (9,1) ..(11.1,0.5 );
\draw[thick,red]   (0.9,1.5) .. controls  (1,1.6) and (2,2.6) ..(11,2.8);
	\draw (11.3,0) node[below] {$s$};
    \draw (-0.3,0) node[below] {{$0$}};
    \draw (10, 1.2) node[above] {$\sigma(s)$};
    \draw (0, 4.1) node[left] {$\sigma(s_0)$};
   \draw (0, 1.5) node[left] {$r_1$};
   \draw (0, 2.7) node[left] {$r_2$};
 \draw[dashed] (0,1.5)--(9.2, 1.5);
 \draw[dashed] (0,2.7)--(7.8, 2.7);
 \draw[dashed] (7.8,0)--(7.8, 2.7);
\draw[dashed] (9.2,0)--(9.2, 1.5);
 \draw (9.2,0) node[below] {$s_+(r_1)$};
\draw (9.2,3.5) node[red,below] {$\tilde\sigma(s)$};
 \draw (7.8,0) node[below] {$s_+(r_2)$};
    \draw (5, 0) node[below] {$s_0$};
 \draw[dashed] (0.9,0)--(0.9, 1.5);
 \draw (1, 0) node[below] {$s_-(r_1)$};
    \end{tikzpicture}
\end{center}
\caption{ {\bf Case I:} Perona-Malik type of profile  $\sigma(s)$ and modified function \textcolor{red}{$\tilde\sigma(s)$}. 
}
\label{fig4}
\end{figure}

\begin{lem}[\textbf{Case II:} H\"ollig type; see Figure \ref{fig5}]\label{lem:modi-H}
Assume \emph{Hypothesis (H)}.
For every $\sigma(s_2)< r_1<r_2<\sigma(s_1)$, there exists a function  $\tilde\sigma\in C^{1+\alpha}([0,\infty))$  such that
\begin{equation}\label{modi-C-1}
\begin{cases}
                     \tilde\sigma(s)=\sigma(s), & 0\leq s \leq s_-(r_1), \\
                     \tilde\sigma(s)<\sigma(s), & s_-(r_1)< s \leq s_-(r_2), \\
                    \tilde\sigma(s) >\sigma(s), & s_+(r_1)\le s < s_+(r_2), \\
                    \tilde\sigma(s) =\sigma(s), & s_+(r_2)\le s < \infty,\\
                   \theta \le \tilde\sigma'(s) \le \Theta, & 0\leq s<\infty
\end{cases}
\end{equation}
for some constants  $\Theta\ge\theta>0.$
With such a function $\tilde \sigma$, define $\tilde f(s)=\tilde\sigma(\sqrt s)/\sqrt s$ ($s>0$) and $\tilde f(0)=f(0);$ then $\tilde f\in C^{1+\alpha}([0,\infty))$ fulfills condition $(\ref{para})$. Moreover, if $f\in C^{3}([0,s_1^2)\cup(s_2^2,\infty))$ in addition, then $\tilde\sigma,\,\tilde f$ can be chosen to be also in $C^3([0,\infty))$.
\end{lem}

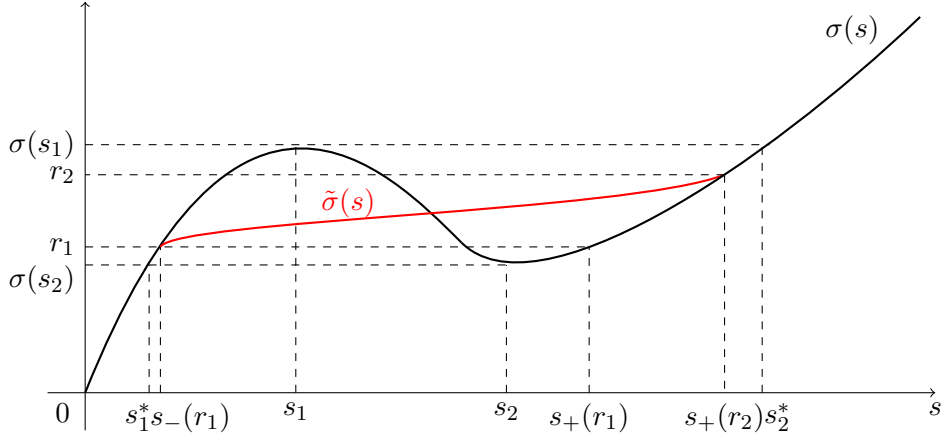
\begin{figure}[ht]
\begin{center}
\begin{tikzpicture}[scale = 1]
    \draw[->] (-.5,0) -- (11.3,0);
	\draw[->] (0,-.5) -- (0,5.2);
 \draw[dashed] (0,1.7)--(5.6,1.7);
 \draw[dashed] (0,3.3)--(9.1,3.3);
 \draw[dashed] (9,0)--(9,3.3);
    \draw[dashed] (5.6, 0)  --  (5.6, 1.7) ;
   \draw[dashed] (0.85, 0)  --  (0.85, 1.7) ;
\draw[dashed] (1, 0)  --  (1, 1.94) ;
\draw[dashed] (0,1.94)  --  (6.7, 1.94) ;
\draw[dashed] (6.7,0)  --  (6.7, 1.94) ;
   \draw[dashed] (2.8, 0)  --  (2.8, 3.3) ;
	\draw[thick]   (0, 0) .. controls (2,5) and  (4, 3)   ..(5,2);
          \draw[thick,red]   (1, 1.94) .. controls (1.2,2.3) and  (7.2, 2.4)   ..(8.5,2.9);
	\draw[thick]   (5, 2) .. controls  (6, 1) and (9,3) ..(11.1, 5 );
	\draw (11.3,0) node[below] {$s$};
    \draw (-0.3,0) node[below] {{$0$}};
    \draw (10.2, 4.5) node[above] {$\sigma(s)$};
 \draw (0, 3.3) node[left] {$\sigma(s_1)$};
 \draw (0, 1.5) node[left] {$\sigma(s_2)$};
 \draw (0, 1.94) node[left] {$r_1$};
 \draw (0, 2.9) node[left] {$r_2$};
\draw[dashed] (8.5,2.9)--(0,2.9);
\draw[dashed] (8.5,2.9)--(8.5, 0);
   \draw (5.6, 0) node[below] {$s_2$};
   \draw (3.5, 2.2) node[red,above] {$\tilde\sigma(s)$};
   \draw (2.8, 0) node[below] {$s_1$};
   \draw (0.7, 0) node[below] {$s_1^*$};
   \draw (1.4, 0) node[below] {$s_-(r_1)$};
   \draw (9.2, 0) node[below] {$s_2^*$};
   \draw (8.5, 0) node[below] {$s_+(r_2)$};
   \draw (6.7, 0) node[below] {$s_+(r_1)$};
    \end{tikzpicture}
\end{center}
\caption{{\bf Case II:}  H\"ollig type of profile   $\sigma(s)$ and modified function \textcolor{red}{$\tilde\sigma(s)$}.  
}
\label{fig5}
\end{figure}

 \subsection{Proofs of Theorems \ref{thm:PM-2} and \ref{thm:H-2}}  Let $M_0=\|Du_0\|_{L^\infty(\Omega)}.$ In \textbf{Case I} (Perona-Malik type), we have $0<M_0<s_0$; so we can select $0<r_1<r_2<\sigma(s_0)$  such   that $s_-(r_1)=M_0.$ In \textbf{Case II} (H\"ollig type), we have $0<M_0<s_1$; so we can select $\sigma(s_2)<r_1<r_2<\sigma(s_1)$ such that $s_-(r_1)\ge M_0.$   Use such a choice of $r_1,\, r_2$   in Lemma \ref{lem:modi-PM} (\textbf{Case I}), Lemma \ref{lem:modi-H} (\textbf{Case II}) to obtain a $C^3$ function $\tilde f$ as stated in the lemma. Let $\tilde A(p)=\tilde f(|p|^2)p.$ For this $\tilde A(p)$, problem (\ref{ib-parabolic}) is uniformly parabolic; hence, by Theorem \ref{existence-gr-max}, (\ref{ib-parabolic}) has a unique solution $u\in C^{2+\alpha,1+\alpha/2}(\bar\Omega_T)$ satisfying
\[
\|Du\|_{L^\infty(\Omega_T)}=\|Du_0\|_{L^\infty(\Omega)}.
\]
Since $\tilde A(p)=A(p)$ on $|p|\le s_-(r_1)$ and $s_-(r_1)\ge M_0$,   it follows that $\tilde A(Du(x,t))=A(Du(x,t))$ in $\Omega_T;$ this proves $u$ is a classical solution to problem (\ref{ib-P}). On the other hand, we easily see that any classical solution to (\ref{ib-P}) satisfying $\|Du\|_{L^\infty(\Omega_T)}=M_0$ is also a classical solution to (\ref{ib-parabolic}) and hence must be unique. This completes the proofs of Theorems \ref{thm:PM-2} and \ref{thm:H-2}.

\subsection{Right inverse of the  divergence operator} To deal with   linear constraint $\dv v=u$, we follow an argument of \cite[Lemma 4]{BB} to construct  a right inverse $\mathcal R$ of the divergence operator: $\dv \mathcal R=Id$ (in the sense of distributions in $\Omega_T$). For the purpose of this paper, the construction of $\mathcal R$ is restricted to  {\em box} domains, by which we mean  domains given by $Q=J_1\times J_2\times \cdots\times J_n$, where $J_i=(a_i,b_i)\subset\R$ is a finite open  interval.

Given a  box  $Q$, we define a linear operator $\mathcal R_n\colon L^\infty(Q)\to L^\infty(Q;\R^n)$ inductively on dimension $n$ as follows.
If $n=1$, for $u\in L^\infty(J_1)$, we define $v=\mathcal R_1 u$ by
\[
v(x_1)=\int_{a_1}^{x_1} u(s)ds\quad (x_1\in J_1).
\]
Assume $n=2$. Let $u\in L^\infty(J_1\times J_2).$ Set
$\tilde u(x_1)=\int_{a_2}^{b_2} u(x_1,s)\,ds$ for $x_1\in J_1.$ Then $\tilde u\in L^\infty(J_1).$ Let $\tilde v=\mathcal R_1\tilde u$; that is,
\[
\tilde v(x_1)=\int_{a_1}^{x_1}\tilde u(s)ds=\int_{a_1}^{x_1}\int_{a_2}^{b_2} u(s,\tau)\,d\tau ds \quad (x_1\in J_1).
\]
 Let $\rho_2\in C^\infty_c(a_2,b_2)$ be such that $0\le \rho_2(s)\le\frac{C_0}{b_2-a_2}$ and $\int_{a_2}^{b_2} \rho_2(s) ds=1.$
Define $v=\mathcal R_2 u\in L^\infty(J_1\times J_2;\R^2)$  by $v=(v^1,v^2)$ with $v^1 (x_1,x_2)=\rho_2(x_2) \tilde v(x_1)$ and
\[
v^2(x_1,x_2) =\int_{a_2}^{x_2}u(x_1,s)ds -\tilde u(x_1) \int_{a_2}^{x_2} \rho_2(s) ds.
\]
Note that if $u\in W^{1,\infty}(J_1\times J_2)$ then $\tilde u\in W^{1,\infty}(J_1)$; hence $v=\mathcal R_2 u\in W^{1,\infty}(J_1\times J_2;\R^2)$ and $\dv v = u$ a.e.\,in $J_1\times J_2.$ Moreover, if $u\in C^1(\overline{J_1\times J_2})$ then $v$ is in $C^1(\overline{J_1\times J_2};\R^2).$

Assume that we have defined the operator $\mathcal R_{n-1}$. Let $u\in L^\infty(Q)$ with $Q=J_1\times J_2\times \cdots\times J_n$  and $x=(x',x_n)\in Q$, where $x'\in Q'=J_1\times\cdots\times J_{n-1}$ and $x_n\in J_n.$ Set
$\tilde u(x')=\int_{a_n}^{b_n} u(x',s)\,ds$ for $x'\in Q'.$
Then $\tilde u\in L^{\infty}(Q').$ By  the  assumption, $\tilde v =\mathcal R_{n-1}\tilde u\in L^\infty(Q';\R^{n-1})$ is defined.  Write $\tilde v(x')=(Z^1(x'),\cdots,Z^{n-1}(x')),$ and let $\rho_n\in C^\infty_c(a_n,b_n)$ be a  function satisfying
$0\le \rho_n(s)\le\frac{C_0}{b_n-a_n}$ and $\int_{a_n}^{b_n} \rho_n(s) ds=1.$
Define $v=\mathcal R_n u\in L^\infty(Q;\R^n)$ as follows.  For $x=(x',x_n)\in Q,$  $v(x)=(v^1(x),v^2(x),\cdots,v^n(x))$ is defined by
\[
\begin{split}
& v^k(x',x_n)=\rho_n(x_n) Z^k(x') \quad  (k=1,2,\cdots, n-1),\\
&v^n (x',x_n)=\int_{a_n}^{x_n}u(x',s)ds -\tilde u(x') \int_{a_n}^{x_n} \rho_n(s) ds.
\end{split}
\]
Then  $\mathcal R_n \colon L^\infty(Q)\to L^\infty(Q;\R^n)$  is a well-defined linear operator; moreover,
\begin{equation}\label{div-0}
 \|\mathcal R_n  u\|_{L^\infty(Q)}  \le C_n \,(|J_1|+\cdots+|J_n|) \|u\|_{L^\infty(Q)},
\end{equation}
where $C_n>0$ is a constant depending only on $n.$

As in the case $n=2$, we  see that if $u\in W^{1,\infty}(Q)$ then $v=\mathcal R_n u\in W^{1,\infty}(Q;\R^n)$ and
$\dv v = u$ a.e.\,in $Q$. Also, if $u\in C^1(\bar Q)$ then $v=\mathcal R_n u$ is in $C^1(\bar Q;\R^n).$ Moreover,   if $u\in W^{1,\infty}_0(Q)$ satisfies  $\int_Q u(x)dx=0$, then one can easily show that $v=\mathcal R_n u\in W_0^{1,\infty}(Q;\R^n).$

Let  $I$ be a finite open interval in $\R$. We now  extend the operator $\mathcal R_n$ to an operator $\mathcal R$ on $L^\infty(Q\times I)$ by defining,  for a.e.\,$(x,t)\in Q\times I$,
\begin{equation}\label{def-R}
(\mathcal R u)(x,t)=(\mathcal R_n u(\cdot,t))(x) \quad \forall\; u\in L^\infty(Q\times I).
\end{equation}
 Then $\mathcal R \colon  L^\infty(Q\times I)\to L^\infty(Q\times I;\R^n)$ is a bounded linear operator.

We have the following result.

\begin{thm}\label{div-inv} Let $u\in W^{1,\infty}_0(Q\times I)$ satisfy  $\int_{Q}u(x,t)\,dx=0$ for all $t\in I$. Then $v=\mathcal R u\in W^{1,\infty}_0(Q\times I;\R^n)$,  $\dv v=u$ a.e.\,in $Q\times I$, and
\begin{equation}\label{div-1}
\|v_t \|_{L^\infty(Q\times I)}  \le C_n \,(|J_1|+\cdots+|J_n|) \|u_t\|_{L^\infty(Q\times I)},
\end{equation}
where $Q=J_1\times\cdots\times J_n$ and $C_n$ is the same constant as in $(\ref{div-0})$.  Moreover, if $u\in C^1(\overline{Q\times I})$ then $v =\mathcal R u \in C^1(\overline{Q\times I};\R^{n}).$
\end{thm}
\begin{proof} Given $u\in  W^{1,\infty}_0(Q\times I)$, let $v=\mathcal R u.$ We easily verify that $v$ is Lipschitz continuous in $t$ and hence  $v_t$ exists. It also follows that $v_t=\mathcal R (u_t).$ Clearly, if $\int_Q u(x,t)dx=0$ then $v(x,t)=0$ whenever $t\in \partial I$ or $x\in \partial Q$. This proves $v\in W^{1,\infty}_0(Q\times I;\R^{n})$ and  the estimate (\ref{div-1}) follows from (\ref{div-0}). Finally, from the definition of $\mathcal R u$, we see that if $u\in C^1(\overline{Q\times I})$ then $v =\mathcal R u \in C^1(\overline{Q\times I};\R^{n}).$
\end{proof}

\section{Geometry of the relevant matrix sets}\label{sec:geometry}

Let $A(p)$ be given by (\ref{fun-A}). We assume Hypothesis (PM) or  (H) unless one is chosen specifically.  Recall the definition (\ref{set-K-l}) with $l=0$:
\begin{equation}\label{set-K}
 K_0=K(0)= \left\{ \begin{pmatrix} p & c\\ B & A(p)\end{pmatrix}\,\Big | \,
  p\in\R^n,\, c\in\R,\,B\in\mathbb{M}^{n\times n},\,  \tr B=0
 \right\}.
\end{equation}

Under Hypothesis  (PM) or (H), certain structures  of  the set $K_0$ turn out to be still quite useful, especially when it comes to the relaxation of homogeneous partial differential inclusion $\nabla \omega(z)\in K_0$ with $z=(x,t)$ and $\omega=(\varphi,\psi)$. We investigate these structures and establish such relaxation results throughout this section.

\subsection{Geometry of the matrix set $K_0$} We study some subsets of $K_0$, depending on the different types of profiles.

\subsubsection*{\bf Case I: Hypothesis (PM)} (See Figure \ref{fig4}.) In this case, we assume the following. Fix any two numbers $0<r_1<r_2<\sigma(s_0)$, and let $F_0=F_{r_1,r_2}(0)$ be the subset of $K_0$ defined by
\[
F_0=\left\{ \begin{pmatrix} p & c\\ B & A(p)\end{pmatrix}\,\Big|\, \begin{array}{l}
                                                                                         p\in\R^n, \, |p|\in(s_-(r_1),s_-(r_2))\cup(s_+(r_2),s_+(r_1)), \\
                                                                                         c\in\R, \, B\in \mathbb M^{n\times n},\, \tr B=0
                                                                                       \end{array}
   \right\}.
\]
We decompose the set $F_0$ into two disjoint subsets as follows:
\[
F_-=\left\{ \begin{pmatrix} p & c\\ B & A(p)\end{pmatrix}\,\Big|\, \begin{array}{l}
                                                                                         p\in\R^n, \, |p|\in(s_-(r_1),s_-(r_2)), \\
                                                                                         c\in\R, \, B\in \mathbb M^{n\times n},\, \tr B=0
                                                                                       \end{array}
   \right\},
\]
\[
F_+=\left\{ \begin{pmatrix} p & c\\ B & A(p)\end{pmatrix}\,\Big|\, \begin{array}{l}
                                                                                         p\in\R^n, \, |p|\in(s_+(r_2),s_+(r_1)), \\
                                                                                         c\in\R, \, B\in \mathbb M^{n\times n},\, \tr B=0
                                                                                       \end{array}
   \right\}.
\]

\subsubsection*{\bf Case II:  Hypothesis (H)} (See Figure \ref{fig5}.) In this case, we assume the following. Fix any two numbers $\sigma(s_2)<r_1<r_2<\sigma(s_1)$, and let $F_0=F_{r_1,r_2}(0)$ be the subset of $K_0$ given by
\[
F_0=\left\{ \begin{pmatrix} p & c\\ B & A(p)\end{pmatrix}\,\Big|\, \begin{array}{l}
                                                                                         p\in\R^n, \, |p|\in(s_-(r_1),s_-(r_2))\cup(s_+(r_1),s_+(r_2)), \\
                                                                                         c\in\R, \, B\in \mathbb M^{n\times n},\, \tr B=0
                                                                                       \end{array}
   \right\}.
\]
The set $F_0$ is also decomposed into two disjoint subsets as follows:
\[
F_-=\left\{ \begin{pmatrix} p & c\\ B & A(p)\end{pmatrix}\,\Big|\, \begin{array}{l}
                                                                                         p\in\R^n, \, |p|\in(s_-(r_1),s_-(r_2)), \\
                                                                                         c\in\R, \, B\in \mathbb M^{n\times n},\, \tr B=0
                                                                                       \end{array}
   \right\},
\]
\[
F_+=\left\{ \begin{pmatrix} p & c\\ B & A(p)\end{pmatrix}\,\Big|\, \begin{array}{l}
                                                                                         p\in\R^n, \, |p|\in(s_+(r_1),s_+(r_2)), \\
                                                                                         c\in\R, \, B\in \mathbb M^{n\times n},\, \tr B=0
                                                                                       \end{array}
   \right\}.
\]


In order to study the homogeneous differential inclusion $\nabla \omega(z)\in K_0$, we first scrutinize the rank-one structure of the set $F_0\subset K_0$. We introduce the following notation.

\begin{defn}
For a given set $E\subset \mathbb{M}^{(1+n)\times(n+1)}$, $L(E)$ is defined to be the set of all matrices $\xi\in \mathbb M^{(1+n)\times (n+1)}$ that are not in $E$ but are representable by $\xi=\lambda\xi_1+(1-\lambda)\xi_2$ for some $\lambda\in (0,1)$ and $\xi_1,\xi_2\in E$  with  $\rank(\xi_1-\xi_2)=1$, or equivalently,
\[
L(E)=\{ \xi\not\in E\; |\; \mbox{$\xi+t_\pm\eta \in E$ for some  $t_-<0<t_+$ and $\rank\eta=1$}\}.
\]
\end{defn}

For the matrix set $F_0$, we define
\[
R(F_0)=\bigcup_{\xi_\pm\in F_\pm,\,\mathrm{rank}(\xi_+-\xi_-)=1}(\xi_-,\xi_+),
\]
where $(\xi_-,\xi_+)$ is the open line segment in $\mathbb{M}^{(1+n)\times(n+1)}$ joining $\xi_\pm$. Then from careful analyses, one can actually deduce
\begin{equation}\label{laminate-PM}
L(F_0)=R(F_0)\cup L(F_+)\quad \mbox{in {\bf Case I}}
\end{equation}
and
\begin{equation}\label{laminate-C}
L(F_0)=R(F_0) \quad \mbox{in {\bf Case II}}.
\end{equation}
In (\ref{laminate-PM}), due to the backward nature of $\sigma(s)$ on $(s_+(r_2),s_+(r_1))$ for \textbf{Case I}, the set $L(F_+)$ turns out to be non-empty. On the other hand, as only forward parts of $\sigma$ are involved in $F_0$ for \textbf{Case II}, no such set appears in (\ref{laminate-C}). Regardless of this discrepancy, we can only stick to the analysis of the set $R(F_0)$ for both cases towards the existence results, Theorems \ref{thm:PM-1} and \ref{thm:H-1}.

We perform the step-by-step analysis of the set $R(F_0)$ for both cases simultaneously.

\subsubsection{\bf Alternate expression for $R(F_0)$} We  investigate more specific criteria for  matrices in  $R(F_0)$.
\begin{lem}\label{lem-form}
Let $\xi\in\mathbb{M}^{(1+n)\times(n+1)}$. Then
$\xi\in R(F_0)$ if and only if there exist numbers $t_-<0<t_+$ and vectors $q,\,\gamma\in\R^n$ with $|q|=1,\,\gamma\cdot q=0$ such that for each $b\in\R\setminus\{0\}$, if
$\eta=\begin{pmatrix} q & b\\ \frac{1}{b}q\otimes\gamma & \gamma\end{pmatrix}$, then
$\xi+t_\pm\eta\in F_\pm$.
\end{lem}

\begin{proof}
Assume $\xi=\begin{pmatrix} p & c\\ B & \beta\end{pmatrix}\in R(F_0)$. By definition, $
\xi+t_\pm\tilde\eta\in F_\pm,
$
where $t_-<0<t_+$ and $\tilde\eta$ is a rank-one matrix given by
\[
\tilde\eta=\begin{pmatrix} a\\ \alpha \end{pmatrix}\otimes (q,\tilde b)=\begin{pmatrix} aq & a\tilde b\\ \alpha\otimes q & \tilde b\alpha\end{pmatrix},\quad a^2+|\alpha|^2\neq0, \quad \tilde b^2+|q|^2\neq0,
\]
for some $a,\,\tilde b\in\R$ and $\alpha,\,q\in\R^n$; here $\alpha\otimes q$ denotes the rank-one or zero matrix $(\alpha_iq_j)$ in $\mathbb{M}^{n\times n}$. Condition $\xi+t_\pm\tilde\eta\in F_\pm$ with $t_-<0<t_+$ is equivalent to the following:
For \textbf{Case I},
\begin{equation}\label{form-PM}
\begin{split}
&\tr B=0,\quad \alpha\cdot q=0,\quad A(p+t_\pm aq)=\beta+t_\pm \tilde b\alpha,\\  |p+t&_+ aq|\in(s_+(r_2),s_+(r_1)),\quad |p+t_- aq|\in(s_-(r_1),s_-(r_2)).
\end{split}
\end{equation}
For \textbf{Case II},
\begin{equation}\label{form-C}
\begin{split}
&\tr B=0,\quad \alpha\cdot q=0,\quad A(p+t_\pm aq)=\beta+t_\pm \tilde b\alpha,\\  |p+t&_+ aq|\in(s_+(r_1),s_+(r_2)),\quad |p+t_- aq|\in(s_-(r_1),s_-(r_2)).
\end{split}
\end{equation}
Therefore, $aq\neq 0$. Upon rescaling $\tilde\eta$ and $t_\pm$, we can assume $a=1$ and $|q|=1$; namely,
\[
\tilde\eta=\begin{pmatrix} q & \tilde b\\ \alpha\otimes q & \tilde b\alpha\end{pmatrix},\quad |q|=1,\quad \alpha\cdot q =0.
\]
We now let $\gamma=\tilde b\alpha$. Let $b\in\R\setminus\{0\}$ and
\begin{equation*}
\eta=\begin{pmatrix} q & b\\ \frac{1}{b}\gamma\otimes q & \gamma\end{pmatrix}.
\end{equation*}
From (\ref{form-PM}) or (\ref{form-C}), it follows that $\xi+t_\pm\eta\in F_\pm$.

The converse directly follows from the definition of $R(F_0)$.
\end{proof}

\subsubsection{\bf Diagonal components of matrices in $R(F_0)$.}
The following gives a  description for the diagonal components of matrices in  $R(F_0)$.
\begin{lem}\label{lem-rough}
\begin{equation}\label{rough-1}
R(F_0)=\left\{ \begin{pmatrix} p & c\\ B & \beta\end{pmatrix}\,\Big|\, c\in\R,\,B\in\mathbb{M}^{n\times n},\,\tr B=0, \, (p,\beta)\in\mathcal{S}
   \right\}
\end{equation}
for some set $\mathcal{S}=\mathcal{S}_{r_1,r_2}\subset\R^{n+n}$.
\end{lem}

\begin{proof}
Let $(c,B),\,(c',B')\in\R\times\mathbb{M}^{n\times n}$ be such that $\tr B=\tr B'=0$, and
define
\[
\mathcal{S}_{(c,B)}=\left\{(p,\beta)\in\R^{n+n} \,\Big|\, \begin{pmatrix} p & c\\ B & \beta\end{pmatrix}\in R(F_0) \right\},
\]
\[
\mathcal{S}_{(c',B')}=\left\{(p,\beta)\in\R^{n+n} \,\Big|\, \begin{pmatrix} p & c'\\ B' & \beta\end{pmatrix}\in R(F_0) \right\}.
\]
It is  sufficient to show that $\mathcal{S}_{(c,B)}=\mathcal{S}_{(c',B')}=:\mathcal{S}.$
Let $(p,\beta)\in\mathcal{S}_{(c,B)}$, that is, $\xi=\begin{pmatrix} p & c\\ B & \beta\end{pmatrix}\in R(F_0)$. Then $\xi_\pm:=\xi+t_\pm\eta\in F_\pm$ for some $t_-<0<t_+$ and $\mathrm{rank}\,\eta=1$. Observe that $\xi=\lambda\xi_++(1-\lambda)\xi_-$ with $\lambda=\frac{-t_-}{t_+-t_-}\in(0,1)$ and that
\[
\xi':=\begin{pmatrix} p & c'\\ B' & \beta\end{pmatrix}=\xi+\begin{pmatrix} 0 & \tilde{c}\\ \tilde B & 0\end{pmatrix}=\lambda\tilde{\xi}_++(1-\lambda)\tilde{\xi}_-
\]
where $\tilde{c}=c'-c$, $\tilde{B}=B'-B$, and $\tilde{\xi}_\pm=\xi_\pm+\begin{pmatrix} 0 & \tilde{c}\\ \tilde B & 0\end{pmatrix}$. Since  $\xi_\pm\in F_\pm$  and $\tr \tilde{B}=0$, we have  $\tilde{\xi}_\pm\in F_\pm$, and so $\xi'\in R(F_0)$. This implies $(p,\beta)\in\mathcal{S}_{(c',B')}$; hence $\mathcal{S}_{(c,B)}\subset \mathcal{S}_{(c',B')}$. Likewise, $\mathcal{S}_{(c',B')}\subset \mathcal{S}_{(c,B)}$; that is, $\mathcal{S}_{(c,B)}=\mathcal{S}_{(c',B')}$.
\end{proof}

\subsubsection{\bf Selection of approximate collinear rank-one connections for $R(F_0)$.}
We begin with a 2-dimensional description for the rank-one connections of diagonal components of matrices in $R(F_0)$ in a general form. The following lemma deals with \textbf{Cases  I} and  \textbf{II} in a parallel manner.

\begin{lem}\label{lem-2d-rank}
For all positive numbers $a,b,c$ with $b>a$, there exists a continuous function
\[
h_{1}(a,b,c;\cdot,\cdot,\cdot):I^1_{a,b,c}=[0,a)\times[0,\infty)\times[0,c)\to[0,\infty)\quad\mbox{\emph{(\textbf{Case I})}}
\]
\[
h_{2}(a,b,c;\cdot,\cdot,\cdot):I^2_{a,b,c}=[0,a)\times[0,b-a)\times[0,c)\to[0,\infty)\quad\mbox{\emph{(\textbf{Case II})}}
\]
with $h_i(a,b,c;0,0,0)=0$
satisfying the following:

Let $\delta_1,\delta_2$ and $\eta$ be any positive numbers with
\[
0<a-\delta_1<a<b<b+\delta_2,\quad 0<c-\eta<c,\quad\mbox{\emph{(\textbf{Case I})}}
\]
\[
0<a-\delta_1<a<b-\delta_2<b,\quad 0<c-\eta<c,\quad\mbox{\emph{(\textbf{Case II})}}
\]
and let $R_1\in[a-\delta_1,a]$, $R_2\in[b,b+\delta_2]$ \emph{(\textbf{Case I})}, $R_2\in[b-\delta_2,b]$ \emph{(\textbf{Case II})}, and $\tilde R_1,\tilde R_2\in[c-\eta,c]$.
Suppose $\theta\in[-\pi/2,\pi/2]$ and
\[
\begin{split}
\Big(\tilde R_1\big( & \cos(\frac{\pi}{2}+\theta),\sin(\frac{\pi}{2}+\theta)\big)-\tilde R_2\big(\cos(\frac{\pi}{2}-\theta),\sin(\frac{\pi}{2}-\theta)\big)\Big)\\
 & \cdot\Big(R_1\big(\cos(\frac{\pi}{2}+\theta),\sin(\frac{\pi}{2}+\theta)\big)-R_2\big(\cos(\frac{\pi}{2}-\theta),\sin(\frac{\pi}{2}-\theta)\big)\Big)=0.
\end{split}
\]
Then $-\frac{\pi}{2}<\theta<\frac{\pi}{2}$, $\tilde R_1\ge \tilde R_2$, and
\[
\max\Big\{\big|(0,a)-R_1\big(\cos(\frac{\pi}{2}+\theta),\sin(\frac{\pi}{2}+\theta)\big)\big|,\,\big|(0,b)-R_2\big(\cos(\frac{\pi}{2}-\theta),\sin(\frac{\pi}{2}-\theta)\big)\big|
\]
\[
\big|(0,c)-\tilde R_1\big(\cos(\frac{\pi}{2}+\theta),\sin(\frac{\pi}{2}+\theta)\big)\big|,\,\big|(0,c)-\tilde R_2\big(\cos(\frac{\pi}{2}-\theta),\sin(\frac{\pi}{2}-\theta)\big)\big|\Big\}
\]
\[
\le h_1(a,b,c;\delta_1,\delta_2,\eta),\quad \emph{(\textbf{Case I})}
\]
\[
\max\Big\{\big|(0,a)-R_1\big(\cos(\frac{\pi}{2}+\theta),\sin(\frac{\pi}{2}+\theta)\big)\big|,\,\big|(0,b)-R_2\big(\cos(\frac{\pi}{2}-\theta),\sin(\frac{\pi}{2}-\theta)\big)\big|
\]
\[
\big|(0,c)-\tilde R_1\big(\cos(\frac{\pi}{2}+\theta),\sin(\frac{\pi}{2}+\theta)\big)\big|,\,\big|(0,c)-\tilde R_2\big(\cos(\frac{\pi}{2}-\theta),\sin(\frac{\pi}{2}-\theta)\big)\big|\Big\}
\]
\[
\le h_2(a,b,c;\delta_1,\delta_2,\eta).\quad \emph{(\textbf{Case II})}
\]
\end{lem}
\begin{proof}
By assumption,
\[
0=(\tilde R_1(-\sin\theta,\cos\theta)-\tilde R_2(\sin\theta,\cos\theta))\cdot(R_1(-\sin\theta,\cos\theta)-R_2(\sin\theta,\cos\theta))
\]
\[
=(-(\tilde R_1+\tilde R_2)\sin\theta,(\tilde R_1-\tilde R_2)\cos\theta)\cdot(-(R_1+R_2)\sin\theta,(R_1-R_2)\cos\theta)
\]
\[
=(\tilde R_1+\tilde R_2)(R_1+R_2)\sin^2\theta+(\tilde R_1-\tilde R_2)(R_1-R_2)\cos^2\theta,
\]
that is,
\[
(R_2-R_1)(\tilde R_1-\tilde R_2)\cos^2\theta=(R_1+R_2)(\tilde R_1+\tilde R_2)\sin^2\theta;
\]
hence, $\theta\ne \pm\frac{\pi}{2}$, $\tilde R_1\ge \tilde R_2$, and
\[
\theta=\pm\tan^{-1}\left(\sqrt\frac{(R_2-R_1)(\tilde R_1-\tilde R_2)}{(R_1+R_2)(\tilde R_1+\tilde R_2)}\right).
\]
So
\[
|\theta|\le\tan^{-1}\left(\sqrt{\frac{(b-a+\delta_1+\delta_2)\eta}{2(a+b-\delta_1)(c-\eta)}}\right)=:g_1(a,b,c;\delta_1,\delta_2,\eta),\;\;\mbox{(\textbf{Case I})}
\]
\[
|\theta|\le\tan^{-1}\left(\sqrt{\frac{(b-a+\delta_1)\eta}{2(a+b-\delta_1-\delta_2)(c-\eta)}}\right)=:g_2(a,b,c;\delta_1,\delta_2,\eta).\;\;\mbox{(\textbf{Case II})}
\]
Note that the function $g_i(a,b,c;\cdot,\cdot,\cdot):I^i_{a,b,c}\to[0,\pi/2)$ is well-defined and continuous and that $g_i(a,b,c;\delta_1,\delta_2,\eta)=0$\, for all $(\delta_1,\delta_2,\eta)\in I^i_{a,b,c}$ with $\eta=0.$

Observe now that
\[
|(0,a)-R_1(\cos(\frac{\pi}{2}+\theta),\sin(\frac{\pi}{2}+\theta))|
\]
\[
\le\max\{|(0,a)-a(-\sin\theta,\cos\theta)|,
|(0,a)-(a-\delta_1)(-\sin\theta,\cos\theta)|\}
\]
\[
=\max\left\{\sqrt{a^2\sin^2\theta+a^2(1-\cos\theta)^2},
\sqrt{(a-\delta_1)^2\sin^2\theta+(a-(a-\delta_1)\cos\theta)^2}\right\}
\]
\[
=\max\left\{\sqrt{2}a\sqrt{1-\cos\theta},\sqrt{(a-\delta_1)^2+a^2-2a(a-\delta_1)\cos\theta}\right\}
\]
\[
\begin{split}
\le\max\Big\{\sqrt{2}a & \sqrt{1-\cos(g_i(a,b,c;\delta_1,\delta_2,\eta))},\\
& \sqrt{(a-\delta_1)^2+a^2-2a(a-\delta_1)\cos(g_i(a,b,c;\delta_1,\delta_2,\eta))}\Big\}
\end{split}
\]
\[
=:h_{i,1}(a,b,c;\delta_1,\delta_2,\eta),\quad (\textbf{i=1,\,2})
\]
\[
|(0,b)-R_2(\cos(\frac{\pi}{2}-\theta),\sin(\frac{\pi}{2}-\theta))|
\]
\[
\le\max\{|(0,b)-b(\sin\theta,\cos\theta)|,
|(0,b)-(b+\delta_2)(\sin\theta,\cos\theta)|\}
\]
\[
=\max\left\{\sqrt{b^2\sin^2\theta+b^2(1-\cos\theta)^2},
\sqrt{(b+\delta_2)^2\sin^2\theta+(b-(b+\delta_2)\cos\theta)^2}\right\}
\]
\[
=\max\left\{\sqrt{2}b\sqrt{1-\cos\theta},\sqrt{(b+\delta_2)^2+b^2-2b(b+\delta_2)\cos\theta}\right\}
\]
\[
\begin{split}
\le\max\Big\{\sqrt{2}b & \sqrt{1-\cos(g_1(a,b,c;\delta_1,\delta_2,\eta))},\\
& \sqrt{(b+\delta_2)^2+b^2-2b(b+\delta_2)\cos(g_1(a,b,c;\delta_1,\delta_2,\eta))}\Big\}
\end{split}
\]
\[
=:h_{1,2}(a,b,c;\delta_1,\delta_2,\eta),\quad\mbox{(\textbf{Case I})}
\]
\[
|(0,b)-R_2(\cos(\frac{\pi}{2}-\theta),\sin(\frac{\pi}{2}-\theta))|
\]
\[
\begin{split}
\le\max\Big\{\sqrt{2}b & \sqrt{1-\cos(g_2(a,b,c;\delta_1,\delta_2,\eta))},\\
& \sqrt{(b-\delta_2)^2+b^2-2b(b-\delta_2)\cos(g_2(a,b,c;\delta_1,\delta_2,\eta))}\Big\}
\end{split}
\]
\[
=:h_{2,2}(a,b,c;\delta_1,\delta_2,\eta),\quad\mbox{(\textbf{Case II})}
\]
\[
|(0,c)-\tilde R_1(\cos(\frac{\pi}{2}+\theta),\sin(\frac{\pi}{2}+\theta))|
\]
\[
\begin{split}
\le\max\Big\{\sqrt{2}c & \sqrt{1-\cos(g_i(a,b,c;\delta_1,\delta_2,\eta))},\\
& \sqrt{(c-\eta)^2+c^2-2c(c-\eta)\cos(g_i(a,b,c;\delta_1,\delta_2,\eta))}\Big\}
\end{split}
\]
\[
=:h_{i,3}(a,b,c;\delta_1,\delta_2,\eta),\quad (\textbf{i=1,\,2})
\]
\[
|(0,c)-\tilde R_2(\cos(\frac{\pi}{2}-\theta),\sin(\frac{\pi}{2}-\theta))|\le h_{i,3}(a,b,c;\delta_1,\delta_2,\eta). \quad (\textbf{i=1,\,2})
\]
Define $h_i(a,b,c;\delta_1,\delta_2,\eta)=\max_{1\le j\le 3} h_{i,j}(a,b,c;\delta_1,\delta_2,\eta)$ for $i=1,\,2$ corresponding to \textbf{Cases I} and \textbf{II}, respectively.
Then it is easy to see that the function $h_i(a,b,c;\cdot,\cdot,\cdot):I^i_{a,b,c}\to[0,\infty)$ is well-defined and satisfies the desired properties.
\end{proof}

We now apply the previous lemma to  choose \emph{approximate} collinear rank-one connections for the diagonal components of matrices in $R(F_0)$.

\begin{thm}\label{lem-inv}
Let $p_\pm\in\R^n$ satisfy
\[
s_-(r_1)<|p_-|<s_-(r_2)<s_+(r_2)<|p_+|<s_+(r_1),\;\;\mbox{\emph{(\textbf{Case I})}}
\]
\[
s_-(r_1)<|p_-|<s_-(r_2)<s_+(r_1)<|p_+|<s_+(r_2),\;\;\mbox{\emph{(\textbf{Case II})}}
\]
and
$(A(p_+)-A(p_-))\cdot(p_+-p_-)=0$. Then there exists a vector $\zeta^0\in\mathbb{S}^{n-1}$ such that, with $p^0_{\pm}=s_\pm(r_2)\zeta^0$,
$A(p^0_\pm)=r_2\zeta^0$, we have
\[
\max\{|p^0_--p_-|,|p^0_+-p_+|,|A(p^0_-)-A(p_-)|,|A(p^0_+)-A(p_+)|\}
\]
\[
\le h_1(s_-(r_2),s_+(r_2),r_2; s_-(r_2)-s_-(r_1),s_+(r_1)-s_+(r_2),r_2-r_1),\;\;\mbox{\emph{(\textbf{Case I})}}
\]
\[
\max\{|p^0_--p_-|,|p^0_+-p_+|,|A(p^0_-)-A(p_-)|,|A(p^0_+)-A(p_+)|\}
\]
\[
\le h_2(s_-(r_2),s_+(r_2),r_2; s_-(r_2)-s_-(r_1),s_+(r_2)-s_+(r_1),r_2-r_1),\;\;\mbox{\emph{(\textbf{Case II})}}
\]
where $h_1,\,h_2$ are the functions in Lemma \ref{lem-2d-rank}.
\end{thm}

\begin{proof}
Let $\Sigma_2$ denote the 2-dimensional linear subspace of $\R^n$ spanned by the two vectors $p_\pm$. (In the case that $p_\pm$ are collinear, we choose $\Sigma_2$ to be any 2-dimensional space containing $p_\pm$.)
Set
\[
\zeta^0=\frac{\frac{p_+}{|p_+|}+\frac{p_-}{|p_-|}}{\big|\frac{p_+}{|p_+|}+\frac{p_-}{|p_-|}\big|}\in\mathbb{S}^{n-1}\cap\Sigma_2.
\]
Since the vectors $p_\pm$, $A(p_\pm)$ and $\zeta^0$ all lie in $\Sigma_2$, we can recast the problem into the setting of the previous lemma via one of the two  linear isomorphisms of
$\Sigma_2$ onto $\R^2$ with correspondence $\zeta^0\leftrightarrow (0,1)\in\R^2.$ Then the results follow with the following choices in applying Lemma \ref{lem-2d-rank}: $a=s_-(r_2)$, $b=s_+(r_2)$, $c=r_2$,
$\delta_1=s_-(r_2)-s_-(r_1)$, $\delta_2=s_+(r_1)-s_+(r_2)$  (\textbf{Case I}), $\delta_2=s_+(r_2)-s_+(r_1)$  (\textbf{Case II}), $\eta=r_2-r_1$, $R_1=|p_-|$, $R_2=|p_+|$, $\tilde R_1=\sigma(|p_-|)$,
$\tilde R_2=\sigma(|p_+|)$, and $\theta\in[0,\pi/2]$ is the half of the angle between $p_+$ and $p_-$.
\end{proof}

\subsubsection{\bf Final characterization of $R(F_0)$.}
We are now ready to establish the  result concerning  the essential structures of $R(F_0)$. For this purpose, it suffices to stick only  to the diagonal components of matrices in $R(F_0)$.

\begin{thm}\label{thm:main2}
Let $0<r_2<\sigma(s_0)$ \emph{(\textbf{Case I})}, $\sigma(s_2)<r_2<\sigma(s_1)$ \emph{(\textbf{Case II})}. Then there exists a number $l_2=l_{r_2}\in(0,r_2)$ \emph{(\textbf{Case I})}, a number $l_2=l_{r_2}\in(\sigma(s_2),r_2)$ \emph{(\textbf{Case II})} such that for any $l_2<r_1<r_2$, the set $\mathcal{S}=\mathcal{S}_{r_1,r_2}\subset\R^{n+n}$ in \emph{(\ref{rough-1})} satisfies the following:
\begin{itemize}
\item[(i)] $\sup_{(p,\beta)\in\mathcal S}|p|\le s_+(r_1)$ \emph{(\textbf{Case I})}, $\sup_{(p,\beta)\in\mathcal S}|p|\le s_+(r_2)$ \emph{(\textbf{Case II})} and   $\,\sup_{(p,\beta)\in\mathcal S}|\beta|\le r_2$; hence $\mathcal{S}$ is bounded.

\item[(ii)] $\mathcal{S}$ is open.

\item[(iii)] For each $(p_0,\beta_0)\in\mathcal{S},$ there exist an open set $\mathcal{V}\subset\subset\mathcal{S}$ containing $(p_0,\beta_0)$ and $C^1$ functions $q:\bar{\mathcal{V}}\to\mathbb{S}^{n-1}$, $\gamma:\bar{\mathcal{V}}\to\R^n$, $t_\pm:\bar{\mathcal{V}}\to\R$ with $\gamma\cdot q=0$ and $t_-<0<t_+$  on $\bar{\mathcal{V}}$ such that for every $\xi=\begin{pmatrix} p & c\\ B & \beta\end{pmatrix}\in R(F_0)=R(F_{r_1,r_2}(0))$ with $(p,\beta)\in\bar{\mathcal{V}}$, we have
\[
\xi+t_\pm\eta\in F_\pm,
\]
where $t_\pm=t_\pm(p,\beta)$, $\eta=\begin{pmatrix} q(p,\beta) & b\\ \frac{1}{b}\gamma(p,\beta)\otimes q(p,\beta) & \gamma(p,\beta)\end{pmatrix}$, and $b\neq 0$ is arbitrary.
\end{itemize}
\end{thm}

\begin{proof}
Fix any $0<r_2<\sigma(s_0)$ (\textbf{Case I}), $\sigma(s_2)<r_2<\sigma(s_1)$ (\textbf{Case II}). For the moment, we let $r_1$ be any number in $(0,r_2)$ (\textbf{Case I}), in $(\sigma(s_2),r_2)$ (\textbf{Case II}) and prove (i). Then we choose later a lower bound $l_2=l_{r_2}\in(0,r_2)$ (\textbf{Case I}), $l_2=l_{r_2}\in(\sigma(s_2),r_2)$ (\textbf{Case II}) of $r_1$ for the validity of (ii) and (iii) above.

We divide the proof into several steps.

1. To show that (i) holds, choose any $(p,\beta)\in\mathcal S$. By Lemma \ref{lem-rough}, $\xi:=\begin{pmatrix} p & 0\\ O & \beta\end{pmatrix}\in R(F_0)=R(F_{r_1,r_2}(0))$, where $O$ is the $n\times n$ zero matrix. By the definition of $R(F_0)$, there exist two matrices $\xi_\pm=\begin{pmatrix} p_\pm & c_\pm\\ B_\pm & \sigma(|p_\pm|)\frac{p_\pm}{|p_\pm|}\end{pmatrix}\in F_\pm$ and a number $0<\lambda<1$ such that
$
\xi=\lambda\xi_++(1-\lambda)\xi_-.
$
So
\[
|p|=|\lambda p_++(1-\lambda)p_-|\le s_+(r_1),\quad\mbox{(\textbf{Case I})}
\]
\[
|p|=|\lambda p_++(1-\lambda)p_-|\le s_+(r_2),\quad\mbox{(\textbf{Case II})}
\]
\[
|\beta|=\left|\lambda\sigma(|p_+|)\frac{p_+}{|p_+|}+(1-\lambda)\sigma(|p_-|) \frac{p_-}{|p_-|}\right|\le r_2\, ;
\]
hence, $\sup_{(p,\beta)\in\mathcal S}|p|\le s_+(r_1)$ (\textbf{Case I}), $\sup_{(p,\beta)\in\mathcal S}|p|\le s_+(r_2)$ (\textbf{Case II}),  $\sup_{(p,\beta)\in\mathcal S}|\beta|\le r_2,$ and $\mathcal S$ is bounded. So (i) is proved.

2. We now turn to the remaining assertions that the set $\mathcal S=\mathcal S_{r_1,r_2}$ fulfills (ii) and (iii) for all $r_1<r_2$ sufficiently close to $r_2$.  In this step, we still assume $r_1$ is any fixed number in $(0,r_2)$ (\textbf{Case I}), in $(\sigma(s_2),r_2)$ (\textbf{Case II}).

Let $(p_0,\beta_0)\in\mathcal S$.
Since $\xi_0:=\begin{pmatrix} p_0 & 0\\ O & \beta_0\end{pmatrix}\in R(F_0)$, it follows from Lemma \ref{lem-form} that there exist numbers $s_0<0<t_0$ and vectors $q_0,\,\gamma_0\in\R^n$ with $|q_0|=1$, $\gamma_0\cdot q_0=0$ such that $\xi_0+s_0\eta_0\in F_-$ and $\xi_0+t_0\eta_0\in F_+$, where $\eta_0=\begin{pmatrix} q_0 & b\\ \frac{1}{b}q_0\otimes\gamma_0 & \gamma_0\end{pmatrix}$ and $b\neq 0$ is any fixed number. Let $q'_0=t_0q_0\ne 0$, $\gamma'_0=t_0\gamma_0$, and $s'_0=s_0/t_0<0$; then
\begin{equation}\label{thm:main2-1}
\begin{cases}
\gamma'_0\cdot q'_0=0,\quad s_-(r_1)<|p_0+s'_0q'_0|<s_-(r_2),\\
s_+(r_2)<|p_0+q'_0|<s_+(r_1),\quad\mbox{(\textbf{Case I})}\\
s_+(r_1)<|p_0+q'_0|<s_+(r_2),\quad\mbox{(\textbf{Case II})}\\
\sigma(|p_0+s'_0q'_0|)\frac{p_0+s'_0q'_0}{|p_0+s'_0q'_0|}=\beta_0+s'_0\gamma'_0,\\
\sigma(|p_0+q'_0|)\frac{p_0+q'_0}{|p_0+q'_0|}=\beta_0+\gamma'_0.\end{cases}
\end{equation}
Observe also
\begin{equation}\label{thm:main2-2}
\begin{split}
&t_0-s_0\ge |(p_0+t_0q_0)|-|(p_0+s_0q_0)|>s_+(r_2)-s_-(r_2),\quad\mbox{(\textbf{Case I})}\\
&t_0-s_0\ge |(p_0+t_0q_0)|-|(p_0+s_0q_0)|>s_+(r_1)-s_-(r_2).\quad\mbox{(\textbf{Case II})}
\end{split}
\end{equation}

Next, consider the function $F$ defined by
\begin{equation}
F(\gamma',q',s';p,\beta) = \begin{pmatrix}  \sigma(|p+s'q'|)\frac{p+s'q'}{|p+s'q'|}-\beta-s'\gamma'  \\
 \sigma(|p+q'|)\frac{p+q'}{|p+q'|}-\beta-\gamma' \\
  \gamma'\cdot q'\end{pmatrix} \in\R^{n+n+1}
\end{equation}
for all $\gamma',\,q',\,p,\,\beta\in\R^n$ and $s'\in\R$ with  $s_-(r_1)<|p+s'q'|<s_-(r_2)$, $s_+(r_2)<|p+q'|<s_+(r_1)$ (\textbf{Case I}), $s_+(r_1)<|p+q'|<s_+(r_2)$ (\textbf{Case II}). Then  $F$ is $C^1$ in the described open subset of $\R^{n+n+1+n+n}$, and the observation (\ref{thm:main2-1}) yields that
\[
F(\gamma'_0,q'_0,s'_0;p_0,\beta_0)=0.
\]

Suppose for the moment that the Jacobian matrix $D_{(\gamma',q',s')}F$ 
is invertible at the point $(\gamma'_0,q'_0,s'_0;p_0,\beta_0)$; then the Implicit Function Theorem implies the following: There exist a bounded domain $\tilde{\mathcal{V}}=\tilde{\mathcal{V}}_{(p_0,\beta_0)}\subset\R^{n+n}$ containing $(p_0,\beta_0)$ and $C^1$ functions $\tilde{q},\,\tilde\gamma\in\R^n$, $\tilde s\in\R$ of $(p,\beta)\in\tilde{\mathcal{V}}$ such that
\[
\tilde{\gamma}(p_0,\beta_0)=\gamma'_0,\;\; \tilde{q}(p_0,\beta_0)=q'_0,\;\; \tilde{s}(p_0,\beta_0)=s'_0
\]
and that
\[
\tilde s(p,\beta)<0,\;\; s_-(r_1)<|p+\tilde s(p,\beta)\tilde q(p,\beta)|<s_-(r_2),
\]
\[
s_+(r_2)<|p+\tilde q(p,\beta)|<s_+(r_1),\quad\mbox{(\textbf{Case I})}
\]
\[
s_+(r_1)<|p+\tilde q(p,\beta)|<s_+(r_2),\quad\mbox{(\textbf{Case II})}
\]
\[
F(\tilde\gamma(p,\beta),\tilde q(p,\beta),\tilde s(p,\beta);p,\beta)=0\quad \forall(p,\beta)\in\tilde{\mathcal{V}}.
\]
Define functions
\[
\gamma=\frac{\tilde\gamma}{|\tilde q|}, \;\; q=\frac{\tilde q}{|\tilde q|},\;\; t_-=\tilde s |\tilde q|,\;\; t_+=|\tilde q| \quad\mbox{in $\tilde{\mathcal{V}}$};
\]
then
\[
s_-(r_1)<|p+ t_-q|<s_-(r_2),
\]
\[
s_+(r_2)<|p+t_+ q|<s_+(r_1),\quad\mbox{(\textbf{Case I})}
\]
\[
s_+(r_1)<|p+t_+ q|<s_+(r_2),\quad\mbox{(\textbf{Case II})}
\]
\[
\sigma(|p+t_\pm q|)\frac{p+t_\pm q}{|p+t_\pm q|}=\beta+t_\pm \gamma,\;\; |q|=1,\;\; \gamma\cdot q=0,\;\; t_-<0<t_+,
\]
where $(p,\beta)\in\tilde{\mathcal{V}}$, $\gamma=\gamma(p,\beta)$, $q=q(p,\beta),$  and $t_\pm=t_\pm(p,\beta)$.

Let $(p,\beta)\in\tilde{\mathcal V}$, $B\in\mathbb{M}^{n\times n}$, $\tr B=0$, $b,c\in\R$, $b\neq 0$, $q=q(p,\beta)$, $\gamma=\gamma(p,\beta)$, $t_\pm=t_\pm(p,\beta)$, $\xi=\begin{pmatrix} p & c\\ B & \beta\end{pmatrix}$, and $\eta=\begin{pmatrix} q & b\\ \frac{1}{b}\gamma\otimes q & \gamma\end{pmatrix}$. Then $\xi_\pm:=\xi+t_\pm\eta\in F_\pm$. By the definition of $R(F_0)$, $\xi\in(\xi_-,\xi_+)\subset R(F_0)$. By Lemma \ref{lem-rough}, we thus have $(p,\beta)\in\mathcal{S}$; hence $\tilde{\mathcal{V}}\subset\mathcal S$. This proves that $\mathcal S$ is open. Choosing any open set $\mathcal V\subset\subset\tilde{\mathcal{V}}$ with $(p_0,\beta_0)\in\mathcal{V}$,  the assertion  (iii) will hold.

3. In this step, we continue Step 2 to deduce an equivalent condition for the invertibility of the Jacobian matrix $D_{(\gamma',q',s')}F$  at  $(\gamma'_0,q'_0,s'_0;p_0,\beta_0)$. By direct computation,
\[
D_{(\gamma',q',s')}F= \begin{pmatrix} -s'I_n & M_{s'} & \omega^-_{s'} \\ -I_n & M_1 & 0  \\ q' & \gamma' & 0  \end{pmatrix}\in\mathbb{M}^{(n+n+1)\times(n+n+1)},
\]
where $I_n$ is the $n\times n$ identity matrix,
\[
M_{s'}=s'\left (\sigma'(|p+s'q'|)-\frac{\sigma(|p+s'q'|)}{|p+s'q'|}\right )\frac{p+s'q'}{|p+s'q'|}\otimes \frac{p+s'q'}{|p+s'q'|}+s'\frac{\sigma(|p+s'q'|)}{|p+s'q'|}I_n,
\]
\[
\omega^\pm_{s'}=\left (\sigma'(|p+s'q'|)-\frac{\sigma(|p+s'q'|)}{|p+s'q'|}\right )\left (\frac{p+s'q'}{|p+s'q'|}\cdot q'\right )\frac{p+s'q'}{|p+s'q'|}+\frac{\sigma(|p+s'q'|)}{|p+s'q'|}q'\pm\gamma'.
\]
Here the prime only in $\sigma'$ denotes the derivative.
For notational simplicity,  we write $(\gamma',q',s';p,\beta)=(\gamma'_0,q'_0,s'_0;p_0,\beta_0)$.
Applying suitable elementary row operations, as $s'<0$, we have
\[
D_{(\gamma',q',s')}F \,\to\, \begin{pmatrix} -s'I_n & M_{s'} & \omega^-_{s'} \\ O & M_1-\frac{1}{s'}M_{s'} & -\frac{1}{s'}\omega^-_{s'} \\ 0 & \gamma'+\frac{q'_1}{s'}(M_{s'})^1+\cdots+\frac{q'_n}{s'}(M_{s'})^n & \frac{1}{s'}q'\cdot\omega^-_{s'}\end{pmatrix}
\]
\[
\to\, \begin{pmatrix} -s'I_n & M_{s'} & \omega^-_{s'} \\ O & s'M_1-M_{s'} & -\omega^-_{s'} \\ 0 & s'\gamma'+q'_1(M_{s'})^1+\cdots+q'_n(M_{s'})^n & q'\cdot\omega^-_{s'}\end{pmatrix},
\]
where $O$ is the $n\times n$ zero matrix, and $(M_{s'})^i$ is the $i$th row of $M_{s'}$. Since $|q'|=t_0$, $\gamma'\cdot q'=0$, and $s_-(r_1)<|p+s'q'|<s_-(r_2)$, we have 
\[
q'\cdot\omega^-_{s'}=\left (\sigma'(|p+s'q'|)-\frac{\sigma(|p+s'q'|)}{|p+s'q'|}\right ) \left (\frac{p+s'q'}{|p+s'q'|}\cdot q'\right )^2+\frac{\sigma(|p+s'q'|)}{|p+s'q'|}t_0^2
\]
\[
=t_0^2\left(\cos^2\theta'\sigma'(|p+s'q'|)+(1-\cos^2\theta')\frac{\sigma(|p+s'q'|)}{|p+s'q'|}\right)>0,
\]
where $\theta'\in[0,\pi]$ is the angle between $p+s'q'$ and $q'$. Observe here that the forward part of $\sigma$ in the definition of $F_-$ becomes essential to guarantee that $\sigma'(|p+s'q'|)>0$.
After some elementary column operations to the last matrix from the above row operations, we obtain
\[
D_{(\gamma',q',s')}F \,\to\,  \begin{pmatrix} -s'I_n & M_{s'}-N_{s'} & \omega^-_{s'} \\ O & s'M_1-M_{s'}+N_{s'} & -\omega^-_{s'} \\ 0 & 0 & q'\cdot\omega^-_{s'} \end{pmatrix},
\]
where the $j$th column of  $N_{s'}\in\mathbb{M}^{n\times n}$ is
$\frac{s'\gamma'_j+q'\cdot(M_{s'})_j}{q'\cdot\omega^-_{s'}}\omega^-_{s'}$. So $D_{(\gamma',q',s')}F$ is invertible if and only if the $n\times n$ matrix  $M_1-\frac{1}{s'}M_{s'}+\frac{1}{s'}N_{s'}$ is invertible. We compute
\[
M_1-\frac{1}{s'}M_{s'}+\frac{1}{s'}N_{s'} = \left(\sigma'(|p+q'|)-\frac{\sigma(|p+q'|)}{|p+q'|}\right)\frac{p+q'}{|p+q'|}\otimes \frac{p+q'}{|p+q'|}
\]
\[
+\frac{\sigma(|p+q'|)}{|p+q'|}I_n-\left(\sigma'(|p+s'q'|)-\frac{\sigma(|p+s'q'|)}{|p+s'q'|}\right )\frac{p+s'q'}{|p+s'q'|}\otimes \frac{p+s'q'}{|p+s'q'|}
-\frac{\sigma(|p+s'q'|)}{|p+s'q'|}I_n
\]
\[
+\frac{1}{q'\cdot\omega^-_{s'}}\omega^-_{s'}\otimes \left [\gamma'
+\left(\sigma'(|p+s'q'|)-\frac{\sigma(|p+s'q'|)}{|p+s'q'|}\right)\left(\frac{p +s'q'}{|p+s'q'|}\cdot q'\right)\frac{p+s'q'}{|p+s'q'|}+\frac{\sigma(|p+s'q'|)}{|p+s'q'|}q'\right]
\]
\[
\begin{split}
=(a_1-a_{s'})I_n & +(b_1-a_1)\frac{p+q'}{|p+q'|}\otimes\frac{p+q'}{|p+q'|}\\
& -(b_{s'}-a_{s'})
\frac{p+s'q'}{|p+s'q'|}\otimes\frac{p+s'q'}{|p+s'q'|}
+
\frac{1}{q'\cdot\omega^-_{s'}}\omega^-_{s'}\otimes\omega^+_{s'},
\end{split}
\]
and set (with an assumption $a_1\neq a_{s'}$)
\[
\begin{split} B:=&\frac{1}{a_1-a_{s'}}(M_1-\frac{1}{s'}M_{s'}+\frac{1}{s'}N_{s'})\\
= &
I_n+\frac{b_1-a_1}{a_1-a_{s'}}\frac{p+q'}{|p+q'|}\otimes\frac{p+q'}{|p+q'|}\\
 &-\frac{b_{s'}-a_{s'}}{a_1-a_{s'}}
\frac{p+s'q'}{|p+s'q'|}\otimes\frac{p+s'q'}{|p+s'q'|} +
\frac{1}{(a_1-a_{s'})q'\cdot\omega^-_{s'}}\omega^-_{s'}\otimes\omega^+_{s'},\end{split}
\]
where
\[
a_{s'}=\frac{\sigma(|p+s'q'|)}{|p+s'q'|},\quad b_{s'}=\sigma'(|p+s'q'|);
\]
 then
$D_{(\gamma',q',s')}F$ is invertible if and only if the matrix
$B\in\mathbb{M}^{n\times n}$ is invertible.

4. To close the arguments in Step 2 and thus to finish the proof, we choose a suitable $l_2=l_{r_2}\in(0,r_2)$ (\textbf{Case I}), $l_2=l_{r_2}\in(\sigma(s_2),r_2)$ (\textbf{Case II}), depending on $r_2$, in such a way that for any $r_1\in(l_2,r_2)$, the matrix $B$, determined through Steps 2 and 3 for any given $(p_0,\beta_0)\in\mathcal S=\mathcal S_{r_1,r_2}$, is invertible.

First, by Hypothesis (PM) or (H), $\tilde r_2\in(0,r_2)$ (\textbf{Case I}), $\tilde r_2\in(\sigma(s_2),r_2)$ (\textbf{Case II}) can be chosen close enough to $r_2$ so that
\[
\frac{\sigma(k)}{k}<\frac{\sigma(l)}{l}\quad\forall  l\in[s_-(\tilde r_2),s_-(r_2)],
\]
\[
\forall k\in[s_+(r_2),s_+(\tilde r_2)]\;\;\mbox{(\textbf{Case I})},\quad \forall k\in[s_+(\tilde r_2),s_+(r_2)]\;\;\mbox{(\textbf{Case II})}.
\]
Then define a real-valued continuous function (to express the determinant of the matrix $B$ from Step 3)
\[
\mbox{DET}(u,v,q,\gamma)=\det\Big(I_n+\frac{\sigma'(|u|)-\frac{\sigma(|u|)}{|u|}}
{\frac{\sigma(|u|)}{|u|}-\frac{\sigma(|v|)}{|v|}}\frac{u}{|u|}\otimes\frac{u}{|u|}
-\frac{\sigma'(|v|)-\frac{\sigma(|v|)}{|v|}}
{\frac{\sigma(|u|)}{|u|}-\frac{\sigma(|v|)}{|v|}}\frac{v}{|v|}\otimes\frac{v}{|v|}
\]
\[
+\frac{1}{(\frac{\sigma(|u|)}{|u|}-\frac{\sigma(|v|)}{|v|})((\sigma'(|v|)-\frac{\sigma(|v|)}{|v|})
(\frac{v}{|v|}\cdot q)^2 + \frac{\sigma(|v|)}{|v|})}
\big((\sigma'(|v|)-\frac{\sigma(|v|)}{|v|})(\frac{v}{|v|}\cdot q)\frac{v}{|v|}
\]
\[
+\frac{\sigma(|v|)}{|v|}q-\gamma \big)\otimes \big((\sigma'(|v|)-\frac{\sigma(|v|)}{|v|})(\frac{v}{|v|}\cdot q)\frac{v}{|v|}
+\frac{\sigma(|v|)}{|v|}q+\gamma \big)\Big)
\]
on the compact set $\mathcal M$ of points $(u,v,q,\gamma)\in\R^n\times\R^n\times\mathbb{S}^{n-1}\times\R^n$ with
\[
|v|\in[s_-(\tilde r_2),s_-(r_2)],\;\;|\gamma|\le 1,
\]
\[
|u|\in[s_+(r_2),s_+(\tilde r_2)]\;\;\mbox{(\textbf{Case I})},\;\;|u|\in[s_+(\tilde r_2),s_+(r_2)]\;\;\mbox{(\textbf{Case II})}.
\]
With $\bar k=s_+(r_2)$ and $\bar l=s_-(r_2)$, for each $q\in\mathbb{S}^{n-1},$
\[
\mbox{DET}(\bar k q,\bar lq,q,0)=\det\Big(I_n+\frac{\sigma'(\bar k)-\frac{\sigma(\bar k)}{\bar k}+\frac{\sigma(\bar l)}{\bar l}}{\frac{\sigma(\bar k)}{\bar k}-\frac{\sigma(\bar l)}{\bar l}}q\otimes q \Big)\ne 0,
\]
since $\sigma'(\bar k)\neq0$ and hence the fraction in front of $q\otimes q$ is different from $-1$. So
\[
d:=\min_{q\in\mathbb{S}^{n-1}}|\mbox{DET}(\bar kq,\bar lq,q,0)|>0.
\]

Next, choose a number $\delta>0$ such that for all $(u,v,q,\gamma),(\tilde u,\tilde v,\tilde q,\tilde \gamma)\in\mathcal M$ with
$|u-\tilde u|,\,|v-\tilde v|,\,|q-\tilde q| ,\,|\gamma-\tilde\gamma|<\delta$, we have
\begin{equation}\label{thm:main2-3}
|\mbox{DET}(u,v,q,\gamma)-\mbox{DET}(\tilde u,\tilde v,\tilde q,\tilde \gamma)|<d/2.
\end{equation}
Let $l_2\in(\tilde r_2,r_2)$ be sufficiently close to $r_2$ so that for all $r_1\in(l_2,r_2)$,
\[
h_1(s_-(r_2),s_+(r_2),r_2; s_-(r_2)-s_-(r_1),s_+(r_1)-s_+(r_2),r_2-r_1)<\tau,\;\mbox{(\textbf{Case I})}
\]
\[
h_2(s_-(r_2),s_+(r_2),r_2; s_-(r_2)-s_-(r_1),s_+(r_2)-s_+(r_1),r_2-r_1)<\tau,\;\mbox{(\textbf{Case II})}
\]
where $h_i$'s are the functions in Theorem \ref{lem-inv}, and
\[
\tau:=\min\{\delta,\delta(s_+(r_2)-s_-(r_2))/4\},\;\;\mbox{(\textbf{Case I})}
\]
\[
\tau:=\min\{\delta,\delta(s_2-s_1)/4\}.\;\;\mbox{(\textbf{Case II})}
\]

Now, fix  any $r_1\in(l_2,r_2)$, and let $B$ be the $n\times n$ matrix determined through Steps 2 and 3 in terms of any given $(p_0,\beta_0)\in\mathcal{S}=\mathcal{S}_{r_1,r_2}$. Let $p_+=p_0+t_0 q_0$ and $p_-=p_0+s_0 q_0$ from Step 2; then $p_\pm$ and $A(p_\pm)$ fulfill the conditions in Theorem \ref{lem-inv}. So this theorem implies that there exists a vector $\zeta^0\in\mathbb{S}^{n-1}$ such that
\[
\max\{|p^0_--p_-|,|p^0_+-p_+|,|A(p^0_-)-A(p_-)|,|A(p^0_+)-A(p_+)|\}<\tau,
\]
where $p^0_+=\bar{k}\zeta^0$, $p^0_-=\bar{l}\zeta^0$, and $A(p_\pm^0)=r_2\zeta^0$. Using (\ref{thm:main2-1}) and (\ref{thm:main2-2}),
\[
|p_+ -\bar{k}\zeta^0|<\delta,\;\;|p_- -\bar{l}\zeta^0|<\delta,
\]
\[
|q_0-\zeta^0|=|\frac{p_+-p_-}{t_0-s_0}-\zeta^0|\le\frac{|(p_+-p_-)-(\bar{k}-\bar{l})\zeta^0|+|(\bar{k}-\bar{l})-(t_0-s_0)|}{t_0-s_0}
\]
\[
\le \frac{2\tau+||p_+^0-p_-^0|-|p_+-p_-||}{t_0-s_0}<\frac{4\tau}{t_0-s_0}<\delta,
\]
\[
|\gamma_0|=|\frac{A(p_+)-A(p_-)}{t_0-s_0}|\le\frac{|A(p_+)-A(p_+^0)|+|A(p_-^0)-A(p_-)|}{t_0-s_0}<\delta.
\]
Since $\det(B)=\mbox{DET}(p_+,p_-,q_0,\gamma_0)$ and $|\mbox{DET}(\bar{k}\zeta^0,\bar{l}\zeta^0,\zeta^0,0)|\ge d$, it  follows from (\ref{thm:main2-3}) that
\[
|\det(B)|>d/2>0.
\]

The proof is now complete.
\end{proof}

\subsection{Relaxation of  $\nabla \omega(z)\in K_0$}

The following result is important for the convex integration with linear constraint; the function $\varphi$ determined here plays a similar role as the tile function $g$ used in \cite{Zh, Zh1}.  For a more general case, see \cite[Lemma 2.1]{Po}.

\begin{lem}\label{approx-lem} Let $\lambda_1,\lambda_2>0$ and $\eta_1=-\lambda_1\eta, \; \eta_2=\lambda_2\eta$ with
\[
\eta=\begin{pmatrix} q &  b\\\frac{1}{b}\gamma\otimes q & \gamma\end{pmatrix},\quad |q|=1,\; \gamma\cdot q=0,\;   b\ne 0.
\]
Let $G\subset \R^{n+1}$ be a bounded domain. Then for each $\epsilon>0$, there exists a function $\omega=(\varphi,\psi)\in C_c^{\infty}(\R^{n+1};\R^{1+n})$ with $\mathrm{supp}(\omega)\subset\subset G$  that satisfies the following properties:

(a) \; $\dv \psi =0$  in $G$,

(b) \; $|\{z\in G\;|\; \nabla\omega(z)\notin \{\eta_1,\;\eta_2\}\}|<\epsilon,$

(c) \; $\dist(\nabla \omega(z),[\eta_1,\eta_2])<\epsilon$ for all $z\in G,$

(d) \; $\|\omega\|_{L^\infty(G)}<\epsilon,$

(e) \; $\int_{\R^n}\varphi(x,t)\,dx=0$ for all $t\in\R$.

\end{lem}

\begin{proof} The proof follows a simplified version of \cite[Lemma 2.1]{Po}.

1. We define a map $\mathcal P\colon C^1(\R^{n+1})\to C^0(\R^{n+1};\R^{1+n})$ by setting $\mathcal P(h)=(u,v)$, where, for $h(x,t)\in C^1(\R^{n+1})$,
\[
u(x,t)= q\cdot Dh(x,t), \quad
 v(x,t)=\frac{1}{b}(\gamma\otimes q-q\otimes \gamma)Dh(x,t).
\]
We easily see  that $\mathcal P(h)=(u,v)\in C_c^\infty(\R^{n+1};\R^{1+n})$,  $\mathrm{supp}(\mathcal{P}(h))\subset\mathrm{supp}(h)$, $\dv v\equiv 0$,  and $\int_{\R^n}u(x,t)\,dx=0$ for all $t\in\R$,  for all $h\in C_c^\infty(\R^{n+1}).$ For $h(x,t)=f(q\cdot x+bt)$ with $f\in C^\infty(\R)$, $w=(u,v)=\mathcal P(h)$ is given by $u(x,t)=f'(q\cdot x+bt)$ and $v(x,t)=f'(q\cdot x+bt)\frac{\gamma}{b}$, and hence $\nabla w(x,t)=f''(q\cdot x+bt)\eta.$
We also note that $\mathcal P(gh)=g\mathcal P(h)+h\mathcal P(g)$ and hence
\begin{equation}\label{bi-form}
\nabla \mathcal P(gh)= g\nabla \mathcal P(h) + h\nabla \mathcal P(g) + \mathcal B(\nabla g,\nabla h) \quad \forall\; g,\;h\in C^\infty(\R^{n+1}),
\end{equation}
where $\mathcal B(\nabla g,\nabla h)$ is a bilinear map of $\nabla g$ and $\nabla h$; so $|\mathcal B(\nabla h,\nabla g)|\le C|\nabla h||\nabla g|$ for some constant $C>0.$

2. Let $G_\epsilon\subset\subset G$ be a smooth sub-domain such that $|G\setminus G_\epsilon|<\epsilon/2,$ and let $\rho_\epsilon\in C^\infty_c(G)$ be a cut-off function satisfying $0\le \rho_\epsilon\le 1$ in $G$, $\rho_\epsilon =1$ on $G_\epsilon$.
As $G$ is bounded, $G \subset\{(x,t)\;|\; k<q\cdot x+bt<l\}$ for some numbers $k<l.$ For each $\tau>0$, we can find a function $f_\tau\in C^\infty_c(k,l)$ satisfying
\[
-\lambda_1\le f_\tau'' \le \lambda_2, \;  |\{s\in (k,l)\;|\;f_\tau''(s)\notin\{-\lambda_1,\;  \lambda_2\}\}|<\tau, \; \|f_\tau\|_{L^\infty}+ \|f_\tau'\|_{L^\infty} <\tau.
\]

3. Define $\omega=(\varphi,\psi)=\mathcal P(\rho_\epsilon(x,t) h_\tau(x,t)),$ where $h_\tau(x,t)= f_\tau(q\cdot x+ bt).$ Then  $
\|h_\tau\|_{C^1}\le C\|f_\tau\|_{C^1}\le C\tau$,  $\omega\in C_c^{\infty}(\R^{n+1};\R^{1+n})$,  $\mathrm{supp}(\omega)\subset\mathrm{supp}(\rho_\epsilon)\subset\subset G$, and (a) and (e) are satisfied. Note that
\[
|\omega|\le |\rho_\epsilon||\mathcal P(h_\tau)|+|h_\tau||\mathcal P(\rho_\epsilon)|\le C_\epsilon \tau,
\]
where $C_\epsilon>0$ is a constant depending on $\|\rho_\epsilon\|_{C^1(G)}$. So we can choose a $\tau_1>0$ so small that (d) is satisfied for all $0<\tau<\tau_1.$ Note also that
\[
\{z\in G\,|\; \nabla\omega(z)\notin \{\eta_1,\eta_2\}\}\subseteq (G\setminus G_\epsilon)\cup  \{z\in G_\epsilon \,|\; f''_\tau(q\cdot x+bt)\notin \{-\lambda_1,\; \lambda_2\}\}.
\]
Since $| \{z\in G_\epsilon \,|\; f''_\tau(q\cdot x+bt)\notin \{-\lambda_1,\; \lambda_2\}| \le N |\{s\in (k,l)\;|\;f_\tau''(s)\notin\{-\lambda_1,\;  \lambda_2\}\}|$ for some constant $N>0$ depending only on set $G$, there exists a $\tau_2>0$ such that
\[
|\{z\in G\,|\; \nabla\omega(z)\notin \{\eta_1,\eta_2\}\}|\le \frac{\epsilon}{2}+ N \tau<\epsilon
\]
for all $0<\tau<\tau_2$. Therefore, (b) is satisfied. Finally, note that
\[
\rho_\epsilon \nabla \mathcal P(h_\tau(x,t))=\rho_\epsilon f''_\tau(q\cdot x+bt) \eta\in [\eta_1,\eta_2] \quad \mbox{in $G$}
\]
and, by (\ref{bi-form}), for all $z=(x,t)\in G$,
\[
|\nabla \omega(z) -\rho_\epsilon \nabla \mathcal P(h_\tau(x,t))|\le |h_\tau| |\nabla \mathcal P(\rho_\epsilon)|+|\mathcal B(\nabla h_\tau,\nabla \rho_\epsilon)|\le C_\epsilon' \tau<\epsilon
\]
for all $0<\tau<\tau_3$, where $C'_\epsilon>0$ is a constant depending on  $\|\rho_\epsilon\|_{C^2(G)}$, and $\tau_3>0$ is another constant. Hence (c) is satisfied. Taking $0<\tau<\min\{\tau_1,\tau_2,\tau_3\}$, the proof is complete.
\end{proof}

We now state  the relaxation theorem for  homogeneous differential inclusion $\nabla \omega(z)\in K_0$  in a form that is more suitable for later use; we  restrict the inclusion
 only to  the diagonal components.

\begin{thm}\label{main-lemma}  Let $0<r_2<\sigma(s_0)$ \emph{(\textbf{Case I})}, $\sigma(s_2)<r_2<\sigma(s_1)$ \emph{(\textbf{Case II})}, and let $l_2=l_{r_2}\in(0,r_2)$  \emph{(\textbf{Case I})}, $l_2=l_{r_2}\in(\sigma(s_2),r_2)$  \emph{(\textbf{Case II})} be some number, depending on $r_2$, from Theorem \ref{thm:main2}. Let $l_2<r_1<r_2$, and let $\mathcal K$ be a compact subset of $\mathcal S=\mathcal S_{r_1,r_2}.$ Let $\tilde Q\times \tilde I$ be a  box  in $\R^{n+1}.$  Then, given any $\epsilon>0$, there exists a $\delta>0$ such that for each box $Q\times I\subset \tilde Q\times \tilde I$, point  $(p,\beta)\in\mathcal K$,  and  number $\rho>0$ sufficiently small, there exists a function $\omega=(\varphi,\psi)\in C^\infty_{c} (Q\times I;\R^{1+n})$  satisfying the following properties:

(a) \; $\dv \psi =0$  in $Q\times I$,

(b) \; $(p'+D\varphi(z), \beta'+ \psi_t(z))\in \mathcal{S}$ for all $z \in Q\times I$ and $|(p',\beta')-(p,\beta)|\leq \delta,$

(c) \; $\|\omega\|_{L^\infty(Q\times I)}<\rho,$

(d) \; $\int_{Q\times I} |\beta+\psi_t(z)-A(p+D\varphi(z))|dz<\epsilon {|Q\times I|}/{|\tilde Q\times \tilde I|},$

(e) \; $\int_{Q\times I} \mathrm{dist}((p+D\varphi(z),\beta+\psi_t(z)),\mathcal A)dz<\epsilon {|Q\times I|}/{|\tilde Q\times \tilde I|},$

(f) \; $\int_{Q} \varphi(x,t)dx=0$  for all $t\in I,$

(g) \; $\|\varphi_t\|_{L^\infty(Q\times I)}<\rho,$

\noindent where $\mathcal A=\mathcal A_{r_1,r_2}\subset \R^{n+n}$ is the set defined by
\[
\mathcal A = \big\{(p',A(p')) \,|\, |p'|\in [s_-(r_1),s_-(r_2)]\cup[s_+(r_2),s_+(r_1)] \big\},\quad (\textbf{\emph{Case I}})
\]
\[
\mathcal A = \big\{(p',A(p')) \,|\, |p'|\in [s_-(r_1),s_-(r_2)]\cup[s_+(r_1),s_+(r_2)] \big\}.\quad (\textbf{\emph{Case II}})
\]
\end{thm}

\begin{proof}
By Theorem \ref{thm:main2}, there exist finitely many open balls $\mathcal{B}_1,\cdots,\mathcal{B}_N\subset\subset\mathcal S$ covering $\mathcal K$ and $C^1$ functions $q_i:\bar{\mathcal{B}}_i\to \mathbb{S}^{n-1}$, $\gamma_i:\bar{\mathcal{B}}_i\to \R^{n}$, $t_{i,\pm}:\bar{\mathcal{B}}_i\to \R$ $(1\le i\le N)$ with $\gamma_i\cdot q_i=0$ and $t_{i,-}<0<t_{i,+}$ on $\bar{\mathcal B}_i$  such that for each $\xi=\begin{pmatrix} p & c \\ B & \beta\end{pmatrix}\in R(F_0)$ with $(p,\beta)\in\bar{\mathcal{B}}_i$, we have
\[
\xi+t_{i,\pm} \eta_i\in F_\pm,
\]
where $t_{i,\pm}=t_{i,\pm}(p,\beta)$, $\eta_i=\begin{pmatrix} q_i(p,\beta) & b \\ \frac{1}{b}\gamma_i(p,\beta)\otimes q_i(p,\beta) & \gamma_i(p,\beta)\end{pmatrix}$, and $b\neq 0$ is arbitrary.

Let $1\le i\le N$. We write $\xi_i=\xi_i(p,\beta)=\begin{pmatrix} p & 0\\O & \beta\end{pmatrix}\in R(F_0)$ for $(p,\beta)\in\bar{\mathcal B}_{i}\subset\mathcal S$, where $O$ is  the $n\times  n$ zero matrix. We omit the dependence on $(p,\beta)\in\bar{\mathcal B}_{i}$ in the following whenever it is clear from the context. Given any $\rho>0$, we choose a constant $b_i$ with
\[
0<b_i<\min_{\bar{\mathcal{B}}_{i}} \frac{\rho}{t_{i,+}-t_{i,-}}.
\]
With this choice of $b=b_i$, let $\eta_{i}$ be defined on $\bar{\mathcal{B}}_i$ as above. Then
\[
\xi_{i,\pm}=\begin{pmatrix} p_{i,\pm} & c_{i,\pm}\\B_{i,\pm} & \beta_{i,\pm}\end{pmatrix}:=\xi_i+t_{i,\pm}\eta_{i}\in F_\pm,
\]
\[
\xi_i=\lambda_i \xi_{i,+} + (1-\lambda_i)\xi_{i,-},\quad\lambda_i=\frac{-t_{i,-}}{t_{i,+}-t_{i,-}}\in(0,1)\quad \textrm{on}\;\;\bar{\mathcal{B}}_{i}.
\]
By the definition of $R(F_0)$,  on $\bar{\mathcal B}_i$,  both $\xi_{i,-}^\tau=\tau \xi_{i,+} + (1-\tau)\xi_{i,-}$ and $\xi_{i,+}^\tau=(1-\tau) \xi_{i,+} +  \tau \xi_{i,-}$ belong to $R(F_0)$ for all $\tau\in (0,1)$.  Let $0<\tau<\min_{1\le j\le N}\min_{\bar{\mathcal B}_j}\min\{\lambda_j,1-\lambda_j\}\le \frac12$ be a small number to be selected later.   Let $\lambda'_i=\frac{\lambda_i-\tau}{1-2\tau}$ on $\bar{\mathcal B}_i$. Then $\lambda'_i\in (0,1)$,
$\xi_i=\lambda'_i\xi_{i,+}^\tau+(1-\lambda'_i)\xi_{i,-}^\tau$ on $\bar{\mathcal B}_i$.
Moreover, on $\bar{\mathcal B}_i$, $\xi_{i,+}^\tau-\xi_{i,-}^\tau=(1-2\tau)(\xi_{i,+} -\xi_{i,-} )$ is rank-one, $[\xi_{i,-}^\tau,\xi_{i,+}^\tau]\subset(\xi_{i,-},\xi_{i,+})\subset R(F_0)$, and
\[
c\tau\leq|\xi_{i,+}^\tau-\xi_{i,+}|=|\xi_{i,-}^\tau-\xi_{i,-}|=\tau |\xi_{i,+} -\xi_{i,-}|=\tau(t_{i,+}-t_{i,-})|\eta_i|\leq C\tau,
\]
where $C=\max_{1\le j\le N}\max_{\bar{\mathcal B}_j}(t_{j,+}-t_{j,-})|\eta_j|\geq\min_{1\le j\le N}\min_{\bar{\mathcal B}_j}(t_{j,+}-t_{j,-})|\eta_j|$ $=c>0.$ By continuity, $H_\tau=\bigcup_{(p,\beta)\in\bar{\mathcal{B}}_j,1\le j\le N}[\xi_{j,-}^\tau(p,\beta),\xi_{j,+}^\tau(p,\beta)]$ is a compact subset of $R(F_0)$, where $R(F_0)$ is open in the space
\[
\Sigma_0=\left\{\begin{pmatrix} p & c\\B & \beta\end{pmatrix}\;\Big|\; \mathrm{tr}B=0\right\},
\]
by Lemma \ref{lem-rough} and Theorem \ref{thm:main2}.
So $d_\tau=\mathrm{dist}(H_\tau,\partial|_{\Sigma_0}R(F_0))>0$, where $\partial|_{\Sigma_0}$ is the relative boundary in  $\Sigma_0$.

Let $\eta_{i,1}=-\lambda_{i,1}\eta_i=-\lambda'_i(1-2\tau)(t_{i,+}-t_{i,-})\eta_i,\,\eta_{i,2}= \lambda_{i,2}\eta_i=(1-\lambda'_i)(1-2\tau)(t_{i,+}-t_{i,-})\eta_i$ on $\bar{\mathcal B}_i$, where $\lambda_{i,1}=\tau(-t_{i,+})+(1-\tau)(-t_{i,-})>0,\,\lambda_{i,2}=(1-\tau)t_{i,+}+\tau t_{i,-}>0$  on $\bar{\mathcal B}_i$, and $\tau>0$ is so small that
\[
\min_{1\le j\le N}\min_{\bar{\mathcal B}_j}\lambda_{j,k}>0\;\;(k=1,\,2).
\]
Applying  Lemma \ref{approx-lem} to matrices $\eta_{i,1}=\eta_{i,1}(p,\beta),\,\eta_{i,2}=\eta_{i,2}(p,\beta)$ with a fixed $(p,\beta)\in\bar{\mathcal{B}}_i$ and a given box $G= Q\times I$, we obtain that for each $\rho>0$, there exist
a function $\omega=(\varphi,\psi)\in C^\infty_c( Q\times I;\R^{1+n})$  and an open set $G_\rho\subset\subset  Q\times I$ satisfying the following conditions:
\begin{equation}\label{approx-1}
\begin{cases} \mbox{(1) \;  $\dv \psi =0$  in $Q\times I$,}\\
\mbox{(2) \; $|(Q\times I) \setminus G_\rho|<\rho$;\; $\xi_i+\nabla  \omega(z)\in \{\xi^\tau_{i,-},\;\xi_{i,+}^\tau\}$ for all $z\in  G_\rho$,}\\
\mbox{(3) \; $\xi_i+\nabla  \omega(z)\in [\xi_{i,-}^\tau,\xi_{i,+}^\tau]_\rho$ for all $z\in Q\times I,$}\\
\mbox{(4) \; $\|\omega\|_{L^\infty(Q\times I)}<\rho$,} \\
\mbox{(5) \; $\int_Q \varphi(x,t)\,dx=0$ for all $t\in I$,} \\
\mbox{(6) \;  $\|\varphi_t\|_{L^\infty(Q\times I)}<2\rho,$}
\end{cases}
\end{equation}
where  $[\xi_{i,-}^\tau,\xi_{i,+}^\tau]_\rho$ denotes the $\rho$-neighborhood of closed line segment $[\xi_{i,-}^\tau,\xi_{i,+}^\tau].$
Here, from (\ref{approx-1}.3), condition (\ref{approx-1}.6) follows as
\[
|\varphi_t|<|c_{i,+}-c_{i,-}|+\rho=(t_{i,+}-t_{i,-})|b_i|+\rho<2\rho \quad\textrm{in $Q\times I$.}
\]

Note (a), (c), (f), and (g) follow from (\ref{approx-1}), where $2\rho$ in  (\ref{approx-1}.6) can be adjusted to $\rho$ as in (g).
By the uniform continuity of $A$ on the set $J=\{p'\in\R^n\;|\; |p'|\le s_+(r_1)\;\mbox{(\textbf{Case I})},\;|p'|\le s_+(r_2)\;\mbox{(\textbf{Case II})}\}$,
we can find a $\delta'>0$ such that $|A(p')-A(p'')|< \frac{\epsilon}{3|\tilde Q\times\tilde I|}$ whenever $p',\,p''\in J$ and $|p'-p''|<\delta'.$
We then choose a $\tau>0$ so small that
\[
C\tau<\delta',\;\;C|\tilde Q\times \tilde I|\tau<\frac{\epsilon}{3}.
\]
Next, we  choose a $\delta>0$ such that
$
\delta<\frac{d_\tau}{2}.
$
If $0<\rho<\delta$, then by (\ref{approx-1}.1) and (\ref{approx-1}.3),
for all $z\in Q\times I$ and $|(p',\beta')-(p,\beta)|\le \delta$,
\[
\xi_i(p',\beta')+\nabla\omega(z)\in\Sigma_0,\quad\mathrm{dist}(\xi_i(p',\beta')+\nabla\omega(z),H_\tau)<
d_\tau,
\]
and so $\xi_i(p',\beta')+\nabla\omega(z)\in R(F_0),$ that is, $(p'+D\varphi(z), \beta'+ \psi_t(z))\in \mathcal{S}$. Thus (b) holds for all $0<\rho<\delta.$ In particular, $(p+D\varphi(z), \beta+ \psi_t(z))\in \mathcal{S}$ and so $|p+D\varphi(z)|\le s_+(r_1)$ (\textbf{Case I}), $|p+D\varphi(z)|\le s_+(r_2)$ (\textbf{Case II}) and $|\beta+\psi_t(z)|\le r_2$ for all $z\in Q\times I$, by (i) of Theorem \ref{thm:main2}. Thus
\begin{equation*}
\begin{split}
\int_{ Q\times I} & |\beta+\psi_t -A(p+D\varphi)|dz \\
 &\le \int_{G_\rho}|\beta+\psi_t -A(p+D\varphi)|dz + (r_2+M_\sigma)\rho
\\
&\le |Q\times I| \max\{|\beta_{i,\pm}^\tau -A(p_{i,\pm}^\tau)|\} +(r_2+M_\sigma)\rho
\\
&\le C|Q\times I|\tau+ |Q\times I| \max\{|A(p_{i,\pm}) -A(p_{i,\pm}^\tau)|\} +(r_2+M_\sigma)\rho \\
& \le \frac{2\epsilon|Q\times I|}{3|\tilde Q\times\tilde I|}+(r_2+M_\sigma)\rho,
\end{split}
\end{equation*}
where  $\xi^\tau_{i,\pm}=\begin{pmatrix} p^\tau_{i,\pm} & c^\tau_{i,\pm}\\B^\tau_{i,\pm} & \beta^\tau_{i,\pm}\end{pmatrix}$ and $M_\sigma=\sigma(s_0)$ (\textbf{Case I}), $M_\sigma=\sigma(s_2^*)$ (\textbf{Case II}).
Thus, (d) holds for all $\rho>0$ satisfying $(r_2+M_\sigma)\rho<\frac{\epsilon|Q\times I|}{3|\tilde Q\times\tilde I|}$. Similarly,
\begin{equation*}
\begin{split}
\int_{Q\times I} & \mathrm{dist}((p+D\varphi(z),\beta+\psi_t(z)),\mathcal A)dz \\
 &\le \int_{G_\rho}\max|(p_{i,\pm}^\tau,\beta_{i,\pm}^\tau)-(p_{i,\pm},\beta_{i,\pm})|dz + 2(r_2+M_\sigma)\rho
\\
&\le C|Q\times I|\tau+ 2(r_2+M_\sigma)\rho \\
& \le \frac{\epsilon|Q\times I|}{3|\tilde Q\times\tilde I|}+2(r_2+M_\sigma)\rho;
\end{split}
\end{equation*}
therefore, (e) also holds for all such $\rho>0$ for (d).

We have verified (a) -- (g) for any $(p,\beta)\in\bar{\mathcal{B}}_i$ and $1\le i\le N$, where $\delta>0$ is independent of the index $i$. Since $\mathcal B_1,\cdots, \mathcal B_N$ cover $\mathcal K$,
the proof is now complete.
\end{proof}

\section{Boundary function $\Phi$ and the admissible set $\mathcal U$}\label{sec:add-set}

 Assume $\Omega$ and $u_0$ satisfy (\ref{assume-1}) and (\ref{av-0}).

\subsection{Boundary function $\Phi$} We first construct  a suitable boundary function $\Phi=(u^*,v^*)$ in each case of the profile  $\sigma(s).$ 
Recall that the goals of \textbf{Cases I} and \textbf{II} are to prove Theorems \ref{thm:PM-1} and \ref{thm:H-1} under Hypotheses (PM) and (H), respectively.

\subsection*{Case I}
In this case, $\Omega$ is assumed to be convex. Let $0<r=r_2<\sigma(M_0)$, and let $l=l_{r_2}\in(0,r)$ be some number determined by Theorem \ref{thm:main2}. Choose any $\tilde r=r_1\in(l,r)$.

\subsection*{Case II} In this case, we assume  $|Du_0(x_0)|\in(s_1^*, s_2^*)$ for some $x_0\in\Omega$.
Let $\sigma(s_2)<r=r_2<\sigma(M_0')$, and let $l=l_{r_2}\in(\sigma(s_2),r)$ be some number determined by Theorem \ref{thm:main2}. Choose any $\tilde r=r_1\in(l,r)$.\\

We now apply Lemma \ref{lem:modi-PM} (\textbf{Case I}), Lemma \ref{lem:modi-H} (\textbf{Case II})  to determine functions $\tilde\sigma,\,\tilde f\in C^{3}([0,\infty))$ (\textbf{Case I}), $\tilde\sigma,\,\tilde f\in C^{1+\alpha}([0,\infty))$ (\textbf{Case II}) satisfying its conclusion. Also, let $\tilde A(p)=\tilde f(|p|^2)p$ $(p\in\R^n).$ Then

\begin{lem}\label{lem-a} We have
\[
(p,\tilde A(p))\in \mathcal S  \quad\forall\; s_-(r_1)<|p|<s_+(r_2),
\]
where $\mathcal S=\mathcal S_{r_1,r_2}$ is the set in Lemma \ref{lem-rough}.
\end{lem}
\begin{proof}
Let $s=|p|$, $r=\tilde\sigma(s)$ and $\zeta=p/|p|$, so that  $s_-(r_1)<s<s_+(r_2)$, $\zeta\in \mathbb{S}^{n-1}$ and $\tilde A(p)=r\zeta$. By Lemma \ref{lem:modi-PM} (\textbf{Case I}), Lemma \ref{lem:modi-H} (\textbf{Case II}), $s_-(r)<s<s_+(r)$ and $r_1<r<r_2$ . Set $p_\pm=s_\pm(r)\zeta$ and $\beta_\pm=r\zeta$. Then $A(p_\pm)=r\zeta=\beta_\pm$. Define
$
\xi=\begin{pmatrix} p & 0\\ O & \tilde A(p)\end{pmatrix}
$
and
$
\xi_\pm=\begin{pmatrix} p_\pm & 0\\ O & \beta_\pm\end{pmatrix}.
$
Then $\xi=\lambda\xi_++(1-\lambda)\xi_-$ for some $0<\lambda<1.$ Since $\xi_\pm\in F_\pm$ and $\mathrm{rank}(\xi_+-\xi_-)=1$, it follows from the definition of $R(F_0)=R(F_{r_1,r_2}(0))$ that $\xi\in(\xi_-,\xi_+)\subset R(F_0)$. Thus, by Lemma \ref{lem-rough}, $(p,\tilde A(p))\in\mathcal S$.
\end{proof}

By Lemma \ref{lem:modi-PM} (\textbf{Case I}), Lemma \ref{lem:modi-H} (\textbf{Case II}), equation $u_t=\dv (\tilde A(Du))$ is uniformly parabolic. So by Theorem \ref{existence-gr-max}, the  initial-Neumann boundary value problem
\begin{equation}\label{ib-para}
\begin{cases} u^*_t =\dv (\tilde A(Du^*))&\mbox{in $\Omega_T$}\\
\partial u^*/\partial \n  =0 & \mbox{on $\partial \Omega\times (0,T)$}  \\
u^*(x,0)=u_0(x), &x\in \Omega
\end{cases}
\end{equation}
admits a unique classical solution  $u^*\in C^{2+\alpha,1+\alpha/2}(\bar\Omega_T)$; moreover, only in \textbf{Case I}, it satisfies
\[
|D u^*(x,t)|\le M_0\quad \forall(x,t)\in\Omega_T.
\]

From conditions (\ref{assume-1}) and (\ref{av-0}),  we can find a function $h\in C^{2+\alpha}(\bar\Omega)$ satisfying
\[
\Delta h=u_0\;\;\mbox{in}\;\;\Omega,\quad \partial h/\partial \n=0\;\;\mbox{on}\;\;\partial\Omega.
\]
Let $v_0=Dh\in C^{1+\alpha}(\bar\Omega;\R^n)$ and define,  for $(x,t)\in\Omega_T$,
\begin{equation}\label{def-v}
v^*(x,t)=v_0(x)+\int_0^t \tilde A(Du^*(x,s))\,ds.
\end{equation}
Then it is easily seen that $\Phi := (u^*,v^*)\in C^{1}(\bar\Omega_T;\R^{1+n})$ satisfies (\ref{bdry-1}); that is,
\begin{equation}\label{bdry-2}
\begin{cases} u^*(x,0)=u_0(x) \; (x\in\Omega),\\
\dv v^*=u^*\;\;\textrm{a.e. in $\Omega_T$}, \\
v^*(\cdot,t) \cdot \n|_{\partial\Omega} =0 \; \; \forall\; t\in [0,T],\end{cases}
\end{equation}
and so $\Phi$ is a boundary function for the initial datum $u_0$.

Next, let
\[
\mathcal F=\left\{(p,A(p))\;|\; |p|\in[0,s_-(r_1)] \right\}, \quad (\textbf{Case I})
\]
\[
\mathcal F=\left\{(p,A(p))\;|\; |p|\in[0,s_-(r_1)]\cup[s_+(r_2),\max\{M_0,s_2^*\}] \right\}. \quad (\textbf{Case II})
\]
Then we have the following:

\begin{lem} \label{lem-b}
\[
(Du^*(x,t), v^*_t(x,t))\in \mathcal S \cup \mathcal F\quad \forall\; (x,t)\in \Omega_T.
\]
\end{lem}
\begin{proof}
Let $(x,t)\in\Omega_T$ and $p=Du^*(x,t)$.

\textbf{Case I:} In this case, $|p|\le M_0<s_+(r_2).$ If $|p|\le s_-(r_1)$, then $\tilde A(p)=A(p)$ and hence by (\ref{def-v})
\[
(Du^*(x,t),v_t^*(x,t))=(p,\tilde A(p))=(p,A(p))\in \mathcal F.
\]
If $s_-(r_1)<|p|\le M_0$, then by Lemma \ref{lem-a} and (\ref{def-v})
\[
(Du^*(x,t),v_t^*(x,t))=(p,\tilde A(p))\in \mathcal S.
\]

\textbf{Case II:} If $|p|\le s_-(r_1)$ or $s_+(r_2)\le|p|\le M_0$, then $\tilde A(p)=A(p)$ and hence as above
\[
(Du^*(x,t),v_t^*(x,t))=(p,\tilde A(p))=(p,A(p))\in \mathcal F.
\]
If $s_-(r_1)<|p|<s_+(r_2)$, then as above
\[
(Du^*(x,t),v_t^*(x,t))=(p,\tilde A(p))\in \mathcal S.
\]

Therefore, for both cases, $(Du^*,v_t^*)\in\mathcal S\cup\mathcal F$ \,in $\Omega_T$.
\end{proof}

\begin{defn}
 We say  a function $u$  is {\em piecewise $C^1$} in $\Omega_T$ and write  $u\in C^1_{piece}(\Omega_T)$  if there exists a sequence of disjoint open sets $\{G_j\}_{j=1}^\infty$ in $\Omega_T$ such that
\[
u\in C^1(\bar G_j)  \;\; \forall j\in\mathbb{N},\quad |\Omega_T\setminus \cup_{j=1}^\infty G_j|=0.
\]
It is then necessary to have   $|\partial G_j|=0\;\forall j\in\mathbb{N}$.
\end{defn}

\subsection{Selection of interface}  To separate the space-time domain $\Omega_T$ into the classical and micro-oscillatory parts for Lipschitz solutions, we assume the following.
\subsubsection*{\textbf{\emph{Case I}}}
Observe   that
\[
|\{(x,t)\in\Omega_T\,|\,|Du^*(x,t)|=s_-(\bar r)\}|>0
\]
for at most countably many $\bar r\in(0,\tilde r)$. We fix any $\bar r\in(0,\tilde r)$ with
\[
|\{(x,t)\in\Omega_T\,|\,|Du^*(x,t)|=s_-(\bar r)\}|=0,
\]
and let
\[
\Omega_T^1 = \{(x,t)\in\Omega_T\,|\,|Du^*(x,t)|<s_-(\bar r)\},
\]
\[
\Omega_T^2 = \{(x,t)\in\Omega_T\,|\,|Du^*(x,t)|>s_-(\bar r)\},
\]
so that $\Omega_T^1$ and $\Omega_T^2$ are disjoint open subsets of $\Omega_T$ whose union has full measure $|\Omega_T|.$
Define
\[
\Omega_0^{\bar r}=\{(x,0) \,|\, x\in\Omega,\,|Du_0(x)|<s_-(\bar r) \};
\]
then $\Omega_0^{\bar r}\subset \partial\Omega_T^1.$

\subsubsection*{\textbf{\emph{Case II}}}
Observe   that
\[
|\{(x,t)\in\Omega_T\,|\,|Du^*(x,t)|=s_-(\bar r_j)\}|>0\quad(j=1,3)
\]
for at most countably many $\bar r_1\in(\sigma(s_2),\tilde r),\,\bar r_3\in(r,\sigma(s_1))$. We fix any two $\bar r_1\in(\sigma(s_2),\tilde r),\,\bar r_3\in(r,\sigma(s_1))$ such that
\[
|\{(x,t)\in\Omega_T\,|\,|Du^*(x,t)|=s_-(\bar r_j)\}|=0,\quad(j=1,3)
\]
and let
\[
\Omega_T^1 = \{(x,t)\in\Omega_T\,|\,|Du^*(x,t)|<s_-(\bar r_1)\},
\]
\[
\Omega_T^3 = \{(x,t)\in\Omega_T\,|\,|Du^*(x,t)|>s_+(\bar r_3)\},
\]
and
\[
\Omega_T^2 = \{(x,t)\in\Omega_T\,|\, s_-(\bar r_1)<|Du^*(x,t)|<s_+(\bar r_3)\},
\]
so that these are disjoint open subsets of $\Omega_T$ whose union has full measure $|\Omega_T|.$
Define
\[
\Omega_0^{\bar r_1}=\{(x,0) \,|\, x\in\Omega,\,|Du_0(x)|<s_-(\bar r_1) \},
\]
\[
\Omega_0^{\bar r_3}=\{(x,0) \,|\, x\in\Omega,\,|Du_0(x)|>s_+(\bar r_3) \};
\]
then $\Omega_0^{\bar r_j}\subset \partial\Omega_T^j\;(j=1,3).$

\subsection{The admissible set $\mathcal U$} Let $m=\|u_t^*\|_{L^\infty(\Omega_T)}+1.$  We finally define the admissible set $\mathcal U$ and approximating sets $\mathcal U_\epsilon\;(\epsilon>0)$ as follows. Recall $0<\bar r<\tilde r=r_1<r=r_2<\sigma(s_0)$ (\textbf{Case I}), $\sigma(s_2)<\bar r_1 <\tilde r=r_1<r=r_2<\bar r_3<\sigma(s_1)$ (\textbf{Case II}).

\subsubsection*{\bf Case I}
\begin{equation*}
\begin{split}
\mathcal U=\Big \{ &u\in C^1_{piece} \cap W_{u^*}^{1,\infty}(\Omega_T)\;\big| \; \mbox{$u\equiv u^*$ in $\Omega_T^1$, $\|u_t\|_{L^\infty(\Omega_T)}<m,$}  \\
& \mbox{$\exists \, v\in C^1_{piece} \cap W_{v^*}^{1,\infty}(\Omega_T;\R^n)$ such that $v\equiv v^*$ in $\Omega_T^1,$}   \\
&\mbox{$\dv v=u$ and  $(Du,v_t)\in \mathcal S\cup \mathcal F$ a.e.\;in $\Omega_T$, and}\\
&\mbox{$(Du,v_t)\in \mathcal S\cup \mathcal B$ a.e.\;in $\Omega_T^2$} \Big\},\end{split}
\end{equation*}
\[
\begin{split}
\mathcal U_\epsilon= \Big \{ & u\in \mathcal U\;\big| \; \mbox{$\exists \, v\in C^1_{piece} \cap W_{v^*}^{1,\infty}(\Omega_T;\R^n)$ such that $v\equiv v^*$ in $\Omega_T^1,$} \\
& \mbox{$\dv v=u$ and $(Du,v_t)\in \mathcal S\cup \mathcal F$ a.e.\,in $\Omega_T$,} \\
& \mbox{$(Du,v_t)\in \mathcal S\cup \mathcal B$ a.e.\;in $\Omega_T^2$, $\int_{\Omega_T} |v_t-A(Du)|dxdt\le\epsilon|\Omega_T|$, and} \\
& \mbox{$\int_{\Omega_T^2} \mathrm{dist}((Du,v_t),\mathcal C) dxdt\le\epsilon|\Omega_T^2|$} \Big\},\end{split}
\]
where
\[
\mathcal B=\mathcal B_{\bar r, \tilde r, r}=\big\{(p,A(p)) \,|\, |p|\in [s_-(\bar r),s_-(\tilde r)]  \big\},
\]
\[
\mathcal C=\mathcal C_{\bar r, \tilde r, r}=\big\{(p,A(p)) \,|\, |p|\in [s_-(\bar r),s_-(r)]\cup[s_+(r),s_+(\tilde r)] \big\}.
\]

\subsubsection*{\bf Case II}
\begin{equation*}
\begin{split}
\mathcal U=\Big \{ &u\in C^1_{piece} \cap W_{u^*}^{1,\infty}(\Omega_T)\;\big| \; \mbox{$u\equiv u^*$ in $\Omega_T^1\cup\Omega_T^3$, $\|u_t\|_{L^\infty(\Omega_T)}<m,$}  \\
& \mbox{$\exists \, v\in C^1_{piece} \cap W_{v^*}^{1,\infty}(\Omega_T;\R^n)$ such that $v\equiv v^*$ in $\Omega_T^1\cup\Omega_T^3,$}   \\
&\mbox{$\dv v=u$ and  $(Du,v_t)\in \mathcal S\cup \mathcal F$ a.e.\;in $\Omega_T$, and}\\
&\mbox{$(Du,v_t)\in \mathcal S\cup \mathcal B$ a.e.\;in $\Omega_T^2$} \Big\},\end{split}
\end{equation*}
\[
\begin{split}
\mathcal U_\epsilon= \Big \{ & u\in \mathcal U\;\big| \; \mbox{$\exists \, v\in C^1_{piece} \cap W_{v^*}^{1,\infty}(\Omega_T;\R^n)$ such that $v\equiv v^*$ in $\Omega_T^1\cup\Omega_T^3,$} \\
& \mbox{$\dv v=u$ and $(Du,v_t)\in \mathcal S\cup \mathcal F$ a.e.\,in $\Omega_T$,} \\
& \mbox{$(Du,v_t)\in \mathcal S\cup \mathcal B$ a.e.\;in $\Omega_T^2$, $\int_{\Omega_T} |v_t-A(Du)|dxdt\le\epsilon|\Omega_T|$, and} \\
& \mbox{$\int_{\Omega_T^2} \mathrm{dist}((Du,v_t),\mathcal C) dxdt\le\epsilon|\Omega_T^2|$} \Big\},\end{split}
\]
where
\[
\mathcal B=\mathcal B_{\bar r_1, \tilde r, r,\bar r_3}=\big\{(p,A(p)) \,|\, |p|\in [s_-(\bar r_1),s_-(\tilde r)]\cup [s_+(r),s_+(\bar r_3)]  \big\},
\]
\[
\mathcal C=\mathcal C_{\bar r_1, \tilde r, r,\bar r_3}=\big\{(p,A(p)) \,|\, |p|\in [s_-(\bar r_1),s_-(r)]\cup[s_+(\tilde r),s_+(\bar r_3)] \big\}.
\]

Observe that for both cases, $\mathcal A\cup\mathcal B= \mathcal C$ and $\mathcal B\subset \mathcal F$, where $\mathcal A$ is as in Theorem \ref{main-lemma}. Also for both cases, as in the proof of Lemma \ref{lem-b}, it is easy to check that $(Du^*,v^*_t)\in \mathcal S\cup \mathcal B$ in $\Omega_T^2$.

Note that one more requirement is imposed on the elements of \,$\mathcal U_\epsilon$ in both cases  than in the general density approach in Subsection \ref{subsec-density}. As we will see later, such a smallness condition on the distance integral is designed to extract the micro-structured ramping  of Lipschitz  solutions with alternate gradients whose magnitudes lie in two disjoint (possibly very small) compact intervals; this occurs only in $\Omega_T^2$, and the solutions are classical elsewhere.

\begin{remk}\label{rmk-1} Summarizing the above, we have constructed a boundary function $\Phi=(u^*,v^*)\in C^1(\bar\Omega_T;\R^{1+n})$ for the initial datum $u_0$ with $u^*\in\mathcal U$; so $\mathcal U$ is nonempty. Also  $\mathcal U$ is a bounded subset of $W^{1,\infty}_{u^*}(\Omega_T)$, since $\mathcal S\cup \mathcal F$ is bounded and $\|u_t\|_{L^\infty(\Omega_T)}<m$ for all $u\in\mathcal U$. Moreover, by (i) of Theorem \ref{thm:main2} and the definition of $\mathcal F$, for each $u\in\mathcal U$, its corresponding vector function $v$ satisfies $\|v_t\|_{L^\infty(\Omega_T)}\le r=r_2$ (\textbf{Case I}), $\|v_t\|_{L^\infty(\Omega_T)}\le \max\{r=r_2,\sigma(M_0)\}$ (\textbf{Case II}); this bound in each case plays the role of a fixed number $R>0$ in the general density approach in Subsection \ref{subsec-density}. Finally, note that $s_-(r_1)<|Du^*|<s_+(r_2)$ on some nonempty open subset of $\Omega_T$ and  so $\tilde A(Du^*)\ne A(Du^*)$ on a subset of $\Omega_T$ with positive measure;  hence $u^*$ itself is not a Lipschitz solution to (\ref{ib-P}).

In view of the general existence theorem, Theorem \ref{thm:main}, it only remains to prove the $L^\infty$-density of $\mathcal U_\epsilon$ in $\mathcal U$ towards the existence of infinitely many Lipschitz solutions for both cases; this core subject is carried out in the next section.
\end{remk}

\section{Density of $\mathcal U_\epsilon$ in $\mathcal U$: \\Final step for the proofs  of Theorems \ref{thm:PM-1} and \ref{thm:H-1}}\label{sec:den-proof}

In this section, we follow Section \ref{sec:add-set} to complete the proofs of Theorems \ref{thm:PM-1} and \ref{thm:H-1}.

\subsection{The density property} The density theorem below is the last preparation  for both cases.

\begin{thm} \label{thm-density-1} For each $\epsilon>0$, $\mathcal U_\epsilon$ is dense in $\mathcal U$ under the $L^\infty$-norm.
\end{thm}

\begin{proof}

Let  $u\in \mathcal U$, $\eta>0$. The goal is to construct a function  $\tilde u\in  \mathcal U_\epsilon$ such that $\|\tilde u-u\|_{L^\infty(\Omega_T)}<\eta.$
For clarity, we divide the proof into several steps.

1. Note that $u\equiv u^*$  in  $\Omega_T^1$ (\textbf{Case I}), in  $\Omega_T^1\cup\Omega_T^3$ (\textbf{Case II}), $\|u_t\|_{L^\infty(\Omega_T)}<m-\tau_0$ for some $\tau_0>0$, and there exists a vector function $v\in C^1_{piece} \cap W_{v^*}^{1,\infty}(\Omega_T;\R^n)$   such that  $v\equiv v^*$  in  $\Omega_T^1$ (\textbf{Case I}), in  $\Omega_T^1\cup\Omega_T^3$ (\textbf{Case II}), $\dv v=u$ and  $(Du,v_t)\in \mathcal S\cup \mathcal F$ a.e.\;in $\Omega_T$, and   $(Du,v_t)\in \mathcal S\cup \mathcal B$ a.e.\;in $\Omega_T^2$.
Since both $u$ and $v$ are piecewise $C^1$ in $\Omega_T$,  there exists a sequence of disjoint open sets $\{G_j\}_{j=1}^\infty$ in $\Omega_T^2$ with  $|\partial G_j|=0$ such that
\[
u\in C^1(\bar G_j),
\;v\in C^1(\bar G_j;\R^n)\; \; \forall j\geq 1,\quad |\Omega_T^2\setminus \cup_{j=1}^\infty G_j|=0.
\]

2. Let $j\in \mathbb{N}$ be fixed. Note that $(Du(z),v_t(z))\in\bar{\mathcal{S}}\cup\mathcal{B}$ for all $z=(x,t)\in G_j$ and that $H_j:=\{z\in G_j\;|\;(Du(z),v_t(z))\in\partial\mathcal{S}\}$ is a (relatively) closed set in $G_j$ with measure zero. So $\tilde G_j:=G_j\setminus H_j$ is an open subset of $G_j$ with $|\tilde G_j|=|G_j|$, and $(Du(z),v_t(z))\in\mathcal{S}\cup\mathcal{B}$ for all $z\in \tilde G_j$.

3. For each $\tau>0$, let
\[
\mathcal G_{\tau}=\left\{(p,\beta)\in  \mathcal S \;|\; \mathrm{dist}((p,\beta),\partial\mathcal S)> \tau,\,\mathrm{dist}((p,\beta),\mathcal C)>\tau \right\};
\]
then
\[
\mathcal S  =  (\cup_{\tau>0}\mathcal G_\tau)\cup\{(p,\beta)\in\mathcal S\;|\;\mathrm{dist}((p,\beta),\mathcal C)=0\}.
\]
Since $\mathcal B\subset\mathcal C$, we have
\[
\begin{split}
\int_{\tilde G_j} & |v_t (z)- A(Du(z))|\,dz\\
& =\lim_{\tau\to 0^+}\int_{\{z\in\tilde G_j\;|\;(Du(z),v_t(z))\in \mathcal G_\tau\}} |v_t(z)-A(Du(z))|\,dz,
\end{split}
\]
\[
\begin{split}
\int_{\tilde G_j} & \mathrm{dist}((Du(z),v_t(z)),\mathcal C)\,dz\\
& =\lim_{\tau\to 0^+}\int_{\{z\in\tilde G_j\;|\;(Du(z),v_t(z))\in \mathcal G_\tau\}} \mathrm{dist}((Du(z),v_t(z)),\mathcal C)\,dz;
\end{split}
\]
thus we can find a $\tau_j>0$ such that $|\partial O_j|=0$,
\begin{equation}\label{density-3}
\int_{F_j} |v_t(z)-A(Du(z))|\,dz < \frac{\epsilon}{3\cdot2^{j}}|\Omega^2_T|,
\end{equation}
and
\begin{equation}\label{density-3-1}
\int_{F_j} \mathrm{dist}((Du(z),v_t(z)),\mathcal C)\,dz < \frac{\epsilon}{3\cdot2^{j}}|\Omega^2_T|,
\end{equation}
where $F_j=  \{z\in \tilde G_j\; |\; (Du(z),v_t(z))\notin \mathcal G_{\tau_j}\}$ and $O_j=\tilde G_j\setminus F_j$ is open.  Let $J$ be the set of all indices $j\in\mathbb{N}$ with $O_j\neq\emptyset$. Then for $j\not\in J$, $F_j=\tilde G_j$.

4. We now fix a $j\in J$. Note that $O_j=\{z\in \tilde G_j\; |\; (Du(z),v_t(z))\in \mathcal G_{\tau_j}\}$ and that $\mathcal K_j:=\bar{\mathcal G}_{\tau_j}$ is a compact subset of $\mathcal S$. Let $\tilde Q\subset\R^n$ be a  box with $\Omega\subset\tilde Q$ and $\tilde I=(0,T)$. Applying Theorem \ref{main-lemma} to box $\tilde Q\times\tilde I,$ $\mathcal K_j\subset\subset\mathcal S=\mathcal S_{r_1,r_2}$, and $\epsilon'=\frac{\epsilon|\Omega_T|}{12}$, we obtain a constant $\delta_j>0$ that satisfies  the conclusion of the theorem. By the uniform continuity of $A$ on compact subsets of $\R^n$, we can find a $\theta=\theta_{\epsilon,\bar s}>0$ such that
\begin{equation}\label{density-5}
|A(p)-A(p')|<\frac{\epsilon}{12}
\end{equation}
whenever $|p|,\,|p'|\le \bar s$ and $|p-p'|\le \theta$, where the number $\bar s>0$ will be chosen later.
Also by the uniform continuity of $u$, $v$ and their gradients on $\bar G_j$,
there exists a $\nu_j>0$ such that
\begin{equation}\label{density-1}
\begin{array}{c}
  |u(z)-u(z')|+|\nabla u(z)-\nabla u(z')| +|v(z)-v(z')| \\
  +|\nabla v(z)-\nabla v(z')| <\min\{\frac{\delta_{j}}{2},\frac{\epsilon}{12},\theta\} \end{array}
\end{equation}
whenever $z,z'\in \bar G_j$ and  $|z-z'|\le \nu_j.$ We now cover $O_j$ (up to measure zero) by a sequence of disjoint boxes $\{Q_j^i\times I^i_j\}_{i=1}^\infty$ in $O_j$ with center $z_j^i$ and diameter $l^i_j<\nu_j.$

5. Fix an $i\in\mathbb{N}$ and write  $w=(u,v)$,  $\xi=\begin{pmatrix} p&c\\B & \beta\end{pmatrix}=\nabla w(z_j^i)=\begin{pmatrix} Du(z_j^i) & u_t(z_j^i)\\Dv(z_j^i) & v_t(z_j^i)\end{pmatrix}.$   By the choice of $\delta_j>0$ in Step 4 via Theorem \ref{main-lemma}, since $Q^i_j\times I^i_j\subset\tilde Q\times\tilde I$ and $(p,\beta)\in\mathcal K_j$,  for all sufficiently small $\rho>0$, there exists a function $\omega^i_j=(\varphi^i_j,\psi^i_j)\in C^\infty_c(Q^i_j\times I^i_j;\R^{1+n})$ satisfying

(a) \; $\dv \psi^i_j =0$  in $Q^i_j\times I^i_j$,

(b) \; $(p'+D\varphi^i_j(z), \beta'+ (\psi^i_j)_t(z))\in \mathcal{S}$ for all $z\in Q^i_j\times I^i_j$

\hspace{8mm} and all $|(p',\beta')-(p,\beta)|\leq\delta_{j},$

(c) \; $\|\omega^i_j\|_{L^\infty(Q^i_j\times I^i_j)}<\rho,$

(d) \; $\int_{Q^i_j\times I^i_j} |\beta+(\psi^i_j)_t(z)-A(p+D\varphi^i_j(z))|dz<\epsilon' |Q^i_j\times I^i_j|/|\tilde Q\times \tilde I|,$

(e) \; $\int_{Q^i_j\times I^i_j} \mathrm{dist}((p+D\varphi^i_j(z),\beta+(\psi^i_j)_t(z)),\mathcal A)dz  <\epsilon' |Q^i_j\times I^i_j|/|\tilde Q\times \tilde I|,$

(f) \; $\int_{Q^i_j} \varphi^i_j(x,t)dx=0$  for all $t\in I^i_j,$

(g) \; $\|(\varphi^i_j)_t\|_{L^\infty(Q^i_j\times I^i_j)}<\rho,$\\
where the set $\mathcal A=\mathcal A_{r_1,r_2}\subset \R^{n+n}$ is as in Theorem \ref{main-lemma}.
Here, we let $0<\rho\leq\min\{\tau_0,\frac{\delta_{j}}{2C},\frac{\epsilon}{12C},\eta\}$, where $C_n>0$ is the constant in Theorem \ref{div-inv} and $C$ is the product of $C_n$ and the sum of the lengths of all sides of $\tilde Q$.
From $\varphi^i_j|_{\partial(Q^i_j\times I^i_j)}\equiv 0$ and (f), we can apply Theorem \ref{div-inv} to $\varphi^i_j$ on $Q^i_j\times I^i_j$ to obtain a function $g^i_j=\mathcal R\varphi^i_j\in C^1( \overline{Q^i_j\times I^i_j};\R^n)\cap W^{1,\infty}_0( Q^i_j\times I^i_j;\R^n)$
such that $\dv g^i_j=\varphi^i_j$ in $Q^i_j\times I^i_j$ and
\begin{equation}\label{density-2}
\|(g^i_j)_t\|_{L^\infty(Q^i_j\times I^i_j)}\leq C\|(\varphi^i_j)_t\|_{L^\infty(Q^i_j\times I^i_j)}\leq\frac{\delta_{j}}{2}.\;\;\mbox{(by (g))}
\end{equation}

6. As $|v_t-A(Du)|,\,\mathrm{dist}((Du,v_t),\mathcal C)\in L^\infty(\Omega_T)$, we can select a finite index set $\mathcal I\subset J\times \mathbb{N}$ such that
\begin{equation}\label{density-4}
\int_{\bigcup_{(j,i)\in (J\times\mathbb{N}) \setminus \mathcal I}Q^i_j\times I^i_j}|v_t(z)-A(Du(z))|dz\le\frac{\epsilon}{3}|\Omega_T^2|,
\end{equation}
\begin{equation}\label{density-4-1}
\int_{\bigcup_{(j,i)\in (J\times\mathbb{N}) \setminus \mathcal I}Q^i_j\times I^i_j} \mathrm{dist}((Du(z),v_t(z)),\mathcal C) dz\le\frac{\epsilon}{3}|\Omega_T^2|.
\end{equation}
We finally define
\[
(\tilde u,\tilde v)=(u,v)+\sum_{(j,i)\in\mathcal I}\chi_{Q^i_j\times I^i_j}(\varphi^i_j,\psi^i_j+g^i_j)\;\;\textrm{ in $\Omega_T$.}
\]
As a side remark, note here that only \emph{finitely} many functions $(\varphi^i_j,\psi^i_j+g^i_j) $ are disjointly patched  to the original $(u,v)$ to obtain a new function $(\tilde u,\tilde v)$ towards the goal of the proof. The reason for using only finitely many pieces of gluing is due to the lack of control over the spatial gradients $D(\psi^i_j+g^i_j)$, and overcoming this difficulty is at the heart of  this paper.

7. Let us finally check that $\tilde u$  together with  $\tilde v$ indeed gives the desired result. By construction, it is clear that $\tilde u\in C^1_{piece}\cap W^{1,\infty}_{u^*}(\Omega_T)$, $\tilde v\in C^1_{piece}\cap W^{1,\infty}_{v^*}(\Omega_T;\R^n)$ and that $\tilde u=u=u^*$ and $\tilde v=v=v^*$ in $\Omega_T^1$ (\textbf{Case I}), in $\Omega_T^1\cup\Omega_T^3$ (\textbf{Case II}).
By the choice of $\rho$ in (g) as $\rho\le\tau_0$, we have $\|\tilde u_t\|_{L^\infty(\Omega_T)}<m.$
Next, let $(j,i)\in\mathcal I,$ and observe that for $z\in Q^i_j\times I^i_j$, with $(p,\beta)=(Du(z^i_j),v_t(z^i_j))\in\mathcal G_{\tau_j}$, since $|z-z_j^i|<l^i_j<\nu_j$, it follows from (\ref{density-1}) and (\ref{density-2}) that
\[
|(Du(z),v_t(z)+(g^i_j)_t(z))-(p,\beta)|\leq\delta_{j},
\]
and so $(D\tilde u(z),\tilde v_t(z))\in\mathcal S$ from (b) above. From (a) and $\dv g^i_j=\varphi^i_j$, for $z\in Q^i_j\times I^i_j$,
\[
\dv \tilde v(z)=\dv (v+\psi^i_j+g^i_j)(z)=u(z)+0+\varphi^i_j(z)=\tilde u(z).
\]
Therefore, $\tilde u\in\mathcal U$.
Next, observe
\[
\int_{\Omega_T}|\tilde v_t-A(D\tilde u)|dz = \int_{\Omega_T\setminus\Omega_T^2}| v_t^*-\tilde A(Du^*)|dz+\int_{\Omega_T^2}|\tilde v_t-A(D\tilde u)|dz
\]
\[
=\int_{\Omega_T^2}|\tilde v_t-A(D\tilde u)|dz=\int_{\cup_{j\in\mathbb{N}}F_j}|v_t-A(Du)|dz
\]
\[
+\int_{\cup_{(j,i)\in (J\times\mathbb{N}) \setminus \mathcal I}Q^i_j\times I^i_j}|v_t-A(Du)|dz+\int_{\cup_{(j,i)\in  \mathcal I}Q^i_j\times I^i_j}|\tilde v_t-A(D\tilde u)|dz
\]
\[
=:I^1_1+I^1_2+I^1_3,
\]
\[
\int_{\Omega_T^2} \mathrm{dist}((D\tilde u,\tilde v_t),\mathcal C) dz=\int_{\cup_{j\in\mathbb{N}}F_j} \mathrm{dist}((D  u,  v_t),\mathcal C) dz
\]
\[
+\int_{\cup_{(j,i)\in (J\times\mathbb{N}) \setminus \mathcal I}Q^i_j\times I^i_j} \mathrm{dist}((D  u,  v_t),\mathcal C)dz+\int_{\cup_{(j,i)\in  \mathcal I}Q^i_j\times I^i_j} \mathrm{dist}((D\tilde u,\tilde v_t),\mathcal C)dz
\]
\[
=:I^2_1+I^2_2+I^2_3.
\]
From (\ref{density-3}), (\ref{density-3-1}), (\ref{density-4}) and (\ref{density-4-1}), we have $I^k_1+I^k_2\leq\frac{2\epsilon}{3}|\Omega_T^2|\;(k=1,2)$. Note   that for $(j,i)\in\mathcal I$ and $z\in Q^i_j\times I^i_j,$ from (\ref{density-1}), (\ref{density-2}) and (g),
\[
|\tilde v_t(z)-A(D\tilde u(z))|=|v_t(z)+(\psi^i_j)_t(z)+(g^i_j)_t(z)-A(Du(z)+D\varphi^i_j(z))|
\]
\[
\leq |v_t(z)-v_t(z^i_j)|+|v_t(z^i_j)+(\psi^i_j)_t(z)-A(Du(z^i_j)+D\varphi^i_j(z))|
\]
\[
+|(g^i_j)_t(z)|+|A(Du(z^i_j)+D\varphi^i_j(z))-A(Du(z)+D\varphi^i_j(z))|
\]
\[
\le\frac{\epsilon}{6}+|v_t(z^i_j)+(\psi^i_j)_t(z)-A(Du(z^i_j)+D\varphi^i_j(z))|
\]
\[
+|A(Du(z^i_j)+D\varphi^i_j(z))-A(Du(z)+D\varphi^i_j(z))|.
\]
Similarly, since $\mathcal A\subset\mathcal C$, we have
\[
\begin{split}
\mathrm{dist}( & (D\tilde u(z),\tilde v_t(z)),\mathcal C)\\
& \le \frac{\epsilon}{4}+\mathrm{dist}((Du(z^i_j)+D\varphi^i_j(z),v_t(z^i_j)+(\psi^i_j)_t(z)),\mathcal C)\\
& \le \frac{\epsilon}{4}+\mathrm{dist}((Du(z^i_j)+D\varphi^i_j(z),v_t(z^i_j)+(\psi^i_j)_t(z)),\mathcal A).
\end{split}
\]
From (b) and (i) of Theorem \ref{thm:main2}, we have
$|Du(z^i_j)+D\varphi^i_j(z)|\le s_+(r_1)=:\bar s$\, (\textbf{Case I}), $|Du(z^i_j)+D\varphi^i_j(z)|\le s_+(r_2)=:\bar s$\, (\textbf{Case II}).
As $(D\tilde u(z),\tilde v_t(z))\in\mathcal S$, we also have $|Du(z)+D\varphi^i_j(z)|=|D\tilde u(z)|\le \bar s$, and by (\ref{density-1}), $|Du(z^i_j)-Du(z)|<\theta$. From (\ref{density-5}), we thus have
\[
|A(Du(z^i_j)+D\varphi^i_j(z))-A(Du(z)+D\varphi^i_j(z))|<\frac{\epsilon}{12}.
\]
Integrating the two inequalities above over $Q^i_j\times I^i_j,$ we now obtain from (d) and (e), respectively, that
\[
\int_{Q^i_j\times I^i_j}|\tilde v_t(z)-A(D\tilde u(z))|dz\le\frac{\epsilon}{4}|Q^i_j\times I^i_j|+\frac{\epsilon|\Omega_T|}{12}\frac{|Q^i_j\times I^i_j|}{|\tilde Q\times \tilde I|}\leq\frac{\epsilon}{3}|Q^i_j\times I^i_j|,
\]
\[
\int_{Q^i_j\times I^i_j} \mathrm{dist}( (D\tilde u(z),\tilde v_t(z)),\mathcal C) dz\le\frac{\epsilon}{4}|Q^i_j\times I^i_j|+\frac{\epsilon|\Omega_T|}{12}\frac{|Q^i_j\times I^i_j|}{|\tilde Q\times \tilde I|}\leq\frac{\epsilon}{3}|Q^i_j\times I^i_j|;
\]
thus $I^k_3\leq\frac{\epsilon}{3}|\Omega_T^2|$, and so $I^k_1+I^k_2+I^k_3\le\epsilon|\Omega_T^2|$, where $k=1,2$. Therefore, $\tilde u\in\mathcal U_\epsilon.$
Lastly, from (c) with $\rho\le \eta$ and the definition of $\tilde u$, we have $\|\tilde u-u\|_{L^\infty(\Omega_T)}<\eta$.

The proof is now complete.
\end{proof}

\subsection{Completion of the proofs of Theorems \ref{thm:PM-1} and \ref{thm:H-1}}

Unless specifically distinguished, the proof below is common for both \textbf{Case I:} Theorem \ref{thm:PM-1} and \textbf{Case II:} Theorem \ref{thm:H-1}.

\begin{proof}[Proofs of Theorems \ref{thm:PM-1} and \ref{thm:H-1}]
We return to Section \ref{sec:add-set}. As outlined in Remark \ref{rmk-1}, Theorem \ref{thm-density-1}  and Theorem \ref{thm:main} together give infinitely many Lipschitz solutions $u$ to problem (\ref{ib-P}).

We now follow the proof of Theorem \ref{thm:main} for detailed information on such a Lipschitz solution $u\in\mathcal G$ to (\ref{ib-P}). Here $Du$ is the a.e.-pointwise limit of some sequence $Du_j$, where the sequence $u_j\in\mathcal U_{1/j}$ converges to $u$ in $L^\infty(\Omega_T)$. Since $u_j\equiv u^*$ in $\Omega_T^1$ (\textbf{Case I}), in $\Omega_T^1\cup\Omega_T^3$ (\textbf{Case II}), we also have $u\equiv u^*\in C^{2+\alpha,1+\alpha/2}(\bar\Omega_T^1)$ (\textbf{Case I}), $u\equiv u^*\in C^{2+\alpha,1+\alpha/2}(\bar\Omega_T^1\cup\bar\Omega_T^3)$ (\textbf{Case II}) so that
\[
u_t=\dv(A(Du))\;\;\mbox{and}\;\; |Du|<s_-(\bar r)\;\;\mbox{in}\;\;\Omega_T^1,\quad (\textbf{Case I})
\]
\[
u_t=\dv(A(Du)),\;\; |Du|<s_-(\bar r_1)\;\;\mbox{in}\;\;\Omega_T^1,
\]
\[
\mbox{and}\;\;|Du|>s_+(\bar r_3)\;\;\mbox{in}\;\;\Omega_T^3.\quad (\textbf{Case II})
\]
Note $(v_j)_t \wcon v_t$\, in\, $L^2(\Omega_T;\R^n)$, where $v_j$ is the corresponding vector function to $u_j$ and $v\in W^{1,2}((0,T);L^2(\Omega;\R^n))$. From (\ref{div-v3}), we can even deduce that  $(v_j)_t \to v_t$\, pointwise a.e. in $\Omega_T$. On the other hand, from the definition of $\mathcal U_{1/j}$,
\[
\int_{\Omega_T^2} \mathrm{dist}((Du_j,(v_j)_t),\mathcal C)\, dxdt\le\frac{1}{j}|\Omega_T^2|\to 0\;\;\mbox{as}\;\; j\to \infty;
\]
thus $(Du,v_t)\in\mathcal C$\, a.e. in $\Omega_T^2$, yielding $|S|+|L|=|\Omega_T^2|$.

For the remaining assertions, we separate the proof for each case.

\textbf{Case I.} If $|L|=0$, then $|Du|\le s_-(r)$ a.e. in $\Omega_T$; so $u$ is a Lipschitz solution to (\ref{ib-para}) with monotone flux $\tilde A(p)$. Thus, by Proposition \ref{unique}, we have $u=u^*$ in $\Omega_T$. This contradicts the fact that $\|Du^*(\cdot,0)\|_{L^\infty(\Omega)}=\|Du_0\|_{L^\infty(\Omega)}=M_0>s_-(r).$ Thus $|L|>0$.

\textbf{Case II.} Suppose $|L|=0.$ Then
\[
\mbox{$|Du|\in[0,s_-(r)]\cup[s_+(\bar r_3),\infty)$\;  a.e. in $\Omega_T$.}
\]
Now, modify the profile $\sigma(s)$ so as to obtain a function $\bar\sigma\in C^{1+\alpha}([0,\infty))$ satisfying
\[
\left\{ \begin{array}{ll}
          \bar\sigma(s)=\sigma(s), & s\in[0,s_-(r)]\cup[s_+(\bar r_3),\infty), \\
          \bar\theta\le\bar\sigma'(s)\le\bar\Theta, & 0\le s<\infty,
        \end{array}
 \right.
\]
for some constants $\bar\Theta\ge\bar\theta>0$, and let $\bar f(s)=\bar\sigma(\sqrt s)/\sqrt{s}\;\;(s>0)$, $\bar f(0)=f(0)$, and $\bar A(p)=\bar f(|p|^2)p\;\;(p\in\R^n).$ Then  the functions $u\in\mathcal G$ are all Lipschitz solutions to the problem of type (\ref{ib-P}) with monotone flux $\bar A(p)$, a contradiction to the uniqueness by  Proposition \ref{unique}. Thus $|L|>0$.

Next, suppose    $|S|=0$. Then
\[
\mbox{$|Du|\in[0,s_-(\bar r_1)]\cup[s_+(\tilde r),\infty)$\;  a.e. in $\Omega_T$.}
\]
So we get a contradiction similarly as above by obtaining a function $\hat\sigma\in C^{1+\alpha}([0,\infty))$ satisfying
\[
\left\{ \begin{array}{ll}
          \hat\sigma(s)=\sigma(s), & s\in[0,s_-(\bar r_1)]\cup[s_+(\tilde r),\infty), \\
          \hat\theta\le\hat\sigma'(s)\le\hat\Theta, & 0\le s<\infty,
        \end{array}
 \right.
\]
for some constants $\hat\Theta\ge\hat\theta>0$. Thus $|S|>0$.
\end{proof}

\subsection{Proof of Corollary \ref{coro:H-3}}
Recall that this corollary is under \textbf{Case II:} Hypothesis (H).
Let $u_0\in C^{2+\alpha}(\bar\Omega)$ with $Du_0\cdot\n|_{\partial\Omega}=0$. The existence of infinitely many Lipschitz solutions to problem (\ref{ib-P}) when $|Du_0(x_0)|\in (s_1^*,s_2^*)$ for some $x_0\in\Omega$ is simply the result of Theorem \ref{thm:H-1}. So we cover the other possibilities here.

Assume $\|Du_0\|_{L^\infty(\Omega)}\le s_1^*$. Fix any two numbers $\sigma(s_2)<r_1<r_2<\sigma(s_1)$, and let $\tilde\sigma,\,\tilde f\in C^{1+\alpha}([0,\infty))$ be some functions from Lemma \ref{lem:modi-H}. Using the flux $\tilde A(p)=\tilde f(|p|^2)p$, Theorem \ref{existence-gr-max} gives a unique solution $u^*\in C^{2+\alpha,1+\alpha/2}(\bar\Omega_T)$ to problem (\ref{ib-parabolic}). If $|Du^*|$ stays on or below the threshold $s_1^*$ in $\Omega_T$, then $u^*$ itself is a Lipschitz solution to (\ref{ib-P}). Otherwise, set $\bar s=\frac{s_1^*+s_-(r_1)}{2}$ and choose a point $(\bar x,\bar t)\in\Omega_T$ such that
\[
|Du^*|\le\bar s\;\;\mbox{in}\;\;\Omega\times(0,\bar t),\quad |Du^*(\bar x,\bar t)|\in (s_1^*,s_2^*).
\]
Regarding $u_1(\cdot):=u^*(\cdot,\bar t)\in C^{2+\alpha}(\bar\Omega)$, satisfying $Du_1\cdot\n|_{\partial\Omega}=0$, as a new initial datum, it follows from Theorem \ref{thm:H-1} that problem (\ref{ib-P}), with the initial datum $u_1$ at  time $t=\bar t$, admits infinitely many Lipschitz solutions $\bar u$ in $\Omega\times(\bar t,T)$. Then the patched functions $u=\chi_{\Omega\times (0,\bar t)} u^*+\chi_{\Omega\times [\bar t,T)} \bar u$ in $\Omega_T$ become Lipschitz solutions to the original problem (\ref{ib-P}).

Lastly, assume $\min_{\bar\Omega}|Du_0|\ge s_2^*$. Let $r_1,\,r_2$ and $\tilde A(p)$ be as above, and let $u^*$ be the solution to (\ref{ib-parabolic}) corresponding to this flux $\tilde A(p)$. If $|Du^*|$ stays on or above the threshold $s_2^*$ in $\Omega_T$, then $u^*$ itself is a Lipschitz solution to (\ref{ib-P}). Otherwise, set $\bar s=\frac{s_2^*+s_+(r_2)}{2}$ and choose a point $(\bar x,\bar t)\in\Omega_T$ such that
\[
|Du^*|\ge\bar s\;\;\mbox{in}\;\;\Omega\times(0,\bar t),\quad |Du^*(\bar x,\bar t)|\in (s_1^*,s_2^*).
\]
Then we can do the obvious as above to obtain infinitely many Lipschitz solutions $u$ to (\ref{ib-P}), and the proof is complete.

\section{Radial and non-radial solutions}\label{sec:r-nr}

In this final section, we prove Theorem \ref{coex} on the coexistence of radial and non-radial Lipschitz solutions to problem (\ref{ib-P}) when $\Omega$ is a ball and $u_0$ is a radial function.  For convenience, we focus only  on \textbf{Case I:} Perona-Malik type of equations; one could equally justify the same for \textbf{Case II:} H\"ollig type of equations. So we assume the flux   $A(p)$ fulfills Hypothesis (PM).

Let $\Omega=B_R(0)$ be an open ball in $\R^n$ and  the initial datum $u_0\in C^{2+\alpha}(\bar\Omega)$ satisfy the compatibility condition
\[
A(Du_0)\cdot\mathbf{n}=0\;\;\mbox{on}\;\;\partial \Omega.
\]

We say that a function $u$ defined in $\Omega_T$ ($\Omega$, resp.) is \emph{radial} if $u(x,t)=u(y,t)\;\forall x,y\in\Omega,\,|x|=|y|$, $\forall t\in(0,T)$ ($u(x)=u(y)\;\forall x,y\in\Omega,\,|x|=|y|$, resp.).

We have the following.

\begin{thm}
Assume $u_0$ is non-constant and  radial. Then there are infinitely many radial and non-radial Lipschitz solutions to \emph{(\ref{ib-P})}.
\end{thm}

\begin{proof}
The existence of infinitely many \emph{radial} Lipschitz solutions to (\ref{ib-P}) follows from the authors' recent paper \cite{KY}. We remark that these radial solutions are not obtained through the existence result of this paper, Theorem \ref{thm:PM-1}.

We now prove the existence of infinitely many \emph{non-radial} Lipschitz solutions to  (\ref{ib-P}).
In the current situation, it is easy to see that the function $u^*\in\mathcal U$ constructed in Section 6 is radial in $\Omega_T$. Our strategy is to imitate the procedure of the density proof in Section 7 to the function $(u^*,v^*).$ We  choose a space-time box in $\Omega_T$ having positive distance from the central axis $\{0\}\times[0,T]$ of $\Omega_T$, where $v_t^*$ is sufficiently away from $A(Du^*)$ in $L^1$-sense. Then as in the density proof, we perform the surgery on $(u^*,v^*)$ only in the box to obtain a function $(u^*_{nr},v^*_{nr})$ with membership $u^*_{nr}\in\mathcal U$ maintained. Such a surgery breaks down the radial symmetry of $u^*$; hence, $u^*_{nr}$ is non-radial. Note also that this $u^*_{nr}$ cannot be a Lipschitz solution to (\ref{ib-P}).

Suppose there are only finitely many (possibly zero) non-radial Lipschitz solutions to (\ref{ib-P}). This forces for the non-radial function $u^*_{nr}\in\mathcal X$ to be the $L^\infty$-limit of some sequence of radial functions in $\mathcal G$ in the context of the proof of Theorem \ref{thm:main}; a contradiction. Therefore, there are infinitely many non-radial Lipschitz solutions to (\ref{ib-P}).
\end{proof}

\end{document}